\documentclass{amsart}
\usepackage{amssymb,amsmath,amsrefs}
\usepackage{float}
\usepackage[colorlinks=true]{hyperref}
\hypersetup{urlcolor=blue, citecolor=blue, linkcolor=blue}
\usepackage[dvips]{graphicx}
\usepackage{xcolor}
\usepackage{tikz}
\usepackage{enumitem}
\usepackage{array}
\usepackage{wrapfig}
\usepackage{tabularx}
\usepackage[normalem]{ulem}

\theoremstyle{plain}
\newtheorem{mainthm}{Theorem}

\newtheorem*{conj*}{Conjecture}
\newtheorem*{cor*}{Corollary}
\newtheorem{theorem}{Theorem}[section]

\newtheorem{proposition}[theorem]{Proposition}
\newtheorem{corollary}[theorem]{Corollary}
\newtheorem{lemma}[theorem]{Lemma}

\newtheorem{problem}{Problem}

\theoremstyle{definition}
\newtheorem*{def*}{Definition}
\newtheorem{remark}[theorem]{Remark}
\newtheorem{rmk}[theorem]{Remark}
\newtheorem{example}[theorem]{Example}

\newtheorem{definition}[theorem]{Definition}

\newcommand{\SE}{{\mathcal E}}

\newcommand{\SM}{{\mathcal M}}

\newcommand{\SU}{{\mathcal U}}

\renewcommand{\epsilon}{\varepsilon}

\newcommand{\Z}{\mathbb{Z}}
\newcommand{\N}{\mathbb{N}}
\newcommand{\R}{\mathbb{R}}
\newcommand{\eps}{\varepsilon}

\newcommand{\interior}{\operatorname{int}}

\newcommand{\su}{\operatorname{Supp}}

\newcommand{\diam}{\operatorname{diam}}		
\makeatletter
\newcommand{\tpitchfork}{
  \vbox{
    \baselineskip\z@skip
    \lineskip-.52ex
    \lineskiplimit\maxdimen
    \m@th
    \ialign{##\crcr\hidewidth\smash{$-$}\hidewidth\crcr$\pitchfork$\crcr}
  }
}
\makeatother

\newcommand{\bbS}{\mathbb{S}}
\newcommand{\bbT}{\mathbb{T}}
\newcommand{\Homeo}{\operatorname{Homeo}}

\newcommand{\htop}{h_{\mathrm{top}}}

\newcommand{\Merg}{\mathcal{M}^{e}}

\DeclareMathOperator{\Sh}{Sh}
\DeclareMathOperator{\id}{id}
\newcommand{\dH}{d_{C^0}^H}
\newcommand{\dC}{d_{C^0}}
\newcommand{\Supp}{\operatorname{Supp}}
\newcommand{\Diff}{\operatorname{Diff}}
\newcommand{\Fix}{\operatorname{Fix}}
\newcommand{\supp}{\operatorname{supp}}
\newcommand{\bbN}{\mathbb{N}}

\title[Local Shadowing Beyond Global Shadowing]{Local Shadowing Beyond Global Shadowing: Entropy and Dense Manifold Realizations}

\author{Piotr Oprocha}
\address[P. Oprocha]{
National Supercomputing Centre IT4Innovations, University of Ostrava,
	IRAFM,
	30. dubna 22, 70103 Ostrava,
	Czech Republic}
\email{piotr.oprocha@osu.cz}

\author{Elias Rego}
\address[E. Rego]{
	AGH University of Krakow, Faculty of Applied Mathematics,
	al. Mickiewicza 30,
	30-059 Krak\'ow,
}
\email{rego@agh.edu.pl}


\begin{document}

\begin{abstract}
Shadowable points were developed to cover cases in which a local shadowing mechanism survives
without a global shadowing property. We show that on every compact manifold of dimension at least two, there is a
$C^0$-dense set $\mathcal{R}$ of homeomorphisms so that each $f\in \mathcal{R}$ has  a transitive chain component $D$ consisting of shadowable points, although $f$, $f|_D$, and every chain recurrent class meeting a neighborhood of $D$ fail shadowing. Thus ambient pointwise tracing is neither inherited from global shadowing nor explained by a shadowing subsystem. Furthermore, on general compact spaces, arbitrary Cantor dynamics can occur as the entire set of shadowable points. 
We also study shadowable points through the local dynamics of chain classes and derive, under additional hypotheses, semi-horseshoes,
entropy-bearing ergodic approximations of measures, and entropy flexibility.
\end{abstract}

\maketitle

\section{Introduction}

The shadowing property is a standard condition for tracing approximate
orbits by exact orbits.  It originated in hyperbolic dynamics and is
closely tied to stability of dynamics under perturbations, the reliability of numerical trajectories, and the qualitative description of systems through spectral decomposition and
symbolic dynamics, see
\cites{AH,Py,Pil2,Pal} for general accounts and
\cites{DGS,KMS,TCY,UG} for representative applications.  The classical definition of shadowing is
global: one tracing scale must work for pseudo-orbits starting everywhere
in the phase space.  Intuitively, failure of shadowing may be caused by
behavior far from a specified invariant region and does not determine
whether pointwise tracing holds in that region.
Shadowable points, introduced by Morales \cite{Mo}, isolate that mechanism.
A point $x$ is \textit{shadowable} when every sufficiently accurate two-sided
pseudo-orbit through $x$ is traced by a true orbit; we denote the set of
such points by $\Sh(f)$.  Thus $f$ has the shadowing property exactly when
$\Sh(f)$ is the whole phase space.  Morales also constructed homeomorphisms
without shadowing for which $\Sh(f)$ is dense \cite{Mo}. The pointwise
viewpoint was subsequently developed for flows \cite{AV}, quantitative
shadowability and local entropy \cites{Ka1,Ka2,Ka3}, pointwise stability
and persistence \cite{DLM}, periodic shadowability \cite{KLT}, and
measure-theoretic shadowing \cites{LMo,Shin} and has also been used to
obtain pointwise criteria for positive entropy \cite{AR}.  These results
develop the theory of the pointwise property, while leaving the following
geometric question: \textit{can a dynamically coherent invariant set consist of
shadowable points when neither the global map nor the restricted system
has the shadowing property?}

This question is especially pertinent on manifolds. The shadowing property is
$C^0$-generic in standard spaces of manifold homeomorphisms, and related
periodic tracing properties are generic in several corresponding spaces
\cites{PP,PK,PKMM,GL}. Consequently, examples inherited from maps with the 
shadowing property do not test whether the pointwise property has an
independent role.  Theorem~\ref{thm:shadowable-dense} answers the question
above for homeomorphisms without the shadowing property.
Since shadowing is $C^0$-generic, the homeomorphisms we produce necessarily form a dense but topologically meager family; the separation between local and global shadowing is thus exceptional, and yet it is unavoidably present throughout $\Homeo(M)$.

The paper is organized around four related issues: (i) whether local
pointwise shadowing can persist when both the global and the intrinsic shadowing property
fail; (ii) whether this separation can be realized densely on manifolds; (iii)
which tracing properties propagate through a chain recurrent class; (iv) which
symbolic, measure-theoretic, and entropy consequences follow from local
tracing. The main theorem addresses the first two.

\begin{mainthm}\label{mainA}\label{thm:shadowable-dense}
Let $M$ be a compact manifold of dimension at least two. There is a
$C^0$-dense set of homeomorphisms $\mathcal{R}\subset\Homeo(M)$ such that if $f\in \mathcal{R}$, then:
\begin{enumerate}
\item The global system $(M,f)$ does not have the shadowing property;
\item $f$ has a transitive chain-recurrent component $D$ with
      $D\subseteq \Sh(f)$, while $(D,f|_D)$ does not have the shadowing
      property;
\item There is a neighborhood $V$ of $D$ meeting only chain-recurrent classes whose
      restricted dynamics fails to have shadowing.
\end{enumerate}
Moreover, the inserted local block has finite topological entropy and, in dimension at least three, it can obtained as a perturbation of Axiom~A (hence hyperbolic) building block, with the distinguished component $D$ of positive finite topological entropy.
\end{mainthm}

%The precise statement is Theorem~\ref{thm:shadowable-dense}.  
The above theorem separates
three properties that coincide in many standard examples but are
logically different: pointwise shadowability in the ambient space,
global shadowing for pseudo-orbits in the system, and intrinsic shadowing of pseudo-orbits in
the restricted system.  In the theorem, every point of $D$ traces all
ambient pseudo-orbits through it, including pseudo-orbits that leave $D$.
Nevertheless, pseudo-orbits contained in $D$ cannot in general be traced
by orbits of $f|_D$.  Neither can the phenomenon be attributed to shadowing on
surrounding chain-recurrent classes.  Thus local shadowing
occurs independently of these forms of shadowing, and the asserted
homeomorphisms form a dense subset of $\Homeo(M)$ for every compact
manifold $M$ of dimension at least two.

The proof uses a controlled realization criterion
(Theorem~\ref{thm:abstract}). Starting
from a chain component accumulated by invariant sets with shadowing, a
Denjoy--Rees modification \cite{BCLR}  (a technique going back to Rees' minimal positive-entropy homeomorphisms of the torus \cite{Rees}) inserts measure-theoretic dynamics
while quantitative control of the approximating conjugacies preserves
tracing through the distinguished component.  Entropy-spectrum gaps,
using the global approximation theory of \cite{LiOp}, then rule out
shadowing both on that component and on the ambient system.  Plykin-type
blocks provide the finite-entropy seed in dimensions at least three; a
separate annular construction handles dimension two.  Controlled
surgery then inserts these local blocks into an arbitrary homeomorphism, producing the dense family described in the statement above.

Before the manifold construction we establish two general facts that
clarify the structure of points carrying local shadowing. First, local shadowing propagates through
chain recurrence.  If
\[
        x\in Sh^+(f)\cap CR(f),
\]
then every point of its chain class $H(x)$ is positively shadowable, with
uniform constants, and $H(x)$ has the two-sided shadowing property as a
subset of the ambient space (Theorems~\ref{ShadClass}
and~\ref{shadonH}).  If the class is locally maximal, the tracing orbit can
be forced to remain in the class, and intrinsic shadowing follows
(Corollary~\ref{cor:loc_max}).  Consequently, local maximality of chain-recurrent class provides a concrete
mechanism by which pointwise tracing upgrades to shadowing of a
restricted system and our examples show what can happen without it.

Here $\Sh^+(f)$ denotes the weaker forward-time version, in which only
positive pseudo-orbits through the point are considered. The classical
two-sided set $\Sh(f)$ is the principal object of the manifold results.
We use $\Sh^+(f)$ only when a proof uses no information about
the past, and state the corresponding theorem at that weaker level, therefore applicable in more generality.
This
keeps every result applicable to $\Sh(f)$ while also revealing which
arguments extend to noninvertible maps, where forward shadowability is the
natural formulation.

The manifold theorem has a complementary realization counterpart on
general compact metric spaces. Given any homeomorphism $T$ of the Cantor
set, Theorem~\ref{extthm} constructs a compact system $(\mathcal M,f)$
containing a copy $X$ of that Cantor system such that
\[
             Sh(f)=Sh^+(f)=X.
\]
Every compact invariant subsystem of $f$ with shadowing is contained in
$X$, as is every Toeplitz subsystem. Thus the theorem prescribes the
entire set of classical shadowable points, whereas Theorem~\ref{mainA} realizes a
dynamically coherent family of such points densely on manifolds. The two
results form the complementary abstract and geometric sides of Part~I.
The realization on the Cantor set should also be viewed against the
backdrop of zero-dimensional dynamics, where the generic homeomorphism of
the Cantor space has the shadowing property \cite{BeDa}. Theorem~\ref{extthm}
shows that, nevertheless, every Cantor system whatsoever---including systems
as far from shadowing as one wishes---arises as the full set of shadowable
points of a compact system.

Part~II of the paper studies dynamical consequences of local
tracing.  A~coding lemma shows that sufficiently separated periodic
pseudo-orbits through one positively shadowable point generate a
semi-horseshoe.  Chain recurrence supplies the required returns, while
nonminimality or non-equicontinuity supplies their separation. This
mechanism leads to approximation and entropy results that localize the
classical approximation theory developed for systems with shadowing and specification
\cites{Sig,Denker,MoOp,LiOp}.

\begin{mainthm}\label{mainB}
Let $f\colon M\to M$ be a homeomorphism and suppose that $f$ has a
non-periodic, positively shadowable, chain-recurrent point $x$.  If
$f|_{H(x)}$ is not equicontinuous, then for every invariant probability
measure $\mu$ with $\operatorname{supp}(\mu)\subset H(x)$ and every
$\eps>0$, there is an ergodic measure $\nu$ with positive entropy such that
$d^*(\nu,\mu)<\eps$.
\end{mainthm}

Here $d^*$ is any metric compatible with the weak${}^*$ topology and $H(x)$ is the chain recurrent class containing $x$.  The
approximating measures are supported on semi-horseshoes. In particular,
measures on symbolic subsystems occur arbitrarily close to the prescribed measure in
the weak${}^*$ topology.  When
the initial measure has positive entropy, the construction can be tuned to
a prescribed entropy level.

\begin{mainthm}\label{mainC}
Let $f\colon M\to M$ be a homeomorphism and
$x\in Sh^+(f)\cap CR(f)$. Suppose that
$\mu$ is invariant and $\operatorname{supp}(\mu)\subset H(x)$.  For every
$c\in[0,h_\mu(f))$ and $\eps>0$, there is an ergodic measure $\mu_c$ such
that
\[
        d^*(\mu_c,\mu)<\eps
        \quad\text{and}\quad
        h_{\mu_c}(f)\geq c.
\]
If, in addition, $f$ is entropy expansive, then $\mu_c$ may be chosen so
that
\[
        h_{\mu_c}(f)=h(f|_{\operatorname{supp}(\mu_c)})=c.
\]
\end{mainthm}

Theorem~\ref{mainC} recovers entropy flexibility under global
shadowing as a corollary, but its hypothesis is local, centered at a single chain
class.  We further combine the coding mechanism with weaker versions of expansivity. In the homeomorphism setting, this answers the
question raised in \cite{AR} in a strengthened form, with countable expansivity at
one point in place of uniform expansivity, and it yields semi-horseshoes
from expansive positively shadowable measures. None of these results
assumes global shadowing.

This measure-approximation mechanism should be compared with two
classical circles of ideas. In smooth ergodic theory, Katok's horseshoe
theorem \cite{Katok} approximates hyperbolic measures of $C^{1+\alpha}$
diffeomorphisms by horseshoes; Theorems~\ref{mainB} and~\ref{mainC} may be
viewed as a topological counterpart in which smoothness and hyperbolicity
are replaced by tracing localized at a single chain class. On the other
hand, the role of entropy expansivity in Theorem~\ref{mainC} parallels its
role in the entropy theory of symbolic extensions of Boyle and Downarowicz
\cite{BD}, where it guarantees that no entropy is hidden at arbitrarily
small scales. It is precisely this property that converts the
semi-horseshoe coding into exact entropy values.

The two parts meet on the examples constructed in Part~I.  In dimensions at
least three, the distinguished manifold component can be chosen with
positive entropy. It then contains semi-horseshoes, and every invariant
measure supported on it is approximated in the weak${}^*$ topology by
positive-entropy ergodic measures
(Corollary~\ref{cor:dense-consequences}).  In dimension two the
distinguished component in our model is an irrational rotation. The
contrast shows why local shadowing alone does not force complexity and why
the additional recurrence and non-equicontinuity hypotheses in Part~II
are natural.

The paper is organized to separate realization results from their
dynamical consequences.
Section~\ref{sec:prel} fixes notation, and Section~\ref{SectionSP} develops
pointwise shadowing and its propagation through chain classes.  Part~I
establishes the controlled manifold realization criterion, constructs
local models in every dimension at least two, and proves
Theorem~\ref{mainA}. Its final subsection proves the complementary compact-space
realization, Theorem~\ref{extthm}.  Part~II studies
shadowable measures, develops the semi-horseshoe mechanism, proves
Theorems~\ref{mainB} and~\ref{mainC}, and derives the expansivity and noninvertible
consequences.

\section{Preliminaries}\label{sec:prel}

This section fixes the notation and recalls the standard concepts used throughout the paper. For a detailed exposition of the background material, we refer the reader to~\cite{Kat}.

\subsection{Basic Topological Dynamics}\label{basics}

Throughout the paper, $(M,d)$ denotes a compact metric space and
$f\colon M\to M$ a homeomorphism, unless stated otherwise. The full orbit
of $x\in M$ is
\[
O(x)=\{f^n(x):n\in\mathbb Z\},
\]
where the usual convention is used for positive, negative, and zero
iterates. Its positive and negative orbits are
\[
O^+(x)=\{f^n(x):n\ge0\},
\qquad
O^-(x)=\{f^n(x):n\le0\}.
\]
The omega- and alpha-limit sets of $x$ are
\[
\begin{aligned}
\omega(x)&=\{y\in M:\text{there is }n_k\to\infty
                    \text{ with }f^{n_k}(x)\to y\},\\
\alpha(x)&=\{y\in M:\text{there is }n_k\to-\infty
                    \text{ with }f^{n_k}(x)\to y\}.
\end{aligned}
\]
A point $x$ is \emph{recurrent} if $x\in\omega(x)$ and
\emph{regularly recurrent} if, for every neighbourhood $U$ of $x$, there
is an integer $n>0$ such that
\[
n\mathbb Z\subset\{i\in\mathbb Z:f^i(x)\in U\}.
\]

We say that $K\subset M$ is invariant if $f(K)=K$. The sets
$\overline{O(x)}$, $\omega(x)$, and $\alpha(x)$ are compact and invariant.
A point
$x$ is \textit{non-wandering} if, for every open neighborhood $U$ of $x$,
there is $n>0$ such that $f^n(U)\cap U\neq\emptyset$.
The set of all non-wandering points, $\Omega(f)$ is closed and invariant. 
A closed invariant set is \emph{minimal} if it contains no proper,
nonempty, closed invariant subset (equivalently, every orbit in it is dense in it).  We call $f$ \textit{minimal} when $M$ is minimal.

The map $f$ is \emph{sensitive} if there is $C>0$ such that for every
nonempty open set $U\subset M$, some $n\in\mathbb N$ satisfies
$\diam(f^n(U))>C$. It is \emph{equicontinuous} if, for every $\eps>0$,
there is $\delta>0$ such that
\[
d(x,y)\le\delta
\quad\Longrightarrow\quad
d(f^n(x),f^n(y))\le\eps\quad\text{for every }n\in\mathbb Z.
\]
A compact invariant set $K$ is equicontinuous if $f|_K$ is
equicontinuous.

\subsection{Pseudo-Orbits, Chain-Recurrence and Shadowing}

Let $A\subset\mathbb Z$ be an integer interval. A sequence
$(x_i)_{i\in A}$ is a \emph{$\delta$-pseudo-orbit} if
\[
d(f(x_i),x_{i+1})\le\delta
\quad\text{whenever }i,i+1\in A.
\]
where $A$ is a sequence of consecutive integers.
We say that the pseudo-orbit is \textit{two-sided} when $A=\mathbb Z$, \textit{positive} or \textit{one-sided} when
$A=\mathbb N_0$, and \textit{finite} when $A$ is finite. in the case of finite $A$ we interchangeably use the name \textit{$\delta$-chain} for a $\delta$-pseudo-orbit.
If a finite pseudo-orbit $(x_i)_{i=0}^n$ satisfies $x_0=x$ and $x_n=y$,
we say that it connects $x$ to $y$ and write $x\to y$ when such
pseudo-orbits exist for every $\delta>0$. A point is
\emph{chain recurrent} if $x\to x$ and the set of all such points is denoted by $CR(f)$.
Points $x,y\in CR(f)$ are \emph{chain related}, written $x\sim y$, if
$x\to y$ and $y\to x$. This is an equivalence relation on $CR(f)$, and
the equivalence class of $x$, called \textit{chain recurrent class}, is denoted by $H(x)$.

The following standard facts will be used:
    \begin{enumerate}
        \item $\Omega(f)\subset CR(f)$. So, $CR(f)$ is always non-empty.
        \item Both $CR(f)$ and $H(x)$ are closed and invariant sets, for every $x\in CR(f)$.        
    \end{enumerate}

By convention, we index every finite pseudo-orbit of length $n$ by
$\{0,\ldots,n-1\}$ and often denote it by a capital letter. If
$X=(x_0,\ldots,x_{m-1})$ ends where
$Y=(y_0,\ldots,y_{n-1})$ begins, their concatenation is
\[
XY=(x_0,\ldots,x_{m-1},y_1,\ldots,y_{n-1}).
\]
A pseudo-orbit $(x_i)_{i\in A}$ is \emph{$\eps$-shadowed} if there is
$z\in M$ such that
\[
d(f^i(z),x_i)\le\eps\quad\text{for every }i\in A.
\]

\begin{definition}\label{shaddef}
Let $f:M\to M$ be a homeomorphism.

\begin{enumerate}
    \item $f$ is said to have the \textit{shadowing property} if for every $\eps>0$, there is $\delta>0$ such that every two-sided $\delta$-pseudo-orbit is $\eps$-shadowed.
    \item A subset $K\subset M$ has the \textit{shadowing property} if 
     for every $\eps>0$, there is $\delta>0$ such that every two-sided $\delta$-pseudo-orbit contained in $K$ is $\eps$-shadowed.
    \item A compact and invariant subset  $K\subset M$ is said to have the \textit{intrinsic shadowing property} if $f|_{K}$ has the shadowing property.
\end{enumerate}
    
\end{definition}

\begin{rmk}
We use the following conventions throughout:
\begin{itemize}
    \item The unqualified statement that $f$ has the shadowing property refers to item~(1), which we also call \emph{global shadowing} or simply \textit{shadowing}. When a set $K$ has the shadowing property in the sense of item~(2), tracing points may lie anywhere in $M$ and therefore we refer to it often as \emph{ambient shadowing}. When we want to emphasize the fact that the tracing point is inside invariant subset $K$ (i.e. the restricted system $f|_K$ that has the shadowing property), we speak about intrinsic shadowing.
    \item For a non-invertible map, the global shadowing property is defined using positive pseudo-orbits. For homeomorphisms this positive global property is equivalent to the standard two-sided property. In fact by compactness, the tracing conditions for all three types of pseudo-orbits (finite, positive, two-sided) are equivalent. The analogous equivalence fails for pointwise shadowability, which is why $\Sh(f)$ and $\Sh^+(f)$ must be distinguished.
\end{itemize}

\end{rmk}

\subsection{Topological Entropy}

Fix $N\subset M$, $\varepsilon>0$, and $n\in\mathbb N$. A set
$K\subset N$ is \emph{$(n,\varepsilon)$-separated} if, for every distinct
$x,y\in K$, there is $0\le n_0<n$ such that
\[
d(f^{n_0}(x),f^{n_0}(y))>\varepsilon.
\]
Let $S(n,\varepsilon,N)$ be the maximal cardinality of such a set and put
\[
h(f,N,\varepsilon)
 =\limsup_{n\to\infty}\frac1n\log S(n,\varepsilon,N),
\qquad
h(f,N)=\lim_{\varepsilon\to0}h(f,N,\varepsilon).
\]
The topological entropy of $f$ is $h(f):=h(f,M)$. Recall that
$h(f^n)=n h(f)$ for $n\ge1$. We call $f$ \emph{entropy expansive} if
there is $e>0$ such that
$h(f,B^\infty_e(x))=0$ for every $x\in M$.

\subsection{Symbolic Dynamics and Semi-Horseshoes}

Let $f:M\to M$ and $g:N\to N$ be two homeomorphisms. We say that $f$ \textit{factors} over $g$ if there is a continuous surjective map $\pi:M \to N$ (a \textit{factor map}) such that $\pi\circ f=g\circ \pi$.  A factor map $\pi$ is \emph{almost 1-1}
if $\{x\in M\colon \pi^{-1}(\pi(x))=\{x\}\}$ is residual in $M$.

Let $A_n=\{0,1,\ldots,n-1\}$ and denote
$\Sigma_n=A_n^{\mathbb{Z}}$, the space of two-sided sequences with entries
in $A_n$. For
$s\neq s'$, set
\[
d(s,s')=2^{-\min\{|i|:s_i\neq s'_i\}},
\]
and set $d(s,s)=0$. With this metric, $\Sigma_n$ is a compact metric
space. The \textit{shift map} on $\Sigma_n$ is the map
\[
\sigma:\Sigma_n\to\Sigma_n,
\]
defined by $\sigma((s_i))=(s_{i+1})$.
It is well known that $h(\sigma)=\log(n)$.
 If $f$ factors over $g$, we also say that $f$ is an extension of $g$ or $g$ is a factor of $f$.

\begin{definition}
We say that $f$ has a \emph{semi-horseshoe} if there are $n>0$ and a
compact invariant set $K\subset M$ such that $f^n|_K$ factors onto a full
shift.
\end{definition}
A consequence of $f$ having a semi-horseshoe is that $f$ has positive entropy (see \cite{Kat}). 

A map $f:M\to M$  is called an \emph{odometer}
if it is equicontinuous and there exists a regularly recurrent point $x\in M$
such that $\overline{O^+(x)}=M$. Note that with this definition, a periodic orbit is also an odometer.
If $f\colon M \to M$ is an almost one-to-one extension of an odometer, then we call it a \textit{Toeplitz flow} or \textit{Toeplitz system}. It is equivalent to say that $M=\overline{O^+(x)}$ for some regularly recurrent point $x$. For further details on odometers and Toeplitz systems the reader is referred to the paper of Downarowicz \cite{Toeplitz}.

\subsection{Ergodic Theory and Metric Entropy}

All measures are defined with respect to Borel $\sigma$-algebra on the underlying compact metric space $M$. A probability measure $\mu$ is
\emph{$f$-invariant} if $\mu(A)=\mu(f^{-1}(A))$ for every measurable
$A$, and an invariant measure is \emph{ergodic} if
$\mu(A)\mu(A^c)=0$ for every invariant measurable set $A$. We denote the
spaces of probability, invariant, and ergodic invariant measures by
$\SM(M)$, $\SM_f(M)$, and $\SM_f^e(M)$, respectively.

The Krylov--Bogolyubov theorem implies that $\SM_f(M)$ is non-empty
whenever $f$ is continuous. The ergodic decomposition theorem then implies
that $\SM_f^e(M)$ is non-empty. We say that $f$ is uniquely ergodic if
$\SM_f(M)$ is a singleton.

For each $x\in M$, define its $n$-th empirical measure by
\[
\SE_n(x)=\frac{1}{n}\sum_{j=0}^{n-1}\delta_{f^j(x)},
\]
where $\delta_y$ denotes the Dirac measure at $y$.

We next equip $\SM(M)$ with a metric inducing the weak${}^*$ topology (see
\cite{Du} for further details). For easy calculations it is often convenient to use the following metric on $\SM(M)$.
Let $BL(M)$ denote the space of bounded
real-valued Lipschitz functions on $M$, and write
$$||\phi||_L=\sup\left\{\frac{|\phi(x)-\phi(y)|}{d(x,y)}:x\neq y\right\},
\qquad ||\phi||_{BL}=||\phi||_{\infty}+||\phi||_L.$$
The space $BL(M)$ is dense in $C(M,\R)$. Fix a sequence $(\phi_n)$ 
dense in the unit ball of $BL(M)$, with respect to the introduced norm. For
any pair of measures $\mu,\nu\in \SM(M)$, define
$$d^*(\mu,\nu)=\sum_{n=1}^{\infty}\frac{1}{2^n}\left|\int \phi_n\,d\mu-\int \phi_n\,d\nu \right|.$$
In this way, $d^*$ is a metric on $\SM(M)$ whose induced topology coincides with its weak$^*$-topology. So, whenever we write $\mu_n\to \mu$, we mean that the convergence is being considered in the metric $d^*$.

 We will use the following lemma.

\begin{lemma}[\cite{LiOp}]\label{lem:measure-approx}
Let $(X,d)$ be a compact metric space and $\eps>0$.
\begin{enumerate}
  \item \label{enum:measure-approx-1}
  For a sequence $(x_i)_{i=0}^\infty$ of points in $X$ and two finite subsets $A,B$ of $\N_0$,
\[
  d^*\Bigl( \frac{1}{|A|}\sum_{i\in A} \delta_{x_i}, \frac{1}{|B|}\sum_{i\in B}\delta_{x_i}\Bigr)
  \leq \frac{|A|+|B|}{|A|\cdot|B|}|A\Delta B| + \frac{\bigl||A|-|B|\bigr|}{|A|\cdot|B|}|A\cap B|.
\]
  \item \label{enum:measure-approx-2}
  For two sequences $(x_i)_{i=0}^{m-1}$ and $(y_i)_{i=0}^{m-1}$ of points in $X$,
    if $d(x_i,y_i)<\eps$ for $i=0,1,\dotsc,m-1$, then
\[
 d^*\Bigl( \frac{1}{m}\sum_{i=0}^{m-1} \delta_{x_i},
        \frac{1}{m}\sum_{i=0}^{m-1} \delta_{y_i}\Bigr)<\eps.
\]
 \item\label{enum:measure-approx-3}
 If $\mu_i,\mu\in \SM(X)$ are such that $d^*(\mu_i,\mu)<\eps$ for $i=1,\ldots, K$,
 then for any choice of $\alpha_i\in [0,1]$
 with $\sum_{i=1}^K\alpha_i=1$ we have
\[
    d^*\Bigl(\sum_{i=1}^K \alpha_i \mu_i,\mu\Bigr)<\eps.
\]
\end{enumerate}
\end{lemma}

We call $x$ a \emph{generic point} of $\mu\in\SM_f(M)$ if
$\SE_n(x)\to\mu$. For a finite measurable partition $\mathcal P$, put
\[
\mathcal P_n=\mathcal P\vee f^{-1}(\mathcal P)\vee\cdots
             \vee f^{-(n-1)}(\mathcal P)
\]
and
\[
h_\mu(f,\mathcal P)
 =-\lim_{n\to\infty}\frac1n
   \sum_{P\in\mathcal P_n}\mu(P)\log\mu(P).
\]
The \textit{metric} or \textit{measure-theoretic entropy} of $f$ with respect to $\mu$ is
\[
h_\mu(f)=\sup\{h_\mu(f,\mathcal P):
                   \mathcal P\text{ is a finite measurable partition}\}.
\]
The metric and topological entropies of $f$ are related by the celebrated variational
principle: for every continuous map $f\colon M\to M$,
\[
h(f)=\sup_{\mu\in\SM_f(M)}h_\mu(f)
    =\sup_{\mu\in\SM_f^e(M)}h_\mu(f).
\]

We end this section with two technical, yet very useful, lemmas that will be used to approximate metric
entropy in the proof of Theorem~\ref{mainC}.

\begin{lemma}[\cite{LiOp}]\label{lemma1}
Let $\mu$ be an invariant measure for $f$. For every $\kappa>0$, there
are ergodic measures $\mu_1,\ldots,\mu_k$ such that
\[
d^*\left(\frac{1}{k}\sum_{i=1}^k\mu_i,\mu\right)<\kappa
\]
and 
\begin{enumerate}
    \item If $h_{\mu}(f)<\infty$, then $$\left|\frac{1}{k} \sum_{i=1}^kh_{\mu_i}(f)-h_{\mu}(f)\right|<\kappa.$$
    \item If $h_{\mu}(f)=\infty$, then $$\frac{1}{k}\sum_{i=1}^kh_{\mu_i}(f)> \frac{1}{\kappa}.$$
\end{enumerate}
\end{lemma}

\begin{lemma}[\cite{LiOp}]\label{lemma2}
    Let $\mu\in \SM^e_f(M)$. For every $\kappa>0$, there is $\eps>0$ such that for every neighborhood $\SU\subset \SM(M)$ of $\mu$, there is $N>0$ such that if $n>N$, there is a $n$-$\eps$-separated set $\Gamma_n\subset  \su(\mu)$ satisfying:
\begin{enumerate}
    \item  $\SE_n(x)\in \SU$, for every $x\in \Gamma_n$.
    \item If $h_{\mu}(f)<\infty$, then $$\left|\frac{1}{n}\log(\#\Gamma_n)-h_{\mu}(f)   \right|<\kappa.$$
    
    \item  If $h_{\mu}(f)=\infty$, then $$\frac{1}{n}\log(\#\Gamma_n)> \frac{1}{\kappa}.$$
     
\end{enumerate}

\end{lemma}

\section{Shadowable points and propagation}\label{SectionSP}

We now pass from global shadowing to its pointwise form.  Our primary object is the standard two-sided notion of a shadowable point introduced in \cite{Mo}. We will also consider a weaker positive version, because several of the arguments developed later require only forward tracing. This distinction allows us to state those results under the hypotheses used in their proofs while retaining the classical notion as the main geometric object.

\begin{definition}[Shadowable point]
A point $x\in M$ is \emph{shadowable} if, for every $\eps>0$, there is
$\delta>0$ such that every two-sided $\delta$-pseudo-orbit with $x_0=x$
is $\eps$-shadowed.
\end{definition}

Let us denote by $\Sh(f)$ the set of shadowable points of $f$. Morales \cite{Mo} proved, among other properties, that $\Sh(f)$ is $f$-invariant and that $\Sh(f)=M$ implies the shadowing property. 

\begin{definition}
    We say that a point $x\in M$ is \textit{positively shadowable} if for every $\eps>0$, there is $\delta>0$ such that every positive $\delta$-pseudo-orbit through $x$ (that is, with $x_0=x$) is $\eps$-shadowed. 
\end{definition}

We denote the set of positively shadowable points by $\Sh^+(f)$. Directly by definition,
$\Sh(f)\subset Sh^+(f).$
The following example shows that this inclusion cannot be reversed.

\begin{example}
Let
$M=\mathbb S^1$ and fix three points $x,y,z$ in counter-clockwise order.
For oriented arcs we write $(a,b)$. Define a homeomorphism $f$ fixing
$x,y,z$ such that $x$ is a source, $z$ is a sink, and $y$ attracts
points of $(x,y)$ while repelling points of $(y,z)$; see
Figure~\ref{FigSvsPS}.

    \begin{figure}[h]
     \centering
   \includegraphics[scale=0.58]{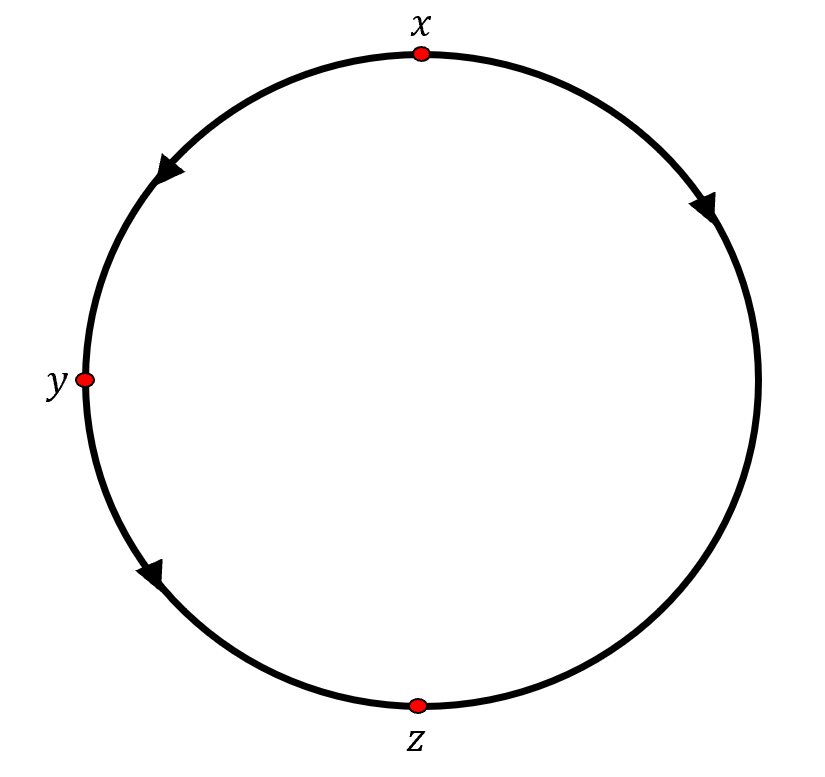}
   \caption{A homeomorphism with positively shadowable points that are not
   shadowable.}
    \label{FigSvsPS}
 \end{figure}

For arbitrarily small $\delta>0$, one can easily construct a two-sided
$\delta$-pseudo-orbit $(x_i)$ of either of the following types:
 \begin{enumerate}
        \item  $x_0\in(x,y)$; $x_i\to x$ as $i\to-\infty$; $x_{i_0}\in(y,z)$ for some $i_0>0$; $x_i\to z$ as $i\to\infty$;
or
        \item $x_0\in(y,z)$; $x_i\to z$ as $i\to\infty$; $x_{i_0}\in(x,y)$ for some $i_0<0$; $x_i\to x$ as $i\to-\infty$.
    \end{enumerate}
Such pseudo-orbits cannot be shadowed, since no orbit in $(x,y)$ visits
$(y,z)$ in the future. On the other hand, forward trapping at the sink
$z$ gives positively shadowable points in $(y,z)$. Hence
$z\in Sh^+(f)\setminus Sh(f)$.

\end{example}

The previous example shows that $\Sh^+(f)$ can be strictly larger than the classical set $\Sh(f)$. By the characterization in \cite{Mo}*{Theorem~1.1}, a homeomorphism has the shadowing property if and only if $\Sh^+(f)=Sh(f)=M$.

As usual in pointwise dynamics, the main obstruction to obtaining compact invariant families is that the tracing constants may shrink from point to point. The following example already exhibits this problem for the classical set $\Sh(f)$: density of $\Sh(f)$ does not imply $\Sh(f)=M$ and therefore does not imply the global conclusion of \cite{Mo}*{Theorem~1.1}.

\begin{example}
Let $I=[0,1]$, identify $\mathbb S^1$ with $\mathbb R/\mathbb Z$, and
endow $I\times\mathbb S^1$ with the maximum metric. For $n\ge1$, put
\[
A_n=\left\{0,\frac{1}{n},\ldots,\frac{n-1}{n}\right\},\qquad
K_n=\left\{\frac{1}{n}\right\}\times A_n,
\]
and define
\[
M=\bigl(\{0\}\times\mathbb S^1\bigr)\cup\bigcup_{n\ge1}K_n.
\]
Then $M$ is compact, the union of the sets $K_n$ is dense in $M$, and
each $K_n$ is isolated in $M$.

Define
\[
f\left(\frac{1}{n},\frac{i}{n}\right)=\left(\frac{1}{n},\frac{i+1}{n}\right)
\]
with the second coordinate taken modulo one, and let $f$ be the identity
on $\{0\}\times\mathbb S^1$. Since $A_n$ converges to $\mathbb S^1$ in
the Hausdorff topology, $f$ is a homeomorphism of $M$. Each $K_n$ is an
isolated periodic orbit and hence $K_n\subset Sh(f)$. Consequently,
\[
\bigcup_{n\ge1}K_n\subset Sh(f),
\]
so $\Sh(f)$ is dense in $M$.

We claim that no point of $\{0\}\times\mathbb S^1$ is shadowable. Given $\delta>0$, choose a
two-sided $\delta$-pseudo-orbit $(x_i)$ in the central circle such that
\begin{enumerate}
    \item $0<C\le d(x_i,x_{i+1})\le\delta$;
    \item $(x_i)$ moves in the reverse orientation of $\mathbb S^1$.
\end{enumerate}  
The first condition prevents tracing by the identity dynamics on the
central circle, while the second prevents tracing by any of the periodic
orbits $K_n$, because each of them is oriented the same as $\frac{1}{n}$ circle rotation. Thus $f$ does not have the shadowing property although $\Sh(f)$ is dense.
   
\end{example}

In the previous example, the shadowing constants of the periodic orbits shrink while approaching the fiber $\{0\}\times \mathbb{S}^1$. Thus compactness alone becomes useful only after the pointwise constants have been made uniform. The next result establishes this uniformity under the weaker hypothesis $K\subset Sh^+(f)$ and hence, in particular, for compact subsets of $\Sh(f)$.

\begin{theorem}\label{localshad}
   Suppose $K\subset Sh^+(f)$ is a compact set. Then, for every $\eps>0$ there is $\delta>0$ such that if $(x_i)$ is a positive $\delta$-pseudo-orbit through $y\in B_{\delta}(K)$, then $(x_i)$ is $\eps$-shadowed.
\end{theorem}

\begin{proof}
Fix $\eps>0$ and $x\in K$, and let $\eta>0$ be a tracing constant for
$x$ with accuracy $\frac{\eps}{2}$. Choose
\[
0<\delta_x\le\min\left\{\frac{\eta}{2},\frac{\eps}{2}\right\}
\]
so small that $d(y,z)\le\delta_x$ implies
$d(f(y),f(z))\le\frac{\eta}{2}$. Let $(x_i)_{i\ge0}$ be a positive
$\delta_x$-pseudo-orbit with $x_0=y\in B_{\delta_x}(x)$, and replace its
initial point by $x$:
\[
x'_i=
\begin{cases}
x,&i=0,\\
x_i,&i>0.
\end{cases}
\]
The only new pseudo-orbit inequality is
\[
d(f(x),x_1)
\le d(f(x),f(y))+d(f(y),x_1)
\le\frac{\eta}{2}+\delta_x\le\eta.
\]
Thus $(x'_i)$ is $\frac{\eps}{2}$-shadowed by some $z\in M$. At index zero,
\[
d(z,y)\le d(z,x)+d(x,y)<\frac{\eps}{2}+\delta_x\le\eps,
\]
and all later terms of the two pseudo-orbits agree. Hence $(x_i)$ is
$\eps$-shadowed.

For every $x\in K$, choose such a constant $\delta_x$. A finite
subcollection of the balls $B_{\delta_x}(x)$ covers $K$. By a Lebesgue
number argument, there is $\delta>0$, no larger than all corresponding
$\delta_x$, such that
\[
B_\delta(K)\subset\bigcup_{j=1}^k B_{\delta_{x_j}}(x_j).
\]
Every positive $\delta$-pseudo-orbit beginning in $B_\delta(K)$ therefore
begins in one of these neighbourhoods and therefore is $\eps$-shadowed.

\end{proof}

Our next result shows that positive pointwise shadowability propagates along a chain-recurrent class. The argument does not work for $\Sh(f)$.

\begin{theorem}\label{ShadClass}
     If $x\in Sh^+(f)\cap CR(f)$, then $H(x)\subset Sh^+(f)$.
\end{theorem}
\begin{proof}

Suppose $x\in Sh^+(f)\cap CR(f)$. Fix $\eps>0$ and let $\delta>0$
be a positive shadowing constant for $x$. Given $y\in H(x)$ and a
positive $\delta$-pseudo-orbit $(y_i)_{i\ge0}$ through $y$, choose a finite
$\delta$-pseudo-orbit $(x_0,\ldots,x_k)$ from $x_0=x$ to $x_k=y$.
Concatenate the two:
\[
z_j=
\begin{cases}
x_j,&0\le j\le k,\\
y_{j-k},&j>k.
\end{cases}
\]
This is a positive $\delta$-pseudo-orbit through $x$, so it is
$\eps$-shadowed by some $z\in M$. Consequently, $f^k(z)$
$\eps$-shadows $(y_i)_{i\ge0}$.
    \end{proof}

\begin{remark}
The proof gives a shadowing constant that is uniform over all points of
$H(x)$.
\end{remark}

As a consequence, we can characterize when the weaker set $\Sh^+(f)$ contains a compact invariant subset. 
\begin{corollary}
$\Sh^+(f)$ contains a nonempty closed invariant subset if and only if
$\Sh^+(f)\cap CR(f)\neq\emptyset$.
\end{corollary}
\begin{proof}
If a nonempty closed invariant set $K$ is contained in $\Sh^+(f)$, then
$K$ contains a minimal set and hence meets $CR(f)$. Conversely, if
$x\in Sh^+(f)\cap CR(f)$, Theorem~\ref{ShadClass} gives
$H(x)\subset Sh^+(f)$, and $H(x)$ is nonempty, closed, and invariant.
\end{proof}

The next result converts the positive pointwise hypothesis into the two-sided shadowing property for pseudo-orbits contained in the chain class, in the ambient sense of Definition \ref{shaddef}.

\begin{theorem}\label{shadonH}
    If $x\in Sh^+(f)\cap CR(f)$, then $H(x)$ has the shadowing property as a subset of $M$.
\end{theorem}
\begin{proof}
Theorem~\ref{ShadClass} and its proof give a uniform positive tracing
constant on $H(x)$. Fix $\eps>0$, choose such a $\delta>0$, and let
$(x_j)_{j\in\mathbb Z}\subset H(x)$ be a two-sided
$\delta$-pseudo-orbit. For each $i\ge0$, the sequence
\[
y^i_j:=x_{j-i},\qquad j\ge0,
\]
is a positive $\delta$-pseudo-orbit through $x_{-i}$. Choose $z_i$ that
$\eps$-shadows $(y^i_j)_{j\ge0}$. Passing to a subsequence, compactness
gives $f^i(z_i)\to z$ for some $z\in M$. For fixed $n\in\mathbb Z$ and
all sufficiently large $i$, we have $i+n\ge0$, and hence
\[
d(f^n(z),x_n)
=\lim_{i\to\infty}d(f^{i+n}(z_i),x_n)\le\eps.
\]
Thus $z$ shadows the original two-sided pseudo-orbit.
\end{proof}

\begin{remark}
    Observe that a set $K$ having the shadowing property in the ambient sense of Definition \ref{shaddef} does not imply $K\subset Sh^+(f)$.  The former concerns only pseudo-orbits entirely contained in $K$, whereas positive pointwise shadowability must also trace pseudo-orbits that start in $K$ and subsequently leave it.
\end{remark}

Theorem \ref{shadonH} gives ambient  shadowing for $H(x)$ but does not ensure that the tracing orbit lies in $H(x)$. Thus it does not by itself give intrinsic shadowing for the restricted system $f|_{H(x)}$. Local maximality supplies precisely this missing conclusion as shown below. Recal that an invariant set $\Lambda$ is locally maximal, if there exists an open set $U\supset \Lambda$ such that $\Lambda$ is the largest invariant set contained in $U$. This implies that $\bigcap_{n\geq 0}f^n(U)=\Lambda$.

\begin{corollary}\label{cor:loc_max}
If $x\in Sh^+(f)\cap CR(f)$ and $H(x)$ is locally maximal, then $H(x)$ has the intrinsic shadowing property.    
\end{corollary}

\begin{proof}
Let $U$ be an isolating neighbourhood of $H(x)$ and fix a desired
accuracy $\eps_0>0$. Choose $0<\eps<\eps_0$ so small that
$B_\eps(H(x))\subset U$. By Theorem~\ref{shadonH}, every sufficiently
accurate two-sided pseudo-orbit in $H(x)$ is $\eps$-shadowed by some
$z\in M$. Its entire orbit lies in $U$, so local maximality gives
$z\in\bigcap_{n\in\mathbb Z}f^n(U)=H(x)$. This tracing point also
$\eps_0$-shadows the pseudo-orbit.
\end{proof}

Theorems~\ref{ShadClass} and~\ref{shadonH} give the propagation
mechanism that drives the rest of the paper. A single positively
shadowable chain-recurrent point gives uniform positive pointwise tracing
throughout its chain class and two-sided ambient shadowing for
pseudo-orbits contained in that class. Corollary~\ref{cor:loc_max} also
marks the exact point at which intrinsic shadowing enters: it follows from
local maximality, but not from pointwise tracing alone. Part~I shows that
this distinction is realized by homeomorphisms of manifolds.

\section{Part I: Realization and density}\label{sec:part-realization}

The principal result of this part is Theorem~\ref{thm:shadowable-dense}. On
every compact manifold of dimension at least two, it produces densely a
homeomorphism with a transitive chain component
$D\subseteq Sh(T)$ although neither the global map nor the restriction
$T|_D$ has the shadowing property. In fact, every chain recurrent class meeting a
neighborhood of $D$ fails shadowing. Thus the pointwise property occurs
on manifolds without approximation by invariant subsets with shadowing. The final subsection
gives a complementary realization on general compact metric spaces, showing that
arbitrary Cantor dynamics can be prescribed on the entire set of
two-sided shadowable points, yet again without shadowing.

The proof is organized in two layers.  The Denjoy--Rees construction and
entropy-spectrum obstructions are first combined in
Theorem~\ref{thm:abstract}, which is the general realization mechanism underlying the construction of the special
dynamical plugs.
Explicit local blocks in every dimension $n\geq 2$ are then inserted
into arbitrary manifold homeomorphisms, yielding the density theorem.
The compact-space realization uses a separate layered construction and
complements the manifold mechanism by prescribing the entire shadowable
set.

\subsection{Tools for controlled manifold realizations}\label{sec:manifold-dense}

\subsubsection{The Denjoy--Rees technique}
Given a homeomorphism $S$ of a compact manifold $M$, a measurable
$A\subseteq M$ and a Cantor set $C$, a map \emph{fibered over $S$} is a
bijection
$$h\colon\bigcup_{i\in\Z}S^i(A)\times C\to\bigcup_{i}S^i(A)\times C, \textrm{ so that }
h(x,c)=(S(x),h_x(c)),$$ with each $h_x\in\Homeo(C)$ and $h$ continuous
on $A\times C$. Recall that systems $(X,T)$, $(Y,S)$ are \emph{universally
	isomorphic} if there are invariant Borel sets $X_0\subseteq X$, $Y_0\subseteq
Y$ of full measure for every invariant measure (\emph{universally full})
and a bi-measurable bijection $\Theta\colon X_0\to Y_0$ with
$\Theta\circ T=S\circ \Theta$.

\begin{theorem}[{\cite[Theorem~1.3 and Addendum~1.4]{BCLR}}]\label{thm:BCLR}
	Let $S$ be a homeomorphism of a compact manifold $M$ of dimension at least $2$,
	let $\mu\in\Merg_S(M)$ be aperiodic, and let $A\subseteq M$ satisfy
	$\mu(A)>0$ and $\nu(A)=0$ for every $\nu\in\Merg_S(M)\setminus\{\mu\}$. Let
	$h$ be fibered over $S$ with Cantor fibre $C$. Then there are
	$T\in\Homeo(M)$ and a continuous surjection $\Pi\colon M\to M$ with
	$\Pi \circ T=S\circ \Pi$ such that:
	\begin{enumerate}[label=\textup{(\arabic*)}]
		\item\label{B:uiso} $(M,T)$ is universally isomorphic to
		$$\big(\bigcup_iS^i(A)\times C,\,h\big)\sqcup\big(M\setminus\bigcup_iS^i(A),\,S\big);$$
		\item\label{B:11} $\Pi$ is one-to-one outside $\Pi^{-1}(\Supp\mu)$;
		\item\label{B:trans} if $S$ is transitive on $M$, then $T$ can be chosen
		transitive on $M$;
		\item\label{B:transD} in any case, $T$ is transitive on $\Pi^{-1}(\Supp\mu)$.
	\end{enumerate}
\end{theorem}

We also use that $T$ is a controlled limit of conjugates of $S$.
The following can be derived from \cite[Prop.~3.1 and \S1.5.1]{BCLR}.
\begin{proposition}\label{prop:limit}
	In Theorem~\ref{thm:BCLR} one has 
    $$T_n \xrightarrow{\text{unif.}} T \textrm{ and } T^{-1}_n \xrightarrow{\text{unif.}} T^{-1},$$
     where $T_n=\Psi_n^{-1}\circ S\circ \Psi_n$ for homeomorphisms $\Psi_n$ of $M$. Consequently, $T\in Homeo(M)$. 
	Also $\Psi_x \xrightarrow{\text{unif.}} \Pi$ and the quantities $$\sigma_n:=\dH(T,T_n) \text{ and } \eta_n:=\dC(\Psi_n,\Pi)$$ can be
	prescribed to converge to $0$ as fast as desired. Specifically, after stage $n$ is fixed, the
	whole tail of the  construction is still free and only summability is
	required for convergence.
\end{proposition}

\begin{remark}
	By definition, each map $T_n$ is topologically
	conjugate to $S$ by conjugacy $\Psi_n$. 
	In particular, if $\Lambda$ is an invariant set for $S$ such that $S|_\Lambda$ has the shadowing property, then each $T_n|_{\Psi_n^{-1}(\Lambda)}$
	has the shadowing property.
\end{remark}

\subsubsection{Entropy gaps as an obstruction to shadowing}
We will destroy intrinsic and global shadowing by creating gaps in the spectrum of possible values of entropy. The required obstruction follows from the entropy-approximation theorem of Li and Oprocha \cite[Corollary~C(1)]{LiOp}.

\begin{theorem}\label{thm:LO}
	Let $(X,T)$ be transitive with the shadowing property. Then for every
	$0\le c<\htop(T)$ and every $\eps>0$, the ergodic measures with entropy in
	$[c,c+\eps]$ are weak${}^*$ dense in $\{\mu\in M_T(X):h_\mu(T)\ge c\}$.
\end{theorem}

\begin{corollary}\label{cor:nogap}
	If $(X,T)$ is transitive with shadowing, then for every $0\le c<\htop(T)$
	and $\eps>0$ there exists an ergodic measure with entropy in $(c,c+\eps)$.
\end{corollary}

\begin{proof}
	Fix $c'\in\big(c,\min\{c+\eps,\htop(T)\}\big)$ and $\eps'\in(0,\,c+\eps-c')$.
	Since $c'<\htop(T)$, the variational principle gives an invariant measure of
	entropy at least $c'$, so $\{\mu:h_\mu(T)\ge c'\}\neq\emptyset$. By
	Theorem~\ref{thm:LO} the ergodic measures of entropy in $[c',c'+\eps']$ are
	dense in this nonempty set, so at least one such ergodic measure exists and its
	entropy lies in $[c',c'+\eps']\subset(c,c+\eps)$.
\end{proof}

\begin{remark}
	Corollary~\ref{cor:nogap} is an effective tool in proving that a map does not have the shadowing property. We will use this later.
\end{remark}

The following is \cite[Theorem 4.3]{LiOp}.
\begin{theorem}\label{thm:sh-gap:classes}
	Suppose that $(X, T)$ has the shadowing property and $\mu \in M_T(X)$. If $\operatorname{supp}(\mu)/\sim_T$ is a singleton, then for every $0 \leq c \leq h_\mu(T)$ there exists a sequence of ergodic measures $(\mu_n)_{n=1}^\infty$ supported on almost one-to-one extensions of odometers such that $\lim_{n\to\infty} \mu_n = \mu$ and $\lim_{n\to\infty} h_{\mu_n}(T) = c$.
\end{theorem}

Together, the Denjoy--Rees theorem and the entropy-gap results provide the
two sides of the construction: a controlled extension and a verifiable
obstruction to shadowing. The next theorem packages them with the required
pseudo-orbit control into a reusable manifold-realization criterion.

\subsection{A controlled manifold-realization criterion}\label{sec:abstract}

The next theorem isolates the mechanism that implies local shadowing.
It reduces the manifold realization problem to
	a concrete base structure: a fully supported non-atomic ergodic measure on the central
	class, confinement of pseudo-orbits near that class, and nearby invariant
	shadowing sets that uniformly trace them.

\begin{theorem}\label{thm:abstract}
Let $S$ be a homeomorphism of a compact manifold $M$ with $\dim M\ge2$ and
$\htop(S)<\infty$. Suppose:
\begin{enumerate}[label=\textup{(H\arabic*)}]
\item\label{H:supp} $H(p_0)$ is a chain component of $S$ with
$H(p_0)=\Supp(\mu)$ for some non-atomic ergodic measure $\mu$;
\item\label{H:conf} there is a nested sequence of open sets
$U_n\supset H(p_0)$ 
such that for every $n$
there is $m$ for which every $\tfrac1m$--pseudo-orbit of $S$ meeting
$H(p_0)$ is contained in $U_n$;
\item\label{H:trace} there
are 
invariant sets $H(p_i)$, $i\ge1$\textup, pairwise disjoint and
disjoint from $H(p_0)$  with $B(H(p_i),\zeta_i)\supset H(p_0)$ for some
$\lim_i \zeta_i=0$, and such that each $S|_{H(p_i)}$ has the shadowing property and for every $\eps>0$ there exist
$i\ge1$, $n$, and $\delta>0$
with the property that every $\delta$--pseudo-orbit of
$S$ lying in
$U_n$ and intersecting
$H(p_0)$ is $\eps$--traced by a
point of $H(p_i)$.
\end{enumerate}
Then, for every $\eta>0$, there exist a homeomorphism $T\in\Homeo(M)$
and a continuous surjection $\Pi\colon M\to M$ such that
\[
        \dH(S,T)<\eta,
\]
$\Pi \circ T=S\circ \Pi$, 
$\Pi$ is injective on $M\setminus D$ where
$D:=\Pi^{-1}(H(p_0))$ and the following conditions hold:
\begin{enumerate}[label=\textup{(P\arabic*)}]
\item\label{con:P1} $(M,T)$ does not have the shadowing property;
\item\label{con:P2} $D$ is a transitive chain component of $(M,T)$ and $(D,T|_D)$ does not
have the shadowing property;
\item\label{con:P3} for every $\eps>0$ there is $\delta>0$ such that every
$\delta$--pseudo-orbit of $T$ meeting $D$ is $\eps$--traced by a point of
$M\setminus D$. In particular $D\subset \Sh(T)$.
\end{enumerate}
\end{theorem}

\begin{proof}
Note that since $\mu$ is ergodic,
$S|_{H(p_0)}$ is transitive.
Fix a uniquely ergodic Cantor system $(C,r)$ with
$H_C:=\htop(C,r)>\htop(S)$. It can be done by starting with Bernoulli measure of sufficiently large entropy on a subshift of large alphabet and then applying the Jewett--Krieger theorem. Denote this unique measure by $\alpha$ which also means that
$h_\alpha(r)=H_C$. Let $$A=\{x\in M : x \text{ is a generic point for }\mu \}\subseteq H(p_0).$$
Then $A$ is a Borel set such that $\mu(A)=1$ and $\nu(A)=0$ for every ergodic $\nu\neq\mu$, and $A$ is
$S$-invariant. Apply Theorem~\ref{thm:BCLR} with the aperiodic
ergodic $\mu$, the set $A$, and the product insertion $h=S|_A\times r$. The
construction can be performed in such a way that
$\dH(S,T)<\eta$. However we will need a more controlled approach. In fact in the perturbation  introduced in \cite{BCLR} we can take the rectangles for perturbation sufficiently small, making the convergence arbitrarily fast, which in the summarized statement of Proposition~\ref{prop:limit} means that for any  decreasing sequence $t_n\to 0$, we can require  $$\sigma_n:=\dH(T,T_n)<t_n \text{ and } \eta_n:=\dC(\Pi,\Psi_n)<t_n.$$
We will explain later how fast $t_n$ should decrease to imply the theorem.
Note however, that we can define $t_n$ when $T_n$ and $\Psi_n$ are already known, because
it influences only further steps of the construction.

At this point we
obtain $T$ and $\Pi$ satisfying:
\begin{itemize}
\item $\Pi\circ T=S\circ \Pi$, 
\item $\Pi$ is injective on complement of
$D=\Pi^{-1}(\Supp\mu)=\Pi^{-1}(H(p_0))$,
\item $T|_D$ is transitive (this is a consequence of Theorem~\ref{thm:BCLR}\ref{B:transD}, which holds unconditionally), 
\item $(M,T)$ universally isomorphic to
$$\big(A\times C,\,S\times r\big)\sqcup\big(M\setminus A,\,S\big).$$
\item  $T_n\xrightarrow{unif} T$ such that the is provided by $t_n$ and each $T_n$ conjugated with $S$ by a map $\Psi_n$.
\end{itemize}

Now, uniform continuity gives the existence of some quantifying functions:
\begin{enumerate}
\item A function $\omega_\Pi\colon (0,\infty)\to (0,\infty)$ such that every $\delta$-pseudo-orbit for $T$ is an $\omega_\Pi(\delta)$-pseudo-orbit for $S$. In addition, $\omega_\Pi(\delta)\to0$ as $\delta\to0$.
\item Symmetrically, for each $n$ we can obtain a function $\omega_{\Psi_n}\colon(0,\infty)\to(0,\infty)$ that quantifies how $\Psi_n$ transform pesudo-orbits of $T_n$ into pseudo-orbits of $S$. Precisely, if $\delta>0$ and $(x_i)$ is a $\omega_{\Psi_n}(\delta)$-pseudo-orbit for $T_n$, then $(\Psi_n(x_i))$ is  a $\delta$-pseudo-orbit for $S$.
\end{enumerate}
Observe that if $\delta_n\to0$, then $\omega_\Pi(\delta_n)\to0$ and hence $\Pi(x) \sim_S \Pi(y)$, provided  $x\sim_T y$. Since  $T|_D$ is transitive, we get that  all points in $D$ are $\sim_T$ related.
On the other hand, no point $x\in D$ is chain-related to a point $z\notin D$. Otherwise, $\Pi (z)\notin
H(p_0)$ and $\Pi (z)\sim_S \Pi(x)\in H(p_0)$, contradicting the assumption that $H(p_0)$ is a chain
component. Consequently, $D$ is a chain component of $(M,T)$, thereby proving  first statement in \ref{con:P2}.

The universal isomorphism is
defined on universally full invariant sets and commutes with the dynamics,
so it induces an entropy-preserving and ergodicity-preserving bijection between the
invariant measures of $(M,T)$ and those of the model. Restricting to the
invariant set $D$, the ergodic measures of $(D,T|_D)$ correspond to those of
$\big(A\times C,S\times r\big)\sqcup\big(H(p_0)\setminus A,S\big)$. If
an ergodic measure on the second piece corresponds to an $S$--measure $\nu$ on $H(p_0)$ giving
$A$ measure $0$, i.e. distinct from $\mu$, its entropy satisfies
$$h_\nu(T)\le\htop(S|_{H(p_0)})\le\htop(S),$$ here for simplicity, with a small abuse of notation we view $\nu$ as a measure for $T$. 

An ergodic measure $\nu$ of the first
piece projects under $\pi_C$ to the uniquely ergodic $(C,r)$ system, so
$(\pi_C)_*\nu=\alpha$ and consequently
$h_\nu\ge h_\alpha(r)=H_C$. Since $H_C>\htop(S)$, by the above discussion, there is no ergodic measure $\eta$ of $T|_D$ satisfying 
$h_{\eta}\in \big(\htop(S),H_C\big)$. Besides,  $$\htop(T|_D)=h_\mu(S)+H_C\ge H_C.$$ Namely, the maximum is attained by the lift
of the independent joining $\mu\times\alpha$, and no other joining has entropy exceeding $h_\mu(S)+H_C$.
Now, $T|_D$ is transitive and cannot have the shadowing property. Otherwise, since $\htop(S)<\htop(T|_D)$, for $c=\htop(S)$ and any $\eps>0$ so that $\htop(S)+\eps<H_C$, Corollary~\ref{cor:nogap}
gives an ergodic measure of $T|_D$ with entropy in
in the interval $(\htop(S),H_C)$. This contradiction completes the
proof of \ref{con:P2}.

Now, we turn our attention to \ref{con:P1}. We first note that the above entropy-gap is global. Indeed, under the universal
isomorphism, every ergodic measure of $T$ corresponds to one of the following cases:
\begin{itemize}
    \item A measure of $(A\times C,S\times r)$ with entropy $\ge H_C$.
    \item A measure of $(M\setminus A,S)$ with  entropy $\le\htop(S)$.
\end{itemize}    
Therefore,  no ergodic measure of
$(M,T)$ has entropy in $(\htop(S),H_C)$.
Since $D$ is a chain component, we have that $D/{\sim_T}$ is a singleton; its
 measure of maximal entropy $\bar\mu$ on $D$  satisfies that $\supp(\bar\mu)/{\sim_T}$ is a
singleton.
If $T$ had shadowing, then Theorem~\ref{thm:sh-gap:classes} applied to
$\bar\mu$ and any $c\in(\htop(S),H_C)$ would produce ergodic measures of $T$
with entropy converging to $c$, hence eventually inside $(\htop(S),H_C)$. This is a contradiction, proving \ref{con:P1}.

It remains to prove \ref{con:P3}. It will heavily depend on the choice of the sequence  $t_n$.
As $T_1$ we just take $T_1=S$.
Now fix $n$ and assume that $T_n$ is already defined.
Since $\Psi_n^{-1}$ is uniformly continuous, choose $\eps_n\in(0,\frac{1}{4n})$ such that
\begin{equation}\label{eq:modulus}
d(a,b)<2\eps_n \quad\Longrightarrow\quad d(\Psi_n^{-1}(a),\Psi_n^{-1}(b))<\tfrac{1}{2n}.
\end{equation}
By \ref{H:trace} applied with accuracy $\eps_n$ there are an index $j_n\ge1$, a neighbourhood index $k_n$, and $\delta_n>0$ such that every $\delta_n$--pseudo-orbit of $S$ lying in $U_{k_n}$ and meeting $H(p_0)$ is $\eps_n$--traced by a point of $H(p_{j_n})$.
By \ref{H:conf} applied to $U_{k_n}$ there is $m_n$ such that every $\tfrac1{m_n}$--pseudo-orbit of $S$ meeting $H(p_0)$ is contained in $U_{k_n}$.
Take any $\gamma_n<\min\{\frac{1}{m_n},\delta_n\}$.
Let $W_1,\ldots, W_s$ be an open cover of $H(p_{j_n})$ such that $\diam \Psi_{n}^{-1}(W_i)<\eps_n$.
Let $\lambda_n>0$ be a Lebesgue number of this cover.
Decrease $t_n$ if necessary, so that
$t_n<\min\{\eps_n,\lambda_n\}$, which will have impact on forthcoming perturbations, since by assumption  $\sigma_n,\eta_n<t_n$.

Before going further, first we claim that under these assumptions,
\begin{equation}\label{eq:near-fibres}
z\in H(p_{j_n}) \quad \Longrightarrow \quad d(\Pi^{-1}(z),\Psi_n^{-1}(z))<\eps_n.
\end{equation}
So assume otherwise, and let $z$ be such a point. Let $p=\Psi_n(\Pi^{-1}(z))$. Since $z=\Pi(\Pi^{-1}(z))$, we have
$$d(p,z)=d(\Psi_n(\Pi^{-1}(z)),\Pi(\Pi^{-1}(z)))\le\eta_n<\lambda_n,
$$
which means there is $r$ such that $p,z\in W_r$. But then $\Psi_n^{-1}(z),\Pi^{-1}(z)\in \Psi_n^{-1}(W_r)$, whose diameter is smaller than $\eps_n$, contradicting $d(\Pi^{-1}(z),\Psi_n^{-1}(z))\ge\eps_n$. This proves the claim.

We construct the sequence $t_n$ recursively in each step of the construction, which can be done by virtue of Proposition~\ref{prop:limit}: once stage $n$ is reached, $\eps_n,\delta_n,m_n,\gamma_n,\lambda_n$ are all determined while the tail of the construction is still free, so the bounds $\sigma_n,\eta_n<t_n$ required above can indeed be met.
Under the above assumptions, we are ready for the proof.

Fix any $\eps>0$ and let $n$ be sufficiently large so that $\frac{1}{n}<\eps$. Let $\delta>0$ be small enough, so that $\omega_\Pi(\delta)\le\gamma_n$. Fix any a $\delta$-pseudo-orbit  $(x_i)$ for $T$ meeting $D$, say $x_{i_0}\in D$.
Since $\Pi\circ T=S\circ\Pi$, the sequence $(\Pi(x_i))$ is an $\omega_\Pi(\delta)$-pseudo-orbit for $S$, hence a $\gamma_n$-pseudo-orbit. In particular it is both a $\tfrac1{m_n}$-pseudo-orbit and a $\delta_n$-pseudo-orbit. Moreover it meets $H(p_0)$, since $\Pi(x_{i_0})\in H(p_0)$. By \ref{H:conf} it is contained in $U_{k_n}$, and hence by \ref{H:trace} there is a point $\bar z\in H(p_{j_n})$ which $\eps_n$--traces $(\Pi(x_i))$ for $S$, that is,
$$ d(S^i(\bar z),\Pi(x_i))\le\eps_n \qquad \text{ for all }i\in\Z. $$

Recall that $\Pi$ has a well defined inverse  on $M\setminus \Pi(D)$.
Hence, there is a unique point $z:=\Pi^{-1}(\bar z)\in M\setminus D$, since $\bar z\in H(p_{j_n})$ is disjoint from $H(p_0)=\Pi(D)$. The point
$w:=\Psi_n^{-1}(\bar z)$ is also uniquely defined.
Denote $y_i:=\Psi_n(x_i)$ for each $i\in \Z$, that is $\Psi_n^{-1}(y_i)=x_i$,  and observe that
\begin{equation} d(S^i(\bar z),y_i)\le d(S^i(\bar z),\Pi(x_i))+d(\Pi(x_i),\Psi_n(x_i))\le\eps_n+\eta_n<2\eps_n.
	\label{eq:3-epsn}
\end{equation}
Note that
$$
\Pi(T^i(z))=S^i(\Pi(z))=S^i(\bar z)=S^i(\Psi_n(w))=\Psi_n(T_n^i(w)).
$$
Hence,   $$T^i(z)=\Pi^{-1}(S^i(\bar z)) \text{ and } T_n^i(w)=\Psi_n^{-1}(S^i(\bar z)).$$
But  $S^i(\bar z)\in H(p_{j_n})$, thus \eqref{eq:near-fibres} gives $d(T^i(z),T_n^i(w))<\eps_n$.
Additionally, by \eqref{eq:modulus} and \eqref{eq:3-epsn} we get
$$
d(T_n^i(w),x_i)=d(\Psi_n^{-1}(S^i(\bar z)),\Psi_n^{-1}(y_i))<\tfrac1{2n}.
$$
This leads to the following estimate
\begin{eqnarray*}
d\big(T^i(z),x_i\big)&\le& d\big(T^i(z), T_n^i(w)\big)+d(T_n^i(w),x_i)< \eps_n+\frac{1}{2n}<\frac{1}{n}<\eps.
\end{eqnarray*}
This shows that the point $z$, which by definition satisfies $z\in M\setminus D$, $\eps$--traces $(x_i)$.
In particular, any $\delta$-pseudo-orbit through a point $x\in D$ meets $D$, hence is $\eps$-traced; as $\eps>0$ was arbitrary, $D\subseteq\Sh(T)$.
This completes the proof of \ref{con:P3}.
\end{proof}

\subsection{A Plykin-type building block on spheres}\label{sec:plykin-block}

To apply the realization criterion, we require a local modification:a plug on the $n$-dimensional cube that leaves the boundary fixed. Such a plug allows us to insert a the desired dynamics into a given system without altering the dynamics outside the plug region. A crucial requirement for constructing this plug is that the systems under consideration are smoothly isotopic to the identity. As we shall see later, this isotopy property ensures that the plug can be seamlessly integrated into the ambient manifold. A classical example that provides the necessary utility is the Plykin attractor on the sphere $\mathbb{S}^2$ (see \cite[Example VII.9.1]{Robinson}), say given by a diffeomorphism $F_2$.

Recall that, for any given compact manifold $M$, the set of $C^\infty$-diffeomorphisms are $C^1$-dense in $\Diff^1(M)$ (see \cite[Ch.~2]{Hirsch}) and, also, that structural stability is a $C^1$-open condition. In this way, we may assume that the Plykin diffeomorphism $F_2$ belongs to $\Diff^\infty(\bbS^2)$, replacing it, if necessary, by a sufficiently $C^1$-close $C^\infty$ diffeomorphism. The resulting diffeomorphism  is then structurally stable and topologically conjugate to the original one.
The conjugacy preserves the shadowing property, carries non-wandering set onto non-wandering set, basic sets to basic sets, attractors to attractors, and preserves topological dimension. Since the smooth diffeomrphism $F_2$ is obtained as an arbitrarily small perturbation of
$C^1$ diffeomorphism in $C^1$-topology, the orientation and hyperbolicity are preserved, in particular the one dimensional attractor remains expanding and the fixed points remain sources.

A generalization of this example to a diffeomorphism $F_n\in \Diff^{\infty}(\bbS)$ in higher dimensions is provided by \cite[ Proposition~1]{MZ}. While the idea
is clear, the given proof is unfortunately incomplete. Precisely, a proof is provided for $n=3$, together with the comment that
for higher dimensions it can be done by an easy induction. However, the passage from $\bbS^{n}$ to $\bbS^{n+1}$
($n\geqslant 2$) as presented in \cite{MZ} requires a
\textit{diffeotopy} from the diffeomorphism constructed at the previous stage to the identity. Remember that, for $1\leqslant r\leqslant\infty$, a
$C^r$ diffeotopy is a $C^r$ map $H\colon\bbS^{n}\times[0,1]\to\bbS^{n}$ such that
every $H_t:=H(\cdot,t)$ is a diffeomorphism. Equivalently in the terminology of \cite[Chapter~8]{Hirsch}, its
\textit{track}, that is, the map $\Phi_H(x,t)=(H(x,t),t)$, is a $C^r$ diffeomorphism of $\bbS^{n}\times[0,1]$.

In \cite{MZ} this diffeotopy to the identity is supplied explicitly only for the base map on $\bbS^2$: a Plykin diffeomorphism of $\bbS^2$ preserves orientation, and $\Diff^+(\bbS^2)$ is path-connected by Smale's theorem \cite{Sm59}, so it is diffeotopic to $\id_{\bbS^2}$. The higher-dimensional cases are then asserted by analogy (the authors in \cite{MZ} say ``continuing in this way''). Each passage $\bbS^n\to\bbS^{n+1}$, however, again feeds the diffeomorphism constructed at stage $n$ into the tubular-neighborhood formula, and hence requires a diffeotopy from that $\bbS^n$-diffeomorphism to $\id_{\bbS^n}$. Orientation preservation no longer guarantees the existence of such a diffeotopy once $n\ge 6$. In particular, Milnor proves in \cite[Theorem~5]{Milnor} that there exists a diffeomorphism $f:\bbS^6 \to \bbS^6$ of degree $1$ which is
not differentiably ($C^\infty$) isotopic to the identity. This obstruction persists in the $C^1$ category.
Indeed, the inclusion $\Diff^\infty(\bbS^{n})\subset\Diff^1(\bbS^{n})$ is a weak homotopy equivalence
\cite[Lemma~1.6 and Theorem~1.5(1)]{KhM}; hence a $C^1$ diffeotopy from $f$ to the identity implies, on the
level of $\pi_0$, that $f$ and $\id$ are joined by a \emph{continuous} path in $\Diff^\infty(\bbS^{6})$, and such
a path, with the endpoints kept fixed, can be replaced by a $C^\infty$ diffeotopy using the relative
parametrized version \cite[Theorem~1.8(2)]{KhM}, applied exactly as in the proof of \cite[Theorem~1.9]{KhM}.
This contradicts Milnor's theorem and shows that his diffeomorphism is not $C^1$-diffeotopic to the
identity either; thus working with $C^1$ diffeomorphisms, as \cite{MZ} does, does not
remove the difficulty.
The proposition itself is nevertheless correct. The repair is to strengthen the
inductive statement. We present a detailed argument because this construction is used in the subsequent local models.
\begin{theorem}\label{thm:Plykin}
	For every $n\geqslant 2$ there is a diffeomorphism $F_n\colon\bbS^{n}\to\bbS^{n}$ such that:
	\begin{enumerate}
		\item\label{thm:Plykin:1} $F_n$ is structurally stable;
		\item\label{thm:Plykin:2} $\Omega(F_n)$ is the union of a one-dimensional nonorientable expanding
		attractor $\Lambda_n$ and finitely many hyperbolic periodic points, at least one of
		which is a fixed source;
		\item\label{thm:Plykin:3} $F_n$ is $C^\infty$-diffeotopic to $\id_{\bbS^{n}}$.
	\end{enumerate}
\end{theorem}
\begin{proof}
	We will prove the theorem by induction on $n$. In what follows, when we write smooth diffeomorphism or map, we aways mean $C^\infty$-regularity.
	The Plykin's diffeomorphism $F_2\colon\bbS^2\to\bbS^2$ (see \cite[Example~VII.9.1]{Robinson}) is structurally stable, and
	$\Omega(F_2)$ consists of a one-dimensional nonorientable expanding attractor and four
	hyperbolic source fixed points. In the standard model $F_2$ preserves orientation.
	As we explained before, we may assume that $F_2\in\Diff^\infty(\bbS^2)$.
	We can apply Smale's theorem \cite[Theorem~6]{Sm59}, obtaining that $F_2$ is diffeotopic
	to $\id_{\bbS^2}$.

	Next, assume that theorem is proved for some $n\geq 2$, and fix $F_n$ such that conditions \eqref{thm:Plykin:1}--\eqref{thm:Plykin:3}
	are satisfied. Let $H\colon\bbS^{n}\times[0,1]\to\bbS^{n}$ be a diffeotopy
	such that $H_0=F_n$ and $H_1=\id_{\bbS^n}$ provided by \eqref{thm:Plykin:3}.

	Regard $$\bbS^{n+1}=\{x : ||x||=1\}\subset\R^{n+2}$$ as the unit sphere and let
	$\bbS^n\times \{0\}$ be its equator with the generalized north/south poles $$N=(0,\dots,0,1) \text{ and } S=(0,\dots,0,-1).$$
	We define a  smooth tangent vector field $X$ on $\bbS^{n+1}$ as follows. We first  write
	$x=(y,t)\in\R^{n+1}\times\R$, where $t=x_{n+2}$ and 
	$\lVert y\rVert^2+t^2=1$, and then define
	\[
	 X(y,t)=\bigl(t^2y,-t(1-t^2)\bigr).
	\]
	So, $X$ is tangent to the sphere.Indeed,
	\[
	 \langle X(y,t),(y,t)\rangle
	 =t^2\lVert y\rVert^2-t^2(1-t^2)=0.
	\]
	  Its zero set is
	$\{N,S\}\cup(\bbS^n\times\{0\})$, its nonconstant trajectories are
	the meridian arcs, and along every trajectory the last coordinate satisfies
	$\dot t=-t(1-t^2)$. Let $\Gamma=(\Gamma_t)$ be the flow of generated by $X$, let 
	$G:=\Gamma_1$ be its time-one map and let $g(t)$ be the time-one map of the flow induced by the equation $\dot t=-t(1-t^{2})$.
	Linearizing $\dot t=-t(1-t^2)$ one finds that $N,S$ are
	hyperbolic sources of $g$ and
	that the equator attracts at the uniform rate $e^{-1}$.

	The map $$\Theta\colon\bbS^{n}\times(-1,1)\to\bbS^{n+1}\setminus\{N,S\}, \text{ given by }
	\Theta(z,t)=(\sqrt{1-t^{2}}\,z,\,t)$$ is a smooth diffeomorphism carrying the product
	foliation to the leafs of meridian foliation.
	Fix small $\delta\in (0,\frac{1}{8})$ and let $Q:=\Theta\big(\bbS^{n}\times[-\tfrac12,\tfrac12]\big)$
	be a tubular neighborhood of equator in $\bbS^{n+1}$.
	Finally, let $\theta\colon (-\frac{3}{4},\frac{3}{4})\to [0,1]$ be a smooth surjective even function, nondecreasing in $|t|$, such that $\theta(t)=0$ for $|t|\leq \delta$ and $\theta(t)=1$ for $|t|\geq \frac{1}{2}-\delta$.

	Define $F_{n+1}\colon\bbS^{n+1}\to\bbS^{n+1}$ by
	$$
	F_{n+1}(x)=
	\begin{cases}
		\Theta\big(H_{\theta(t)}(z),g(t)\big), & x=(z,t)\in Q\\
		G(x), & x\notin Q.
	\end{cases}
	$$
	Note that, by the definition, on the set $$\Theta\left(\bbS^{n},\left[\frac{-1}{2},\frac{-1}{2}+\delta\right]\cup \left[\frac{1}{2}-\delta,\frac{1}{2}\right]\right), \text{ one has } F_{n+1}(x)=G(x).$$
	This shows that $F_{n+1}$ is a smooth diffeomorphism as a composition of smooth diffeomorphisms.
	It is also clear that
	$$
	\Omega(F_{n+1})=(\Omega(F_n)\times\{0\})\cup\{N,S\}.
	$$
	By definition
	$N$ and $S$ are hyperbolic sources of $F_{n+1}$.
	On $Q_\delta:=\Theta(\bbS^{n}\times[-\delta,\delta])$ we have
	\begin{equation}\label{eq:product}
		F_{n+1}(z,t)=\Theta(F_n(z),g(t)).
	\end{equation}
	Therefore, along the equator, we have
	$$DF_{n+1}(z,0)=DF_n(z)\oplus\hat g'(0), \text{ with }
	\hat g'(0)=e^{-1}<1.$$ 
    Thus, for $x\in\Omega(F_n)$ the hyperbolic splitting of $F_{n+1}$ is given by
	\[
	E^{s}_{F_{n+1}}=E^{s}_{F_n}\oplus\nu,\qquad E^{u}_{F_{n+1}}=E^{u}_{F_n},
	\]
	where $\nu$ is the trivial and $DF_{n+1}$-invariant normal line bundle of the equator.
	By \eqref{eq:product} it is clear that $\Lambda_{n+1}=\Lambda_n$ is the only nondegenerate attractor of $F_{n+1}$
	and the number of sources/saddles increased by $2$ compared to $F_n$. In particular, periodic points are dense in $\Omega(F_{n+1})$.
	It is clearly one-dimensional and expanding by the construction. Since $\Lambda_n$ was nonorientable, so is $\Lambda_{n+1}$.
	This all together, shows that $\Omega(F_{n+1})$ is hyperbolic and $F_{n+1}$ is an axiom~A diffeomorphism.

	We claim that $F_{n+1}$ satisfies strong transversality condition, which together with Axiom A implies its
	structural stability (e.g. see \cite[Theorem~IX.9.2]{Robinson}).
To prove the claim, fix any two basic sets $\xi,\eta$ of $F_{n+1}$ and suppose
	$W^{u}(p)\cap W^{s}(q)\neq\varnothing$ for some $p\in\xi$, $q\in\eta$. If $\xi\in\{N,S\}$ then
	$W^{u}(p)$ is open and any intersection is transversal. If $\xi\subset \bbS^n\times \{0\}$  and
	$\eta\in\{N,S\}$ then $W^{s}(q)=\{q\}$ misses $W^{u}(p)\subset g^{-1}(\{0\})$, so this case is excluded. If both $\xi,\eta$ belong to the equator $\bbS^n\times\{0\}$
	every intersection point $x=(z,0)$ satisfies% was: lsatisfies
	$z\in W^{u}_{F_n}(p')\cap W^{s}_{F_n}(q')\subset\bbS^n$ for $p=(p',0)$, $q=(q',0)$.
	Since $F_n$ is a structurally stable diffeomorphism, it satisfies the strong transversality condition, as first proved by Ma\~n\'e (see \cite[Theorem~X.4.3 and Remark~X.4.5]{Robinson}).
	But $$W^s(F_{n+1})(q)=W^s(F_n)(q')\oplus \nu,$$ where $\nu$ is direction transversal to equator, proving strong transversality condition
	for $F_{n+1}$, hence proving the claim.

	Finally, for $s\in[0,\frac{1}{2}]$ set
	$\sigma_s(t):=(1-\theta(s))\theta(t)+\theta(s)$ and define
	\[
	F^{(s)}(x)=
	\begin{cases}
		\Theta\big(H_{\sigma_s(t)}(z),g(t)\big), & x=(z,t)\in Q\\
		G(x), & x\not\in Q .
	\end{cases}
	\]
	Since $\sigma_s(t)=1$ for every $|t|\geq \frac12-\delta$ and every $s$, each
	$F^{(s)}$ glues smoothly with $G$, so each $F^{(s)}$ is a
	diffeomorphism and clearly the family is smooth in $(s,x)$ jointly. Now
	$\sigma_0=\theta$ gives $F^{(0)}=F_{n+1}$, while $\sigma_{\frac{1}{2}}\equiv1$ gives $F^{(\frac{1}{2})}=G$ on $Q$, hence
	$F^{(\frac{1}{2})}=G$ globally. Additionally, observe that $\sigma_s\equiv 1$ for $s$ close to $\tfrac12$ and $\theta(1-s)=1$ for $s$ close to $\tfrac12$, $\theta(1-s)=0$ for $s$ close to $1$, so the concatenation below is smooth in $s$. Concatenating $s\mapsto F^{(s)}$ with the flow path
	$(x,s)\mapsto \Gamma(x,\theta(1-s))$ for $s\in [\frac{1}{2},1]$ yields
	a $C^\infty$ diffeotopy from $F_{n+1}$ to $\Gamma_0=\id_{\bbS^{n+1}}$. This proves \eqref{thm:Plykin:3}, completing the proof.
\end{proof}

\subsection{Local models on cubes in dimensions at least three}\label{sec:cube-models}

The aim of this subsection is to verify the three hypotheses of
Theorem~\ref{thm:abstract} in an explicit local model. The central fibre
carries the Plykin attractor and the neighboring fibres reproduce its shadowing
dynamics and trace pseudo-orbits through the central class. Radial dynamics prevents these pseudo-orbits from escaping. This yields the local realization
in Theorem~\ref{thm:example}.

For simplicity of notation, we fix $n\geq 2$ and will perform the construction
on $\bbS^n\times [-1,1]$ and then collapse one of the spheres to a point.
First we will define a map $\gamma$ which will be responsible for dynamics on lines transversal to $\bbS^n$.

Fix a decreasing sequence $(a_i)\subset(0,1)$ with $a_i\to0$. Let
$\gamma\in\Homeo([-a_1,a_1])$ be the increasing homeomorphism with
\[
\Fix(\gamma)=\{0\}\cup\{\pm a_N:N\ge1\}.
\]
Its nonzero fixed points are alternately attracting and repelling, and
for $x\ge0$ the map is given by
$$
\gamma|_{[a_{2n+1},a_{2n}]}(x)=\frac{(x-a_{2n+1})^2}{a_{2n}-a_{2n+1}}+a_{2n+1},
\qquad
\gamma|_{[a_{2n},a_{2n-1}]}(x)=\frac{(x-a_{2n-1})^2}{a_{2n}-a_{2n-1}}+a_{2n-1},
$$
and by $\gamma(-x)=-\gamma(x)$ for $x>0$. A direct computation gives $\gamma<\mathrm{id}$ on
each interval $(a_{2n+1},a_{2n})$ and $\gamma>\mathrm{id}$ on each interval $(a_{2n},a_{2n-1})$.

For a connected component $J=(p,q)$ of $[-a_1,a_1]\setminus\Fix(\gamma)$ we set
$$
\beta_J:=\max_{t\in\overline J}|\gamma(t)-t|=|\gamma(u_J)-u_J|>0,
$$
where $u_J\in J$ is a point of maximal displacement. We call $J$ a \emph{down-gap} if
$\gamma<\mathrm{id}$ on $J$ and an \emph{up-gap} if $\gamma>\mathrm{id}$ on $J$. Recall that a
sequence $(x_i)$ is a \emph{$\delta$-pseudo-orbit} of $\gamma$ if
$|x_{i+1}-\gamma(x_i)|\le\delta$ for all $i$.

\begin{lemma}\label{lem:gamma}
With $\gamma$ as above:
\begin{enumerate}[label=\textup{(\alph*)}]
\item Consecutive gaps (those sharing a common endpoint $a_N$) alternate between down-gaps and
up-gaps; by oddness of $\gamma$, gaps of both types occur at levels arbitrarily close to $0$
on each side of $0$.
\item Let $J$ be a gap on the positive side, so that $u_J>0$, and let $0<\delta<\beta_J$.
If $(x_i)$ is a $\delta$-pseudo-orbit of $\gamma$ with $x_0\in[-u_J,u_J]$, then
  \begin{itemize}
  \item\label{lem:gamma:a} if $J$ is a down-gap, $x_i\in[-u_J,u_J]$ for every $i\ge0$;
  \item\label{lem:gamma:b} if $J$ is an up-gap, $x_i\in[-u_J,u_J]$ for every $i\le0$.
  \end{itemize}
A gap on the negative side reduces to this case by replacing $J$ with $-J$, which has the
opposite type and the same $\beta$ and $|u_J|$.
\end{enumerate}
\end{lemma}

\begin{proof}
The proof is an easy calculation; details are left to the reader.
\end{proof}

We are now ready to assemble the example. Let $F_n\colon\bbS^n\to\bbS^n$ be the
diffeomorphism provided by Theorem~\ref{thm:Plykin}, and let $\Lambda$ be its one-dimensional expanding attractor.
Since $\Lambda$ is a basic set, there exists
an ergodic measure $\mu$ with $\Supp\mu=\Lambda$. In particular
$F_n|_\Lambda$ is transitive and $\mu$
is non-atomic. By result of \cite[Theorem~1]{Ito}, we have $\htop(F_n)<\infty$, because it is a diffeomorphism.
Extend $\gamma$ to an increasing homeomorphism of $[-1,1]$ with
$\gamma(-1)=-1$ and $\gamma(1)=1$ and no fixed point in $(a_1,1)$ nor $(-1,-a_1)$ and both $1,-1$ are attractors.
Define homeomorphism $F$ of $\bbS^n\times [-1,1]$ in the following way. First we put:
$$F|_{\bbS^n\times [-a_1,a_1]}:=F_n\times \gamma|_{[-a_1,a_1]}.$$
Next, let $H\colon \bbS^n\times [a_1,1]\to \bbS^n$ be a smooth diffeotopy between $H_{a_1}=F_n$ and $H_1=\id_{\bbS^n}$.
Then we define $$F(x,t)=(H(x,|t|),\gamma(t)), \text{ for } (x,t)\in \bbS^n\times [-1,-a_1]\cup [a_1,1].$$
It is clear that $F$ is well defined and the identity on the boundary of $\bbS^n\times [-1,1]$ (which is attracting).
Define the compact manifold
\[
  M:=\bigl(\bbS^n\times[-1,1]\bigr)\big/\bigl(\bbS^n\times\{-1\}\sim *\bigr),
\]
by collapsing the external sphere $\bbS^n\times\{-1\}$ to a single point $*$. Since the cone on $\bbS^n$ is
the ball $D^{n+1}$ with the apex an interior point, $M$ is homeomorphic to the $(n+1)$-dimensional cube,
with interior fixed point $*$ and boundary $\partial M=\bbS^n\times\{1\}$.
Clearly $F$ induces quotient map $S$ on $M$, which is a homeomorphism.
On the closed cylinder $\bbS^n\times[-a_1,a_1]$ map $S$ equals
$F_n\times\gamma$.
Observe $\dim M=n+1\ge3$ and $\htop(S)=\htop(F)<\infty$, because $F$ is identity in the collapsed fibre. In addition, the 
sets
$$
  H(p_0):=\Lambda\times\{0\},\qquad H(p_i):=\bbS^n\times\{a_i\}\ \ (i\ge1),\qquad
  U_k:=\bbS^n\times(-\tfrac1k,\tfrac1k),
$$
are all contained in the interior cylinder, away from $*$ and $\partial M$, so their dynamics is not affected by the extension of $\gamma$, neither by the collapse.
In what follows, we assume that the product $\bbS^n\times [-1,1]$ is endowed with the maximum metric and that this metric
coincides with metric on $M$ within the cylinder $\bbS^n\times [-a_1,a_1]$.

\begin{lemma}\label{prop:example}
Dynamical system $(M,S)$ satisfies hypotheses \ref{H:supp}--\ref{H:trace} of
Theorem~\ref{thm:abstract} with the data above.
\end{lemma}

\begin{proof}
The map $S|_{H(p_0)}$ is conjugate to $F_n|_\Lambda$ via $(x,0)\mapsto x$,
hence transitive, and so chain transitive. To see that $H(p_0)$ is a whole chain component,
suppose $w\in M$ is chain equivalent to some $(\xi,0)$ with $\xi\in\Lambda$. The map $\gamma$ is
gradient-like, so its chain components are the singletons $\{0\}$ and $\{\pm a_N\},\{\pm 1\}$; in
particular $0$ is chain related to no other fixed level, so no chain issued from level $0$ can
reach the levels $\pm a_1$ or $\pm 1$. Hence $w$ lies in the
interior cylinder, where the product structure holds, say $w=(x_w,t_w)\in \bbS^n\times [-a_1,a_1]$; projecting the chains to the two factors
gives $t_w\sim_\gamma 0$, so $t_w=0$, and $x_w\sim_{F_n}\xi$, so $x_w\in\Lambda$. Thus the chain component is exactly
$\Lambda\times\{0\}$. Finally $H(p_0)=\Supp(\mu\times\delta_0)$, and $\mu\times\delta_0$ is
non-atomic and ergodic. This proves that \ref{H:supp} is satisfied.

The sets $U_k$ are open, nested and contain $H(p_0)$. Fix any $k$.
By
Lemma~\ref{lem:gamma}\eqref{lem:gamma:a} there are a down-gap $J_d$ and an up-gap $J_u$ inside $(0,\frac{1}{k})$. Put
$\beta:=\min\{\beta_{J_d},\beta_{J_u}\}>0$ and pick $m$ with $\frac{1}{m}<\beta$.
Let
$(z_j)=(x_j,t_j)_{j\in \Z}$ be a $\frac{1}{m}$-pseudo-orbit meeting $H(p_0)$, say $t_{j_0}=0$, $x_{j_0}\in \Lambda$.
The intervals $[-u_{J_d},u_{J_d}]$ and $[-u_{J_u},u_{J_u}]$ lie in
$(-\frac{1}{k},\frac{1}{k})\subset(-a_1,a_1)$ and, by Lemma~\ref{lem:gamma}\eqref{lem:gamma:b}, trap the $t_j$-coordinate forward
(via $J_d$) and backward (via $J_u$) starting from $t_{j_0}=0$. Thus the pseudo-orbit $(z_j)$ never leaves the
set $U_k$, proving that \ref{H:conf} holds.

The sets $H(p_i)=\bbS^n\times\{a_i\}$ are invariant, pairwise disjoint
disjoint from $H(p_0)$ and, for each $i$, satisfy:
\begin{itemize}
\item $B(H(p_i),\zeta_i)\supset H(p_0)$, where $\zeta_i:=2a_i$ and $a_i\to0$.
\item $S|_{H(p_i)}(x,t)=(F_n(x),t)$ has the shadowing property.
\end{itemize}
Furthermore, $\delta$ chosen for $\eps$
by the shadowing property depends only on $F_n$, and hence is independent
of $j$.
Given $\eps>0$, choose $i$ with $a_i<\frac{\eps}{2}$ and $k$
with $\frac{1}{k}<\frac{\eps}{2}$, and let $\delta>0$ be such that every $\delta$-pseudo-orbit of $F_n$ is $\eps$-traced. Let
$(z_j)=((x_j,t_j))$ be a $\delta$-pseudo-orbit lying in $U_k$ and meeting $H(p_0)$. Then, since $S$ is the product map in the interior cylinder, $(x_j)$
is a $\delta$-pseudo-orbit of $F_n$. There exists $z\in\bbS^n$ with $d_{\bbS^n}(F_n^j(z),x_j)\le\eps$
for all $j$. The point $(z,a_i)\in H(p_i)$ satisfies $$S^j(z,a_i)=(F_n^j(z),a_i) \text{ and }
|a_i-t_j|\le a_i+\frac{1}{k}<\eps,$$ whence
\[
  d\big(S^j(z,a_i),(x_j,t_j)\big)=\max\{d_{\bbS^n}(F_n^j(z),x_j),\,|a_i-t_j|\}\le\eps .
\]
Thus $(z_j)$ is $\eps$-traced by a point of $H(p_i)$, completing the proof of \ref{H:trace} and finishing the proof of Lemma.
\end{proof}

Lemma~\ref{prop:example} together with Theorem~\ref{thm:abstract} yields the following.

\begin{theorem}\label{thm:example}
Let $n\geq 3$ and let $M$ be an $n$-dimensional cube. There exists a homeomorphism $T\in \Homeo(M)$
and an uncountable set $D\subset \interior M$ such that:
\begin{enumerate}
\item $D$ is a chain component such that $T|_D$ is transitive and $D\subseteq\Sh(T)$;
\item  none of the maps $(M,T)$, $(D,T|_D)$ has the shadowing property;
\item $T_{\partial M}=\id_{\partial M}$;
\item $\htop(T)<\infty$.
\end{enumerate}
\end{theorem}

Let us note that the set $D$ is indeed uncountable: by
Theorem~\ref{thm:BCLR}\ref{B:uiso}, up to a universally full set, $D$
contains a copy of the product $A\times C$, where $C$ is a Cantor set.

\begin{remark}
Observe that the tracing mechanism used to prove \ref{con:P3} requires only the uniform $\eps_i$-tracing of $\delta_i$-pseudo-orbits appearing in \ref{H:trace}. It does not require the restrictions to the sets $H(p_i)$ themselves to retain the shadowing property.
	Therefore, by a careful, consecutive applications of Theorem~\ref{thm:BCLR} with sufficiently fast decreasing range of perturbation,
	it is possible to extend the statement of Theorem~\ref{thm:example} by consecutive blow-up perturbations performed on sets $H(p_n)$.
	This way, we can destroy shadowing on each $H(p_i)$. While we cannot get rid of hyperbolic sinks, sources this way, we can prove that $D$ has a neighborhood
	$V$ such that if $D'\cap V$ is a chain-recurrent class then $T|_{D'}$ does not have shadowing.
\end{remark}

\begin{proposition}\label{prop:L3}
Let $n\ge3$, let $M$ be the $n$-dimensional cube, and let $T$ be the
homeomorphism of Theorem~\ref{thm:example}. For every $\eta>0$ there is
$T'\in\Homeo(M)$ with $\dH(T,T')<\eta$ which satisfies all conclusions of
Theorem~\ref{thm:example} (with the same chain component $D$) and, in
addition, there is an open set $V\supset D$ such that every chain component
of $T'$ meeting $V$ fails to have the shadowing property.
\end{proposition}

\begin{proof}
Recall the data of the construction: $T$ was obtained from
Theorem~\ref{thm:abstract} applied to $(M,S)$, with a factor map
$\Pi_0\colon(M,T)\to(M,S)$. We have $\Pi_0\circ T=S\circ\Pi_0$, $\Pi$ is injective outside
$D=\Pi_0^{-1}(\Lambda\times\{0\})$, and there is an inserted strictly ergodic
Cantor system of entropy $H_C>\htop(S)=:b$, where
$b=\htop(F_n|_\Lambda)=\htop(F_n)$.

\emph{\textbf{Step 1:} chain components of $T$ near $D$.}
Since $\gamma$ is gradient-like and $F_n$ is Axiom~A, the chain components
of $S$ are: $\Lambda\times\{0\}$; the sets $\Lambda\times\{\pm a_i\}$ and
$P\times\{\pm a_i\}$, where $P$ runs over the finitely many periodic
orbits (sources and saddles) of $F_n$; the fixed point $*$; and $\partial
M$. As in the proof of Theorem~\ref{thm:abstract}, $x\sim_{T}y$ implies
$\Pi_0(x)\sim_S\Pi_0(y)$, so distinct chain components of $S$ cannot
merge under the extension. On the complement of $D$, the map $\Pi_0$ is
one-to-one, so on a compact neighbourhood of any remaining chain component
it is therefore a homeomorphism onto its image, with uniformly continuous
inverse. Chains in that component consequently lift through
$\Pi_0^{-1}$. Hence the chain components of $T$ are $D$ together with the
homeomorphic $\Pi_0$-preimages of the remaining components; in particular,
for each $i$ the set
$\widehat\Lambda_{\pm i}:=\Pi_0^{-1}(\Lambda\times\{\pm a_i\})$ is a
transitive chain component of $T$, conjugate to $F_n|_\Lambda$, carrying a
non-atomic ergodic measure of full support, and with
$\htop(T|_{\widehat\Lambda_{\pm i}})=b<\infty$.

\emph{\textbf{Step 3:} choice of $V$.}
The periodic orbits of $F_n$ lie at positive distance from the attractor
$\Lambda$; choose an open set $N_\Lambda\supset\Lambda$ in $\bbS^n$ whose
closure misses all of them, and $r\in(0,a_1)$, and put
$V:=\Pi_0^{-1}\bigl(N_\Lambda\times(-r,r)\bigr)$, an open neighbourhood of
$D$. By Step 1, the chain components of $T$ meeting $V$ are exactly $D$
and the sets $\widehat\Lambda_{\pm i}$ with $a_i<r$.

\emph{\textbf{Step 3:} consecutive relative insertions.}
Fix $i_0$ with $a_i<r$ for $i\ge i_0$. For $i\ge i_0$ choose open sets
$V_{\pm i}\Subset W_{\pm i}\Subset N_{\pm i}\subset V$ around
$\widehat\Lambda_{\pm i}$ with
$V_{\pm i}\cup T^{-1}(V_{\pm i})\subset W_{\pm i}$, such that the sets
$N_{\pm i}\cup T(N_{\pm i})$ are pairwise disjoint and disjoint from $D$
and from all periodic-orbit components. Let $(\eta_i)_{i\ge i_0}$ be
summable with $\sum\eta_i<\eta$, each $\eta_i$ smaller than the trapping
constants of the adjacent $\gamma$-gaps (so that the radial chain barriers
of Lemma~\ref{lem:gamma} persist under perturbations of size $\eta_i$).
Apply Proposition~\ref{prop:2d-BCLR-control} (whose statement and proof
concern a general compact manifold) to $T$ at
$K:=\widehat\Lambda_{\pm i}$ with the pair
$V_{\pm i}\Subset W_{\pm i}$, size $\eta_i$, and the product insertion by
a strictly ergodic Cantor system of entropy $H_C$ (the same constant as
above), over the set of generic points of the transported measure. This
yields homeomorphisms $T^{(\pm i)}$ equal to $T$ off $W_{\pm i}$ and
factor maps $\Theta_{\pm i}$ equal to $\id$ off $V_{\pm i}$ with
$$\Theta_{\pm i}\circ T^{(\pm i)}=T\circ\Theta_{\pm i} \text{ and }
\dH(T^{(\pm i)},T)+\dC(\Theta_{\pm i},id)<\eta_i.$$ Patch the maps
$T^{(\pm i)}$ into a single homeomorphism $T'$ and the maps
$\Theta_{\pm i}$ into a factor map $\Theta\colon(M,T')\to(M,T)$, exactly
as in the patching argument of Section~\ref{sec:square} (the same
partition-image argument shows $T'$ is a homeomorphism, and continuity at
$D$, where the supports accumulate, follows from $\eta_i\to0$). Then
$\dH(T,T')\le\sum\eta_i<\eta$, $T'=T$ on
$M\setminus\bigcup W_{\pm i}$ (in particular near $\partial M$, $*$, $D$
and the periodic-orbit components), and $\Theta=\id$ off
$\bigcup V_{\pm i}$.

\emph{\textbf{Step 4:} chain components of $T'$ and failure of shadowing near $D$.}
Uniform continuity of the factor map $\Theta$ shows that a chain relation
for points under $T'$ projects to a chain relation for $T$. Thus components of $T'$ over
distinct chain components of $T$ cannot merge. Over every unmodified
component, $\Theta$ is the identity and $T'=T$ while over
$\widehat\Lambda_{\pm i}$, $i\ge i_0$, its full preimage is
$D_{\pm i}:=\Theta_{\pm i}^{-1}(\widehat\Lambda_{\pm i})\subset N_{\pm i}$.
The radial barriers, which persist by the choice of the perturbation
sizes, prevent additional chain recurrence in the regions between these
preimages. Hence the chain components of $T'$ are those of $T$ with each
$\widehat\Lambda_{\pm i}$ replaced by $D_{\pm i}$;
transitivity of $T'|_{D_{\pm i}}$ and the fact that $D_{\pm i}$ is a chain
component follow by the argument proving \ref{con:P2} of
Theorem~\ref{thm:abstract} applied to $(T^{(\pm i)},\Theta_{\pm i})$. By
the universal isomorphism of Theorem~\ref{thm:BCLR}, every ergodic measure
of $T'|_{D_{\pm i}}$ has entropy at most $b$ or at least $H_C$, while
a product measure supplied by that theorem has entropy at least $H_C$. Thus
$\htop(T'|_{D_{\pm i}})\ge H_C$, and the nonempty interval $(b,H_C)$ is
disjoint from the the possible values of entropies of ergodic measures and lies below the topological
entropy. Since $T'|_{D_{\pm i}}$ is transitive with finite entropy,
Corollary~\ref{cor:nogap} shows that
$(D_{\pm i},T'|_{D_{\pm i}})$ does not have the shadowing property. The
chain components of $T'$ meeting $V$ are exactly $D$ and the sets
$D_{\pm i}$ with $i\ge i_0$, and none of them has the shadowing property
(for $D$ this is Theorem~\ref{thm:example}(2), reproved for $T'$ in
Step~5). This establishes the additional conclusion with the set $V$ of
Step~2.

\emph{\textbf{Step 5:} the conclusions of Theorem~\ref{thm:example} persist.}
All new ergodic measures created in Step~3 have entropy at most $b$ or at
least $H_C$, so the global ergodic entropy gap $(b,H_C)$ of $(M,T)$
persists for $(M,T')$; hence the arguments proving \ref{con:P1} and
\ref{con:P2} in Theorem~\ref{thm:abstract} (via
Theorem~\ref{thm:sh-gap:classes} and Corollary~\ref{cor:nogap}) apply to
$T'$ verbatim and show that neither $(M,T')$ nor $(D,T'|_D)$ has the
shadowing property; moreover
$\htop(T')\le b+H_C<\infty$ and $T'|_{\partial M}=\id$.
It remains to check $D\subseteq\Sh(T')$. Fix $\eps>0$. In the proof of
\ref{con:P3} of Theorem~\ref{thm:abstract} the tracing point was found in
$\widehat H(p_j):=\Pi_0^{-1}(\bbS^n\times\{a_j\})$, where the index $j$
may be taken arbitrarily large; choose $j$ with $\eta_j<\frac{\eps}{3}$ and
$\delta_0>0$ such that every $\delta_0$-pseudo-orbit of $T$ meeting $D$ is
$(\frac{\eps}{3})$-traced by a point of $\widehat H(p_j)$. Now let $(z_k)$ be a
$\delta$-pseudo-orbit of $T'$ meeting $D$, with $\delta$ small enough that
(a) the sequence $(\Theta(z_k))$ is a $\delta_0$-pseudo-orbit of $T$
(uniform continuity of $\Theta$ together with
$\Theta\circ T'=T\circ\Theta$), which meets $D$ since $\Theta=\id$ on $D$,
and (b) by the persistence of the confinement \ref{H:conf}, $(z_k)$ meets
only those sets $V_{\pm i}$ with $\eta_i<\frac{\eps}{3}$, so that
$d(\Theta(z_k),z_k)<\frac{\eps}{3}$ for all $k$. Let
$\zeta\in\widehat H(p_j)$ be a point whose $T$-orbit $(\frac{\eps}{3})$-traces
$(\Theta(z_k))$, and pick $\zeta'\in\Theta^{-1}(\zeta)$. Then
$$\Theta((T')^k(\zeta'))=T^k(\zeta)\in\widehat H(p_j) \text{ for all } k,$$ and
since a neighborhood of $\widehat H(p_j)$ meets the support of $\Theta$
only in $V_{\pm j}$, we get
$d((T')^k(\zeta'),T^k(\zeta))\le\eta_j<\frac{\eps}{3}$. Combining,
\[
d\bigl((T')^k(\zeta'),z_k\bigr)\le
d\bigl((T')^k(\zeta'),T^k(\zeta)\bigr)+
d\bigl(T^k(\zeta),\Theta(z_k)\bigr)+
d\bigl(\Theta(z_k),z_k\bigr)<\eps .
\]
Hence every sufficiently fine pseudo-orbit of $T'$ through $D$ is
$\eps$-traced, i.e.\ $D\subseteq\Sh(T')$, which completes the proof.
\end{proof}

\begin{remark}
Since $T$ is the identity on $\partial M$, we can collapse $\partial M$ to a point, obtaining a homeomorphism of $\bbS^{n+1}$.
Similarly, using an automorphisms of $\bbT^n$
we can perform this construction on a solid torus of any dimension $n\geq 2$, with the main difference that the map on the boundary is not the identity, since the relevant torus automorphisms are not homotopic to identity.
\end{remark}

\subsection{The two-dimensional local model}\label{sec:square}

Theorem~\ref{thm:example} gives the local block in dimensions at
	least three and yields the stronger inclusion
	$D\subseteq Sh(T)$. The product construction starts in dimension three,
	however, because it requires a transverse direction in addition to the
	sphere carrying the Plykin dynamics. Dimension two therefore needs a
	different mechanism. To deal with this problem, we will provide an annular family of circle maps and controlled
Brown--Barge--Martin-type blocks, while preserving the same three required
conclusions: classical shadowability of the central component, failure of
intrinsic and standard shadowing, and finite entropy.  We first construct the
model on an annulus and then pass to the square by collapsing one boundary
component.

Write
$\mathbb A=[-1,1]\times\mathbb{S}^1$ and fix an irrational rotation
$R_\alpha\colon\mathbb{S}^1\to\mathbb{S}^1$.

\subsubsection{A tracing family of circle maps}

Let $\tau\colon\R\to[-1,1]$ be the $1$-periodic piecewise linear map such
that $\tau(k)=-1$, $\tau(k+\tfrac12)=1$, and whose slopes are $\pm4$.
Given $0<\eps<\frac{1}{4}$, choose $0<\delta<\frac{\eps}{6}$ and $N\in\bbN$ such that
$\frac{4}{\delta}<N\le \frac{5}{\delta}$, and define a degree-one circle map $T_\eps$ by the lift
\[
 \widetilde T_\eps(x)=x+\alpha+3\delta\tau(Nx).
\]
Then $d_{C^0}(T_\eps,R_\alpha)<3\delta<\eps$.  Moreover, every arc
$[a,a+\delta]$ contains a full period of $\tau(N\cdot)$, and hence
\begin{equation}\label{eq:2d-covering}
 T_\eps([a,a+\delta])\supset
 [R_\alpha(a)-2\delta,R_\alpha(a)+2\delta].
\end{equation}
The absolute values of the slopes of $T_\eps$ are greater than $4$.
Consequently every nondegenerate arc eventually covers the circle, so
$T_\eps$ is topologically mixing.  It has a non-atomic ergodic invariant
measure of full support and finite topological entropy, since it is
piecewise monotone and uniformly Lipschitz (indeed, the above choice gives $\delta N\le5$), so its topological entropy is finite and bounded independently of $\eps$.

\begin{proposition}\label{prop:2d-circle-tracing}
Every two-sided $\delta$-pseudo-orbit $(x_i)_{i\in\Z}$ of $R_\alpha$ is
$\eps$-traced by a full orbit of $T_\eps$.
\end{proposition}

\begin{proof}
Put $I_i=[x_i-\delta,x_i+\delta]$.  By
\eqref{eq:2d-covering}, $T_\eps(I_i)\supset I_{i+1}$ for every $i$.
Applying this inclusion to finite blocks and then using compactness gives a
full orbit $(z_i)_{i\in\Z}$ with $z_i\in I_i$ for every $i$.  Thus
$d(z_i,x_i)\le\delta<\eps$.
\end{proof}

\subsubsection{Local Brown--Barge--Martin blocks}

We use a standard local form of the Brown--Barge--Martin construction; see
\cite[pp.~125--126]{CincOprocha2024} and compare \cite{BargeRoe1990}.

\begin{lemma}[Local BBM block]\label{lem:2d-bbm}
For every $a>0$ and $\eps>0$ there is a near-homeomorphism
$F_{a,\eps}$ of $[-a,a]\times\mathbb{S}^1$ such that:
\begin{enumerate}[label=\textup{(\roman*)}]
\item $F_{a,\eps}(0,x)=(0,T_\eps(x))$;
\item $F_{a,\eps}(\pm a,x)=(\pm a,R_\alpha(x))$;
\item there are $0<b<a$ and $\eta(\eps)>0$ such that
$F_{a,\eps}([-b,b]\times\mathbb{S}^1)=\{0\}\times\mathbb{S}^1$, the radial
coordinate is not increased in absolute value by $F_{a,\eps}$, and if
$(z_i)$ is an $\eta(\eps)$-pseudo-orbit of $F_{a,\eps}$ with
$z_{i_0}\in[-b,b]\times\mathbb{S}^1$ for some $i_0$, then
$z_i\in[-b,b]\times\mathbb{S}^1$ for every $i\ge i_0$ (forward trapping);
\item there is $\omega(\eps)\to0$ as $\eps\to0$ such that
\[
 \sup_{(t,x)}d_{\mathbb{S}^1}
 \bigl(\pi_2(F_{a,\eps}(t,x)),R_\alpha(x)\bigr)\le\omega(\eps).
\]
\end{enumerate}
\end{lemma}

\begin{proof}
Fix $0<b_0<b<a$ with $b_0\le\delta$, where $\delta=\delta(\eps)$ is the
constant from the construction of $T_\eps$. We will obtain $F_{a,\eps}$ as
the composition
\[
F_{a,\eps}:=P\circ S\circ\overline q ,
\]
where $P(t,x)=(t,R_\alpha(x))$ is the rigid rotation,
$S(t,x)=(s(t),x)$ is the radial collapse determined by $s(t)=0$ for
$|t|\le b$ and
$$s(t)=\operatorname{sgn}(t)\frac{a(|t|-b)}{a-b}$$ for $b\le|t|\le a$, and
$\overline q$ is a homeomorphism supported in $\{|t|<b_0\}$ constructed
below.
The map $s$ is a continuous nondecreasing surjection of $[-a,a]$ fixing
$\pm a$, and $s=\lim_k s_k$ uniformly, where $s_k$ is the homeomorphism
with $s_k(t)=\frac{t}{k}$ on $[-b,b]$, $s_k(\pm a)=\pm a$, affine in between.
Hence $S=\lim_k(s_k\times\id)$ is a uniform limit of homeomorphisms, and so
is $F_{a,\eps}=P\circ S\circ\overline q$, being a composition of a
near-homeomorphism with two homeomorphisms.

A lift of $q$ is $\widetilde q(x)=x+u(x)$ with
$u(x)=3\delta\tau(Nx)$, so $\|u\|_\infty=3\delta$ and $q$ has degree one.
Put $c:=\frac{b_0}{8\delta}$ and $v(x):=c\,u(x)$. Then
$\|v\|_\infty\le\frac{3b_0}{8}<\frac{b_0}{2}$. We claim that
\[
 e\colon\mathbb S^1\longrightarrow\left(-\frac{b_0}{2},\frac{b_0}{2}\right)\times\mathbb S^1,
 \qquad e(x):=(v(x),q(x)),
\]
is an embedding. Indeed, suppose $e(x)=e(x')$ and choose lifts of $x,x'$ to $\mathbb R$, still
denoting them as $x,x'$. After the lifting, equality of the first coordinates under $e$ still gives
$u(x)=u(x')$, while equality of the second coordinates gives
$x+u(x)=x'+u(x')+k$ for some $k\in\mathbb Z$. Hence $x-x'=k$, so $x$ and
$x'$ represent the same point of $\mathbb S^1$. Since the angular
projection of $e$ has degree one, $e(\mathbb S^1)$ is an essential
piecewise linear circle. Moreover, with $\iota(x):=(0,x)$,
\[
 \sup_x d\bigl(e(x),\iota(x)\bigr)
 \le \|v\|_\infty+\|u\|_\infty<4\delta .
\]

Let $a:=\frac{3b_0}{8}$ and for $t\in [-b_0,b_0]$ define the piecewise linear functions
\[
\varphi(t):=\max\Bigl(0,\,1-\tfrac{|t|}{b_0}\Bigr),
\qquad
\sigma(t):=
\begin{cases}
	\frac{t}{c}, & |t|\le a,\\[2pt]
	\operatorname{sgn}(t)\,3\delta\,\dfrac{b_0-|t|}{b_0-a}, & a\le|t|\le b_0,\\[2pt]
%	0, & |t|\ge b_0,
\end{cases}
\]
so that $\sigma$ is continuous (since $\frac{a}{c}=3\delta$) and $\|\sigma\|_\infty=3\delta$.
Define 
$\overline q:=S\circ G$, where
\[
G(t,x):=\bigl(t+\varphi(t)v(x),\,x\bigr),
\qquad
H(t,y):=\bigl(t,\,y+\sigma(t)\bigr),
\]
that is,
\[
\overline q(t,x)=\Bigl(t+\varphi(t)v(x),\;
x+\sigma\bigl(t+\varphi(t)v(x)\bigr)\Bigr).
\]
We extend $\overline q(t,x)=(t,x)$ for $|t|\geq b_0$.

The map $G$ preserves each fiber
	$[-b_0,b_0]\times\{x\}$, and since 
	$\varphi$ is $\frac{1}{b_0}$-Lipschitz and
	$\|v\|_\infty\le \frac{3b_0}{8}$, this map has slopes in
	$[\tfrac58,\tfrac{11}8]$, so it is an increasing homeomorphism of
	$[-b_0,b_0]$ fixing $\pm b_0$. Hence $G$ is a homeomorphism fixing the
	boundary, and it is piecewise linear because $\varphi$ and $v$ are. The map $H$
	rotates each fiber $\{t\}\times\mathbb S^1$ rigidly by $\sigma(t)$ and
	is the identity for $|t|=b_0$, so it is a piecewise linear homeomorphism as well.
	Thus $\overline q=H\circ G$ is a homeomorphism, equal to the
	identity on $\{|t|=b_0\}$. Observe that
	\[
	\overline q(0,x)=\bigl(v(x),\,x+u(x)\bigr)=\bigl(v(x),q(x)\bigr)=e(x)
	\]
	and be immediate estimate of the maximal displacement introduced by $G$ and $H$ we have
	\[
	d\bigl(\overline q(t,x),(t,x)\bigr)\le 3\delta+3\delta=6\delta
	\qquad\text{for all }(t,x).
	\]

It remains to verify that the constructed map satisfies the claims of the theorem.
(i) Since $|v(x)|<\frac{b_0}{2}<b$ we get $s(v(x))=0$, so
$F_{a,\eps}(0,x)=P(S(v(x),q(x)))=(0,R_\alpha(q(x)))=(0,T_\eps(x))$.
(ii) For $t=\pm a$ we have $\overline q(\pm a, \cdot)=\id$ and $s(\pm a)=\pm a$, so
$F_{a,\eps}(\pm a,x)=(\pm a,R_\alpha(x))$.
(iv) Neither $S$ nor $P$ changes the angular coordinate except by the
rigid rotation, hence
\[
d_{\mathbb{S}^1}
\bigl(\pi_2(F_{a,\eps}(t,x)),R_\alpha(x)\bigr)
\le d_{C^0}(\overline q,\id)\le6\delta
=:\omega(\eps)<\eps.
\]
(iii) Since $\overline q$ is supported in $\{|t|<b_0\}$, for $|t|\le b$
we have $\pi_1(\overline q(t,x))\in[-b,b]$ and therefore
$F_{a,\eps}(\{|t|\le b\})=\{0\}\times\mathbb{S}^1$. For $|t|\ge b_0$ the
radial coordinate of $F_{a,\eps}(t,x)$ is $s(t)$ and
$|s(t)|\le|t|$, so the radial coordinate is never increased. Set
$\eta(\eps):=\frac{b}{2}$. If $(z_i)$ is an $\eta(\eps)$-pseudo-orbit with
$z_{i_0}\in\{|t|\le b\}$, then
$F_{a,\eps}(z_{i_0})\in\{0\}\times\mathbb{S}^1$, so the radial coordinate
of $z_{i_0+1}$ is at most $\eta(\eps)\le \frac{b}{2}<b$, and inductively
$z_i\in\{|t|\le \frac{b}{2}\}\subset\{|t|\le b\}$ for all $i>i_0$. This proves
the forward trapping and completes the proof.
\end{proof}

Choose numbers $\gamma_n\searrow0$ and pairwise disjoint closed intervals
$I_{\pm n}$ centered at $\pm\gamma_n$, with lengths tending to zero.  Put
$$\mathcal A_{\pm n}=I_{\pm n}\times\mathbb{S}^1, C_{\pm n}=\{\pm\gamma_n\}\times\mathbb{S}^1 \text{ and }
\mathcal A_0=\{0\}\times\mathbb{S}^1.$$  Choose $\eps_n\searrow0$ and write
$T_n=T_{\eps_n}$.  On each $\mathcal A_{\pm n}$ place a translated copy of
the local block from Lemma~\ref{lem:2d-bbm}.  On every complementary
annulus use a map
\[
 (t,x)\longmapsto(f_J(t),R_\alpha(x)),
\]
where $f_J$ fixes the endpoints of $J$, has no interior fixed point, and is
strictly monotone toward one endpoint.  Finally set
$S(0,x)=(0,R_\alpha(x))$.

\begin{figure}[H]
\centering
\begin{tikzpicture}[
  x=1cm,y=1cm,
  every node/.style={font=\small},
  block/.style={draw=black,fill=gray!10,line width=.7pt},
  central/.style={draw=black,line width=1.2pt},
  drift/.style={draw=black,->,line width=.7pt}
]
  \draw[black,line width=.8pt] (-5.7,-1.05) rectangle (5.7,1.05);
  \node[rotate=90] at (-5.95,0) {$\mathbb S^1$};
  \node at (0,-1.28) {$0$};
  \node at (-5.65,-1.28) {$-1$};
  \node at (5.65,-1.28) {$1$};
%  \draw[teal,->] (-5.65,-1.28) -- (5.35,-1.28);

  \filldraw[block] (-4.75,-1.05) rectangle (-3.95,1.05);
  \filldraw[block] (-3.25,-1.05) rectangle (-2.65,1.05);
  \filldraw[block] (-1.45,-1.05) rectangle (-1.05,1.05);
  \filldraw[block] (1.05,-1.05) rectangle (1.45,1.05);
  \filldraw[block] (2.65,-1.05) rectangle (3.25,1.05);
  \filldraw[block] (3.95,-1.05) rectangle (4.75,1.05);

  \draw[central] (-4.35,-1.05) -- (-4.35,1.05);
  \draw[central] (-2.95,-1.05) -- (-2.95,1.05);
  \draw[central] (-1.25,-1.05) -- (-1.25,1.05);
  \draw[central] (0,-1.05) -- (0,1.05);
  \draw[central] (1.25,-1.05) -- (1.25,1.05);
  \draw[central] (2.95,-1.05) -- (2.95,1.05);
  \draw[central] (4.35,-1.05) -- (4.35,1.05);

  \node at (-4.35,1.22) {$C_{-1}$};
  \node at (-2.95,1.22) {$C_{-2}$};
  \node at (-1.25,1.22) {$C_{-n}$};
  \node[fill=white,inner sep=1.5pt] at (0,.68) {$\mathcal A_0$};
  \node at (1.25,1.22) {$C_n$};
  \node at (2.95,1.22) {$C_2$};
  \node at (4.35,1.22) {$C_1$};

  \node at (-4.35,-1.22) {$\mathcal A_{-1}$};
  \node at (-2.95,-1.22) {$\mathcal A_{-2}$};
  \node at (-1.25,-1.22) {$\mathcal A_{-n}$};
  \node at (1.25,-1.22) {$\mathcal A_n$};
  \node at (2.95,-1.22) {$\mathcal A_2$};
  \node at (4.35,-1.22) {$\mathcal A_1$};

  \node at (-1.9,0) {$\cdots$};
  \node at (-.4,0) {$\cdots$};
  \node at (.45,0) {$\cdots$};
  \node at (1.9,0) {$\cdots$};

  \draw[drift] (-3.82,.05) -- (-3.38,.05);
  \draw[drift] (-2.52,.05) -- (-2.18,.05);
  \draw[drift] (-.98,.05) -- (-.64,.05);
  \draw[drift] (.64,.05) -- (.98,.05);
  \draw[drift] (2.18,.05) -- (2.52,.05);
  \draw[drift] (3.38,.05) -- (3.82,.05);
\end{tikzpicture}
\caption{The radial arrangement in
$\mathbb A=[-1,1]\times\mathbb S^1$.  The local blocks
$\mathcal A_{\pm n}$ shrink toward the central rotation circle
$\mathcal A_0$; their central circles $C_{\pm n}$ carry the maps $T_n$.
The arrows indicate the monotone radial motion on complementary
annuli.}
\label{fig:annular-arrangement}

\end{figure}

The definitions agree on the boundary circles.  Since the widths of the
inserted annuli and $\omega(\eps_n)$ tend to zero, $S$ is continuous at
$\mathcal A_0$.  Approximating the local blocks by boundary-fixing
homeomorphisms, with errors tending to zero, shows that $S$ is a
near-homeomorphism of $\mathbb A$. Indeed, by Step~1 of the proof of
Lemma~\ref{lem:2d-bbm} each block equals
$P\circ S\circ\overline q=\lim_k P\circ(s_k\times\id)\circ\overline q$,
where the approximating maps are homeomorphisms fixing the block boundary
circles; replacing each block by its $k$-th approximation, with $k=k(n)$
chosen so large that the errors are summable over $n$, and keeping the
interpolating skew products unchanged, produces global homeomorphisms of
$\mathbb A$ converging uniformly to $S$.

\begin{theorem}\label{thm:2d-central-tracing}
For every $\eps>0$ and every $N_0\in\bbN$ there is $\delta>0$ such that
every two-sided $\delta$-pseudo-orbit of $S$ meeting $\mathcal A_0$ is
$\eps$-traced by a full orbit contained in some $C_{\pm m}$ with
$m\ge N_0$.
\end{theorem}

\begin{proof}
For $k\ge1$ write $a_k$, $b_k$, $\eta_k:=\eta(\eps_k)$ and
$\delta_k:=\delta(\eps_k)$ for the parameters of the block placed on
$\mathcal A_{\pm k}$ (so $a_k$ is half the width of $I_{\pm k}$, and
$\delta_k$ is the constant from the construction of
$T_k=T_{\eps_k}$, for which Proposition~\ref{prop:2d-circle-tracing}
traces every $\delta_k$-pseudo-orbit of $R_\alpha$ with accuracy
$\delta_k$). Denote by
$B_{\pm k}\subset\mathcal A_{\pm k}$ the trapping strip of radial width
$2b_k$ around $C_{\pm k}$ provided by
Lemma~\ref{lem:2d-bbm}(iii).

Choose $m\ge N_0$ so large that
$$2\gamma_m+\delta_m<\frac{\eps}{2}.$$ Next choose $n>m$ so large that (a) the
closed region $\Sigma_n$ between the outer boundaries of
$\mathcal A_{-n}$ and $\mathcal A_n$ lies within radial distance
$\frac{\eps}{4}$ of $\mathcal A_0$, and (b) the angular displacement of $S$
relative to $R_\alpha$ on $\Sigma_n$ is at most $\frac{\delta_m}{2}$; (b) is
possible because on $\Sigma_n$ the map $S$ is either an interpolating
skew product or a block of level $k$ with $|k|\ge n$, whose angular
displacement is at most
$$\omega(\eps_k)\le\sup_{k\ge n}\omega(\eps_k).$$.
In particular, $\omega(\eps_k)\to 0$.

Set
\[
\delta:=\min\left\{\frac{\delta_m}{2},\eta_n,b_n\right\}
\]
and let $(z_i)_{i\in\Z}$ be a $\delta$-pseudo-orbit of $S$ with
$z_{i_0}\in\mathcal A_0$. Write $z_i=(t_i,x_i)$.

We claim that $|t_i|\le\gamma_n+a_n$ for all $i$, i.e.\ the pseudo-orbit
never passes beyond the level-$n$ blocks. First note that if
$z_i\in B_n$ for some $i$, then by Lemma~\ref{lem:2d-bbm}(iii) (applied
in the block, whose trapping threshold is $\eta_n\ge\delta$) we get
$z_j\in B_n$ for all $j\ge i$; the same holds for $B_{-n}$. Since
$\mathcal A_0\cap B_{\pm n}=\emptyset$, no index $i<i_0$ can lie in
$B_{\pm n}$, for otherwise $z_{i_0}\in B_{\pm n}$. Moreover, a
transition of the radial coordinate from the inner side of a level-$n$
block to its outer side (or vice versa) forces some $z_i$ to lie in
$B_{\pm n}$: within the block the map does not increase the radial
distance from $C_{\pm n}$ by more than the pseudo-orbit error
$\delta\le b_n$, and a single jump cannot cross the strip $B_{\pm n}$ of
width $2b_n$. Consequently, for $i\le i_0$ the pseudo-orbit stays on the
inner side of both strips $B_{\pm n}$; for $i\ge i_0$ it either stays
there as well, or enters one of the strips and remains in it forever. In
all cases $$|t_i|\le\gamma_n+a_n<\frac{\eps}{4},$$ for every $i$, proving the claim.

For every $i$, one gets
\[
\begin{split}
d_{\mathbb{S}^1}(x_{i+1},R_\alpha(x_i))
&\le d(z_{i+1},S(z_i))\\
&\quad+
d_{\mathbb{S}^1}\bigl(\pi_2(S(z_i)),R_\alpha(x_i)\bigr)\\
&\le\delta+\frac{\delta_m}{2}\le\delta_m,
\end{split}
\]
using (b) and the fact that $z_i\in\Sigma_n$. Hence $(x_i)$ is a
$\delta_m$-pseudo-orbit of $R_\alpha$, and
Proposition~\ref{prop:2d-circle-tracing} yields a full $T_m$-orbit
$(y_i)_{i\in\Z}$ satisfying
$d_{\mathbb{S}^1}(y_i,x_i)\le\delta_m$, for all $i$.

Put $w_i:=(\gamma_m,y_i)\in C_m$. By
Lemma~\ref{lem:2d-bbm}(i), applied to the block on $\mathcal A_m$, the
circle $C_m$ is $S$-invariant and $S|_{C_m}$ acts as $T_m$, so
$(w_i)_{i\in\Z}$ is a full $S$-orbit contained in $C_m$. Finally, in the
maximum metric,
\[
d(w_i,z_i)\le
\max\{|\gamma_m-t_i|,d_{\mathbb{S}^1}(y_i,x_i)\}
\le\max\left\{\gamma_m+\frac{\eps}{4},\delta_m\right\}<\eps,
\]
since $\gamma_m<\frac{\eps}{4}$ by the choice of $m$. Thus $(z_i)$ is
$\eps$-traced by a full orbit in $C_m$ with $m\ge N_0$ (symmetrically one
may use $C_{-m}$), completing the proof.
\end{proof}

Inside each local block, the central circle attracts all points away from
the boundary, while on each complementary annulus the radial coordinate is
strictly monotone.  It follows that the chain components of $S$ are
$\mathcal A_0$, the circles $C_{\pm n}$, the boundary circles of the local
blocks, and the two boundary components of $\mathbb A$.

Let $X=\varprojlim(\mathbb A,S)$ and let $H_0$ be the shift homeomorphism.
By Brown's approximation theorem \cite{Brown1960}, $X$ is an annulus.  The
inverse limits of the two boundary circles and of $\mathcal A_0$ are
circles carrying $R_\alpha$; denote the latter by $A_0$.  It is an
essential circle.  Put
\[
 K_{\pm n}=\varprojlim(C_{\pm n},T_n).
\]
Each $H_0|_{K_{\pm n}}$ is transitive, has finite topological entropy, and
supports a non-atomic ergodic measure of full support.  Moreover, the sets
$K_{\pm n}$ and the rotation circles are precisely the chain components of
$H_0$, and every $K_{\pm n}$ is isolated.

\begin{lemma}\label{lem:2d-extension-tracing}
For every $\eps>0$ and $N\ge1$ there is $\delta>0$ such that every
two-sided $\delta$-pseudo-orbit of $H_0$ meeting $A_0$ is $\eps$-traced by
an orbit contained in some $K_{\pm m}$ with $m\ge N$.
\end{lemma}

\begin{proof}
We regard
\[
X=\{\xi=(\xi_0,\xi_1,\ldots):\xi_i\in\mathbb A,\
\xi_i=S(\xi_{i+1})\}
\]
with the metric
$d(\xi,\zeta)=\sum_{i\ge0}2^{-i}d(\xi_i,\zeta_i)$, and
$H_0(\xi)=(S(\xi_0),\xi_0,\xi_1,\ldots)$, so that
$(H_0(\xi))_i=S(\xi_i)$ for every $i\ge0$.

Fix $\eps>0$ and $N\ge1$. Choose a depth $k\ge0$ with
\[
\sum_{i>k}2^{-i}\diam(\mathbb A)<\frac{\eps}{4}.
\]
Let $\eps_1>0$ be so small that $d(a,b)\le\eps_1$ implies
$d(S^r(a),S^r(b))\le\frac{\eps}{8}$ for all $0\le r\le k$ (uniform continuity of
$S,\ldots,S^k$). Let $\delta^*>0$ be the constant given by
Theorem~\ref{thm:2d-central-tracing} for the accuracy $\eps_1$ and the
index bound $N$, and set $\delta:=2^{-k}\delta^*$.

Let $(\xi^j)_{j\in\Z}$ be a $\delta$-pseudo-orbit of $H_0$ with
$\xi^{j_0}\in A_0$. Put $a_j:=\xi^j_k\in\mathbb A$. Since
$$(H_0(\xi^j))_k=S(\xi^j_k) \text{ and }
d(H_0(\xi^j),\xi^{j+1})\le\delta,$$ the definition of the metric gives
\[
d(S(a_j),a_{j+1})\le2^k\delta=\delta^*,
\]
so $(a_j)_{j\in\Z}$ is a $\delta^*$-pseudo-orbit of $S$. Moreover
$a_{j_0}=\xi^{j_0}_k\in\mathcal A_0$, because
$A_0=\varprojlim(\mathcal A_0,S|_{\mathcal A_0})$ has all coordinates in
$\mathcal A_0$. By Theorem~\ref{thm:2d-central-tracing} there are
$m\ge N$ and a full $S$-orbit $(w_j)_{j\in\Z}$ contained in
$C_{\pm m}$ with $d(w_j,a_j)\le\eps_1$ for all $j\in\Z$.

Define, for $j\in\Z$, the point $\eta^j\in X$ by
$\eta^j_i:=w_{j+k-i}$ for $i\ge0$. The bonding condition
$S(\eta^j_{i+1})=w_{j+k-i}=\eta^j_i$ holds because $(w_j)$ is a full
$S$-orbit; moreover $\eta^{j+1}=H_0(\eta^j)$, so
$(\eta^j)_{j\in\Z}$ is a full $H_0$-orbit, and all its coordinates lie in
the $S$-invariant circle $C_{\pm m}$, whence
$\eta^j\in\varprojlim(C_{\pm m},T_m)=K_{\pm m}$.

Finally, we estimate $d(\eta^j,\xi^j)$. For each $0\le i\le k,$ we have
$$\xi^j_i=S^{k-i}(\xi^j_k)=S^{k-i}(a_j) \text{ and }
\eta^j_i=w_{j+k-i}=S^{k-i}(w_j).$$ Since
$d(w_j,a_j)\le\eps_1$, the choice of $\eps_1$ gives
$d(\eta^j_i,\xi^j_i)\le\frac{\eps}{8}$. Hence
\[
d(\eta^j,\xi^j)\le
\sum_{i=0}^{k}2^{-i}\frac{\eps}{8}
+\sum_{i>k}2^{-i}\diam(\mathbb A)
\le\frac{\eps}{4}+\frac{\eps}{4}<\eps .
\]
Thus $(\xi^j)$ is $\eps$-traced by the $H_0$-orbit $(\eta^j)$ contained
in $K_{\pm m}$ with $m\ge N$.
\end{proof}

\subsubsection{Ensuring lack of shadowing on the isolated components}

We need a relative version of the control in
Proposition~\ref{prop:limit}.
Let  $K=\Supp\mu$ be the support of measure used for extension in Theorem~\ref{thm:BCLR}
and let $T$, $\Pi$ be provided by its application. Let $W$ be an open
neighborhood of $K$ and let $V$ be an open set with
$K\subset V\subset \overline{V}\subset W$ and $V\cup S^{-1}(V)\subset W$.
Such $V$ exists
for every neighborhood $W$ of the invariant set $K$, by continuity of
$S$ and invariance of $K$. Note that smallness of
$\dH(T,S)+\dC(\Pi,\id)$ alone does not localize the perturbation; what we
use instead is the structure of the construction. In the notation of
Proposition~\ref{prop:limit}, the conjugacies produced in \cite{BCLR} are
supported in finite families of topological rectangles surrounding orbit
segments of points of $A\subset K$, and the diameters of these rectangles
are free parameters; hence all rectangles may be chosen with closures
contained in $V$. With this choice every conjugacy, and in the limit the
factor map, is the identity off $V$, so $\Pi=\id$ on $M\setminus V$.
Moreover, if $x\notin V\cup S^{-1}(V)$, then $\Psi_n(x)=x$ and
$S(x)\notin V$, so $T_n(x)=\Psi_n^{-1}(S(x))=S(x)$ for every $n$. Passing
to the limit, $T=S$ on $M\setminus(V\cup S^{-1}(V))\supset M\setminus W$.

\begin{proposition}
\label{prop:2d-BCLR-control}
For every $\eta>0$ one can require in Theorem~\ref{thm:BCLR} that
\[
 \dH(T,S)+\dC(\Pi,\id)<\eta.
\]
\end{proposition}

\begin{proof}
Recall the shape of the construction in \cite[\S\S2--5]{BCLR}, in the
summarized form of Proposition~\ref{prop:limit}: one produces
homeomorphisms
$\Psi_n=h_n\circ h_{n-1}\circ\cdots\circ h_1$ of $M$ and sets $T_n$, $T$ and $\Pi$ such that:
\begin{itemize}
\item $T_n:=\Psi_n^{-1}\circ S\circ\Psi_n$,
\item $T_n\xrightarrow{unif} T$,
\item $\Psi_n\xrightarrow{unif} \Pi$.
\end{itemize}
Each $h_n$ is supported in a finite disjoint family
$\mathcal E_n$ of topological rectangles chosen around finite segments
of orbits of points of $A\subset K$. The only requirements imposed on
$\mathcal E_n$ in \cite{BCLR} are combinatorial (disjointness and a
Markov intersection pattern with $\mathcal E_{n-1}$ and with the inserted
Cantor fibres), together with the requirement that the diameters of the
rectangles of $\mathcal E_n$ tend to $0$ as $n\to\infty$. 
%Since the orbit
%segments in question lie in the compact invariant set $K$ and the
%diameters are free parameters of the construction, all rectangles of all
%families $\mathcal E_n$ may be chosen with closures contained in the open
%neighbourhood $V$ of $K$. With this choice:
%
%(1) $h_n=\id$ on $M\setminus V$ for every $n$, hence $\Psi_n=\id$ and, in
%the limit, $\Pi=\id$ on $M\setminus V$.
%
%(2) If $x\notin V\cup S^{-1}(V)$, then $\Psi_n(x)=x$ and
%$S(x)\notin V$, so $\Psi_n^{-1}(S(x))=S(x)$; therefore $T_n(x)=S(x)$ for
%every $n$, and passing to the limit,
%$T=S$ on
%$M\setminus(V\cup S^{-1}(V))\supset M\setminus W$.
%
%(3) 

For the quantitative estimate, prescribe first a positive summable
sequence $(\rho_n)$. 
At the $n$-th stage only finitely many %orbit pieces
%in the compact set $K\subset V$ are involved. Hence the rectangles can be
rectangles are involved and they can be
chosen recursively %with closures in $V$ and
with diameters below $\rho_n$, keeping their closures in $V$.
Shrinking them and, when necessary, passing to a finer subdivision
preserves the finite disjointness and incidence requirements of that
stage. This is precisely the freedom in the choice of rectangles used in
\cite[Proposition~3.1 and \S1.5.1]{BCLR}.

Since $h_n$ moves points only within rectangles of $\mathcal E_n$, we get
$\dC(\Psi_n,\Psi_{n-1})\le\rho_n$ and hence
$r_n:=\dC(\Psi_n,\id)\le\sum_{j\le n}\rho_j$. Notice also that
$\dC(\Psi_n^{-1},\id)=r_n$. From
$T_n=\Psi_n^{-1}\circ S\circ \Psi_n$ and
$T_n^{-1}=\Psi_n^{-1}\circ S^{-1}\circ \Psi_n$ we obtain
\[
\begin{split}
\dC(T_n,S)&\le r_n+\omega_S(r_n),\\
\dC(T_n^{-1},S^{-1})&\le r_n+\omega_{S^{-1}}(r_n),
\end{split}
\]
where $\omega_S$ and $\omega_{S^{-1}}$ are moduli of uniform continuity. 
Passing to the limits $T_n\to T$ in $\dH$ and $\Psi_n\to\Pi$ uniformly
gives the same bounds with $$r:=\sum_n\rho_n \text{ and }\dC(\Pi,\id)\le r.$$ Choosing $r$ sufficiently small therefore yields
$$\dH(T,S)+\dC(\Pi,\id)<\eta.$$
\end{proof}

We apply this proposition independently inside pairwise disjoint invariant
annuli surrounding the sets $K_{\pm n}$.  Let
$b_n=\htop(H_0|_{K_{\pm n}})$.  By the Jewett--Krieger theorem
\cites{Jewett,Krieger}, choose a strictly ergodic Cantor system
$(C_n,r_n)$ measure-theoretically isomorphic to a Bernoulli system and with
unique measure entropy $H_n>b_n$.  Apply Theorem~\ref{thm:BCLR} over a set
of generic points of a full-support ergodic measure on $K_{\pm n}$, using
the product insertion with $(C_n,r_n)$ and the relative control above.
Write $D_{\pm n}$ for the preimage of $K_{\pm n}$.

The universal isomorphism shows that every ergodic measure on
$D_{\pm n}$ either has entropy at most $b_n$ or at least $H_n$, whereas a
product measure has entropy at least $H_n$.  Hence the ergodic entropy
spectrum has a nonempty gap below $\htop(H|_{D_{\pm n}})$.  By
Corollary~\ref{cor:nogap}, the transitive system $H|_{D_{\pm n}}$ does not
have the shadowing property.

For each $n\ge1$ choose open sets
$V_{\pm n}\subset \overline{V_{\pm n}}\subset W_{\pm n}\subset N_{\pm n}$ around $K_{\pm n}$ such
that:
\begin{enumerate}
    \item $V_{\pm n}\cup H_0^{-1}(V_{\pm n})\subset W_{\pm n}$. 
    \item The sets $N_{\pm n}\cup H_0(N_{\pm n})$ are pairwise disjoint over all indices.
     \item  $N_{\pm n}$ is disjoint from $A_0$, from the rotation circles, and from the trapping regions of all blocks of level different from $n$.
\end{enumerate}
This is possible because the sets $K_{\pm n}$ are pairwise disjoint
isolated invariant sets accumulating only on $A_0$. Let $\eta_n>0$ be a
summable sequence, with $\eta_n$ smaller than the trapping thresholds of
the level-$n$ and neighbouring blocks, to be further restricted below.

Apply Proposition~\ref{prop:2d-BCLR-control} to $(X,H_0)$ at $K:=K_n$
with the pair $V_n\subset W_n$ and size $\eta_n$, obtaining a
homeomorphism $T^{(n)}$ of $X$ with $T^{(n)}=H_0$ off $W_n$, a factor map
$\Pi_n$ with $\Pi_n=\id$ off $V_n$,
$\Pi_n\circ T^{(n)}=H_0\circ\Pi_n$, and
$\dH(T^{(n)},H_0)+\dC(\Pi_n,\id)<\eta_n$; proceed symmetrically for
$K_{-n}$. Define
\[
H(x):=
\begin{cases}
T^{(\pm n)}(x),&x\in W_{\pm n},\\
H_0(x),&x\notin\displaystyle\bigcup_k W_{\pm k},
\end{cases}
\qquad
\Pi(x):=
\begin{cases}
\Pi_{\pm n}(x),&x\in V_{\pm n},\\
x,&x\notin\displaystyle\bigcup_k V_{\pm k}.
\end{cases}
\]
The definitions are consistent
($T^{(\pm n)}=H_0$ and $\Pi_{\pm n}=\id$ near
$\partial W_{\pm n}$ and $\partial V_{\pm n}$), and both maps are
continuous: at points of $A_0$, which is the only place where infinitely
many supports accumulate, continuity follows from $\eta_n\to0$. To see
that $H$ is a homeomorphism, note that since $T^{(\pm n)}$ is a
homeomorphism of $X$ equal to $H_0$ off $W_{\pm n}$, we have
$T^{(\pm n)}(W_{\pm n})=H_0(W_{\pm n})$; hence $H$ maps the members of
the partition
\[
\{W_{\pm k}\}_k\cup
\left\{X\setminus\bigcup W_{\pm k}\right\}
\]
onto the members of the partition
\[
\{H_0(W_{\pm k})\}_k\cup
\left\{X\setminus\bigcup H_0(W_{\pm k})\right\},
\]
bijectively on each member. Thus $H$ is a continuous bijection of a
compact space, hence a homeomorphism. The identity
$\Pi\circ H=H_0\circ\Pi$ is verified piecewise: on $W_{\pm n}$ it is the
semiconjugacy relation for $T^{(\pm n)}$ (here we use that
$H(W_{\pm n})=H_0(W_{\pm n})\subset
N_{\pm n}\cup H_0(N_{\pm n})$ meets no $V_{\pm k}$ with $k\neq n$, by
the disjointness required above), and off $\bigcup W_{\pm k}$ both sides
equal $H_0$ at points whose image avoids every $V_{\pm k}$.

The chain components of $H$ are exactly the rotation circles and the sets
$D_{\pm n}:=\Pi_{\pm n}^{-1}(K_{\pm n})$. Indeed, outside
$\bigcup N_{\pm k}$ the maps $H$ and $H_0$ coincide, and since every
$\eta_k$ is smaller than the relevant trapping thresholds, the radial
trapping strips of all blocks remain forward trapping for
$H$-pseudo-orbits with sufficiently small error; hence chains cannot
cross between different levels any more than they could for $H_0$, so no
two of the listed sets merge and no new chain recurrence appears in the
radially strictly monotone regions. Within $N_{\pm n}$, the argument
proving \ref{con:P2} in Theorem~\ref{thm:abstract}, applied verbatim to
the pair $(T^{(\pm n)},\Pi_{\pm n})$, shows that $D_{\pm n}$ is a
transitive chain component of $H$.

Finally, Lemma~\ref{lem:2d-extension-tracing} survives the perturbation in
the following form: for every $\eps>0$ and $N\ge1$ there is $\delta>0$
such that every $\delta$-pseudo-orbit of $H$ meeting $A_0$ is
$\eps$-traced by an $H$-orbit contained in some $D_{\pm m}$ with
$m\ge N$. To see this, fix $\eps,N$ and choose $n_0$ so large that
$\eta_m<\frac{\eps}{4}$ for all $m\ge n_0$. The persistent radial barriers used
above also give a finite-support avoidance property: there is
$\delta_1>0$ such that every $\delta_1$-pseudo-orbit of $H$ meeting
$A_0$ avoids all $V_{\pm k}$ with $k<n_0$. Indeed, choose a pair of
trapping strips between $A_0$ and the finitely many sets in question and
take $\delta_1$ smaller than their persistent trapping thresholds. The
same forward-and-backward argument as in the radial-confinement part of
Theorem~\ref{thm:2d-central-tracing} applies.

Let $\delta^*$ be the constant of
Lemma~\ref{lem:2d-extension-tracing} for the accuracy $\frac{\eps}{4}$ and index
bound $$N':=\max\{N,n_0\}.$$ Since $\Pi$ is uniformly continuous,
$\Pi=\id$ on a neighbourhood of $A_0$, and
$\Pi\circ H=H_0\circ\Pi$, for sufficiently small $\delta$ the
$\Pi$-image of any $\delta$-pseudo-orbit $(\xi^j)$ of $H$ meeting $A_0$
is a $\delta^*$-pseudo-orbit of $H_0$ meeting $A_0$, and $\delta$ may
also be chosen below $\delta_1$. By
Lemma~\ref{lem:2d-extension-tracing} it is $(\frac{\eps}{4})$-traced by an
$H_0$-orbit $(\eta^j)$ contained in some $K_{\pm m}$, $m\ge N'$. Pick
any $\zeta\in\Pi_{\pm m}^{-1}(\eta^0)\subset D_{\pm m}$. Then
$$\Pi(H^j(\zeta))=H_0^j(\Pi(\zeta))=\eta^j,$$ for all $j\in\Z$, and since
$\dC(\Pi_{\pm m},\id)<\eta_m$ we get
$$d(H^j(\zeta),\eta^j)<\eta_m<\frac{\eps}{4}.$$ Moreover, the definition of $\Pi$ give
$d(\Pi(\xi^j),\xi^j)<\frac{\eps}{4}$ for all $j$. Hence
\[
\begin{split}
d(H^j(\zeta),\xi^j)
&\le d(H^j(\zeta),\eta^j)+d(\eta^j,\Pi(\xi^j))\\
&\quad+d(\Pi(\xi^j),\xi^j)<3\frac{\eps}{4}<\eps.
\end{split}
\]
Thus the $H$-orbit of $\zeta\in D_{\pm m}$ $\eps$-traces $(\xi^j)$, as
required; in particular
$A_0\subseteq\Sh(H)$.

\begin{theorem}\label{thm:2d-annulus}
There exist a homeomorphism $H$ of the annulus and an essential circle
$A\subset\interior\mathbb A$ such that:
\begin{enumerate}
\item $A$ is a chain component, $H|_A$ is conjugate to $R_\alpha$, and
$A\subseteq\Sh(H)$;
\item the restrictions of $H$ to both boundary components are conjugate to
$R_\alpha$;
\item neither $(\mathbb A,H)$ nor the restriction of $H$ to any chain
component has the shadowing property;
\item $\htop(H)<\infty$.
\end{enumerate}
\end{theorem}

\begin{proof}
Take $A=A_0$.  The tracing conclusion above implies
$A\subseteq\Sh(H)$.  Every chain component is either a rotation circle or
one of the sets $D_{\pm n}$.  Irrational rotations do not have shadowing,
and the latter systems do not have shadowing by the entropy-gap argument.
Thus the restriction of $H$ to no chain component has the shadowing property.

Each $D_{\pm n}$ is isolated.  If $H$ had shadowing, then its restriction
to an isolated invariant set $D_{\pm n}$ would also have shadowing, a
contradiction.  Finally, the entropies of the inserted Cantor systems may be
chosen uniformly bounded, and the perturbations occur in disjoint invariant
regions; hence $\htop(H)<\infty$.
\end{proof}

\begin{corollary}\label{cor:2d-square}
Let $M=[0,1]^2$.  There exist $T\in\Homeo(M)$ and an uncountable chain
component $D\subset\interior M$ such that:
\begin{enumerate}
\item $T|_D$ is transitive and $D\subseteq\Sh(T)$;
\item neither $(M,T)$ nor $(D,T|_D)$ has the shadowing property;
\item $\htop(T)<\infty$.
\end{enumerate}
\end{corollary}

\begin{proof}
Collapse one boundary component of the annulus in
Theorem~\ref{thm:2d-annulus} to a point.  The quotient is a closed disc,
hence is homeomorphic to the square.  The distinguished interior circle and
all pseudo-orbits through it are disjoint from the collapsed boundary, so
its chain-component and shadowability properties are unchanged.  The other
boundary component remains a rotation circle, while the collapsed boundary
becomes an additional fixed-point chain component.  The entropy does not
increase under the quotient.
\end{proof}

\subsection{Insertion into arbitrary manifolds and density}\label{subsec:dense-insertion}

This subsection proves
Theorem~\ref{thm:shadowable-dense}.  The local blocks constructed above are
not isolated examples. They are used as plugs, inserted by a
periodic-point surgery within an arbitrarily small $C^0$-neighborhood of
any manifold homeomorphism.  We first create a periodic
orbit and then replace a sufficiently small neighbourhood of each point of
that orbit creating a periodic cube.  The extension step is an immediate
application of Theorem~A.1 of \cite{BargeConjecture}.

The local models constructed above concern the classical two-sided set $\Sh(T)$. We use this stronger conclusion because the distinguished component consists entirely of shadowable points even though its restricted dynamics has no intrinsic shadowing. Since $\Sh(T)\subset Sh^+(T)$, the models also satisfy every forward-time conclusion used later.

\begin{lemma}[$C^0$ closing lemma for homeomorphisms]\label{lem:C0-closing}
Let $M$ be a compact manifold and let $f\in\Homeo(M)$. Every $C^0$-neighbourhood of $f$ contains a homeomorphism with a periodic point.
\end{lemma}
\begin{proof}
Choose an $f$-invariant Borel probability measure and a recurrent point $x$ in its support. Given a sufficiently small coordinate neighbourhood $U$ of $x$, recurrence provides $q\geq1$ such that $f^q(x)\in U$ and the intermediate points $f^i(x)$, $1\leq i<q$, lie outside a smaller neighbourhood containing both $x$ and $f^q(x)$. There is a homeomorphism $h$, supported in that smaller coordinate neighbourhood and arbitrarily close to the identity, such that $h(f^q(x))=x$. For $g=h\circ f$ the first $q-1$ iterates of $x$ are unchanged and $g^q(x)=x$. By choosing the return and the support sufficiently small, both $g$ and $g^{-1}$ are arbitrarily uniformly close to $f$ and $f^{-1}$, respectively.
\end{proof}

\begin{lemma}[Creation of a periodic spot]\label{lem:periodic-spot}
Let $M$ be a compact manifold, let $f\in\Homeo(M)$, and suppose that $p$ is
a periodic point of least period $q$.  Assume that the local return map
$f^q$ preserves orientation near $p$.  Then, for every $C^0$-neighbourhood
$\mathcal U$ of $f$ and every neighbourhood $U$ of the orbit of $p$, there
are $g\in\mathcal U$ and a locally flat closed $n$-cell $D\subset U$ such
that
\[
 D,g(D),\ldots,g^{q-1}(D)
\]
are pairwise disjoint and
\[
 g^q|_D=\id_D.
\]
Moreover, $g=f$ outside $U$.
\end{lemma}

\begin{proof}
Choose a locally flat closed $n$-cell $C$ containing $p$ in its interior,
so small that
\[
 C\cap f^i(C)=\varnothing, \text{ for all } 1\le i<q \text{ and } \bigcup_{i=0}^{q-1}f^i(C)\subset U.\]  Choose a still smaller locally
flat cell $D\Subset C$ such that $f^q(D)\Subset C$.  After shrinking $D$,
if necessary, the restriction $f^q|_D$ has the same local orientation as
the return at $p$.

Apply Theorem~A.1 of \cite{BargeConjecture} with
\[
 X=Y=C,\qquad \alpha=\id_C,\qquad B_1=f^q(D),\qquad C_1=D,
\]
and
\[
 \beta_1=(f^q|_D)^{-1}\colon f^q(D)\longrightarrow D.
\]
It gives a homeomorphism $h_C\colon C\to C$ such that
\[
 h_C|_{\partial C}=\id,
 \qquad
 h_C|_{f^q(D)}=(f^q|_D)^{-1}.
\]
Extend $h_C$ by the identity on $M\setminus C$ and denote the resulting
homeomorphism by $h$.  Put $g=h\circ f$.  Since $f^i(D)\cap C=\varnothing$
for $1\le i<q$, we have
\[
 g^i|_D=f^i|_D \text{ for every }0\le i<q,
\]
and consequently
\[
 g^q|_D=h\circ f^q|_D
       =(f^q|_D)^{-1}\circ f^q|_D
       =\id_D.
\]
The cell $C$ may be chosen arbitrarily small.  Hence $h$ is arbitrarily
$C^0$-close to the identity, both $g$ and $g^{-1}$ are arbitrarily close to
$f$ and $f^{-1}$, respectively, and the perturbation is supported in $U$.
\end{proof}

\begin{remark}\label{rem:orientation-return}
The orientation assumption causes no restriction in the application below.
If the return along the periodic orbit reverses orientation, an arbitrarily
small local perturbation splits the orbit into an orbit of twice the period;
the new return preserves orientation.  Concretely, replace each point of
the orbit by two nearby points and use Theorem~A.1 of
\cite{BargeConjecture} in pairwise disjoint coordinate cells to prescribe
the successive transitions.  The square of the original local orientation
sign is positive.  We shall use this standard orbit-doubling modification
without further mention.
\end{remark}

\begin{lemma}[Insertion into a periodic spot]\label{lem:periodic-insertion}
Let $g\in\Homeo(M)$ and let $D$ be a locally flat periodic spot of period
$q$, so $g^q|_D=\id_D$.  Let $L\in\Homeo(D)$ satisfy
$L|_{\partial D}=\id$.  Then there is $G\in\Homeo(M)$ such that
\begin{enumerate}[label=\textup{(\roman*)}]
\item $G=g$ outside $g^{q-1}(D)$;
\item $G^q|_D=L$;
\item $G$ can be made arbitrarily $C^0$-close to $g$ by choosing the
      periodic spot sufficiently small;
\item the dynamics of $G$ on
      $\widehat D=\bigcup_{i=0}^{q-1}G^i(D)$ is the cyclic extension of
      $L$, and therefore chain components, transitivity, shadowable points,
      and failure of shadowing correspond under the first-return map.
\end{enumerate}
Furthermore,
\[
 \htop(G|_{\widehat D})=\frac1q\htop(L).
\]
\end{lemma}

\begin{proof}
Set $G=g$ outside $g^{q-1}(D)$ and, for $x\in g^{q-1}(D)$, put
\[
 G(x)=L(g(x)).
\]
Here $L$ is regarded as a map of $D=g^q(D)$.  Since
$L|_{\partial D}=\id$, the two definitions agree on
$\partial g^{q-1}(D)$, and $G$ is a homeomorphism.  For $x\in D$ the first
$q-1$ transitions are unchanged, and hence
\[
 G^q(x)=L(g^q(x))=L(x).
\]
The $C^0$ estimate follows from the diameter of $D$ and its finitely many
images.  The remaining assertions follow by transporting chains,
pseudo-orbits and tracing orbits through the finitely many uniformly
continuous transition maps between the levels of the tower.  Finally,
$G^q|_{\widehat D}$ is the disjoint union of $q$ conjugate copies of $L$,
which gives the entropy formula.
\end{proof}

For $n=2$ let $L_n$ be the square homeomorphism supplied by
Corollary~\ref{cor:2d-square}, restricted to a slightly smaller square and
extended by the identity through a collar.  For $n\ge3$ use the cube
homeomorphism of Theorem~\ref{thm:example}, together with the consecutive
localized Denjoy--Rees modifications of Proposition~\ref{prop:L3}. In both
cases we obtain the following local block: for every $n\ge2$ there are
$L_n\in\Homeo([0,1]^n)$, a chain component $D_n\subset(0,1)^n$, and an
open neighborhood $W_n\subset(0,1)^n$ of $D_n$ such that
\begin{enumerate}[label=\textup{(L\arabic*)}]
\item $L_n|_{\partial[0,1]^n}=\id$, $\htop(L_n)<\infty$, and
      $([0,1]^n,L_n)$ does not have shadowing;
\item $L_n|_{D_n}$ is transitive, $D_n\subseteq\Sh(L_n)$, and
      $(D_n,L_n|_{D_n})$ does not have shadowing;
\item every chain component $C$ of $L_n$ meeting $W_n$ fails to have the
      shadowing property.
\end{enumerate}
The collar does not affect these assertions, because $W_n$ is chosen with
closure disjoint from the collar.
For $n\ge3$, property \textup{(L3)} is provided by
Proposition~\ref{prop:L3}. For $n=2$ it follows from
Theorem~\ref{thm:2d-annulus}(3) via Corollary~\ref{cor:2d-square}, since
every chain component of the square homeomorphism fails shadowing; thus
$W_2$ may be any open neighbourhood of $D_2$ with closure in the open
square.
The above insertion procedure now yields the proof of the first main
theorem.

%\begin{theorem}\label{thm:shadowable-dense}
%Let $M$ be a compact manifold of dimension $n\ge2$. The set of homeomorphisms $T\in\Homeo(M)$ for which
%\begin{enumerate}
%\item $(M,T)$ does not have the shadowing property;
%\item there is a chain component $D$ of $T$ such that $T|_D$ is transitive and
%      $D\subseteq\Sh(T)\subseteq\Sh^+(T)$, while $(D,T|_D)$ does not have the shadowing property;
%\item there is an open set $V\supset D$ such that every chain-recurrent class
%      $D'$ meeting $V$ fails to have the shadowing property,
%\end{enumerate}
%is dense in $\Homeo(M)$ endowed with the $C^0$ topology.
%Moreover, the local block inserted into the periodic spot has finite
%topological entropy. {\color{magenta} If $n\ge3$, it may be chosen using an Axiom~A
%(hence hyperbolic) building block, and $D$ may be chosen with positive
%topological entropy.}
%\end{theorem}

\begin{proof}[Proof of Theorem~\ref{mainA}]
Fix $f\in\Homeo(M)$ and a $C^0$-neighbourhood $\mathcal U$ of $f$. By Lemma~\ref{lem:C0-closing}, after an arbitrarily small perturbation we obtain $f_1\in\mathcal U$ with a periodic point $p$ of least period $q$. If the local return reverses orientation, use the orbit-doubling modification from Remark~\ref{rem:orientation-return}.

Apply Lemma~\ref{lem:periodic-spot} in a neighbourhood of the periodic orbit chosen so small that the resulting map $g$ still belongs to $\mathcal U$ and possesses a periodic spot $D_0$ of period $q$. Identify $D_0$ with $[0,1]^n$ and insert the local block $L_n$ by Lemma~\ref{lem:periodic-insertion}, choosing the spot sufficiently small that the resulting homeomorphism $T$ remains in $\mathcal U$.

Put
\[
 D=\bigcup_{i=0}^{q-1}T^i(D_n),
 \qquad
 V=\bigcup_{i=0}^{q-1}T^i(W_n),
\]
Here $D_n$ and $W_n$ are transported through the chosen cube coordinate.
Lemma~\ref{lem:periodic-insertion} shows that $D$ is a chain
component of $T$, that $T|_D$ is transitive, and that every point of $D$
is shadowable, hence positively shadowable, for $T$. The same lemma shows
that $(D,T|_D)$ does not have the shadowing property. More generally,
every chain component meeting $V$ corresponds, under the first-return
map, to a chain component of $L_n$ meeting $W_n$, and therefore fails
shadowing by property~\textup{(L3)}.

The nonshadowable pseudo-orbits in the local block remain nonshadowable in the ambient manifold, so $(M,T)$ does not have the shadowing property. Finally,
\[
 h_{\mathrm{top}}(T|_{\widehat D})=\frac1q h_{\mathrm{top}}(L_n)<\infty,
\]
by Lemma~\ref{lem:periodic-insertion}. For $n\ge3$, the construction of
$L_n$ starts from the structurally stable Axiom~A maps of
Theorem~\ref{thm:Plykin}. The inserted Cantor system in
Theorem~\ref{thm:abstract} has positive finite entropy, so the
distinguished component of $L_n$ has positive finite entropy. The passage
around the periodic spot gives
\[
 0<h_{\mathrm{top}}(T|_D)
   =\frac1q h_{\mathrm{top}}(L_n|_{D_n})<\infty.
\]
This proves the last assertion.
\end{proof}
%
%{\color{magenta} Theorem~\ref{thm:shadowable-dense} shows that classical shadowable points
%may fill a transitive chain component even though shadowing holds neither
%globally nor for the distinguished restricted system, and although every
%chain class meeting a neighborhood of that component also fails
%shadowing.}  {\color{magenta} The homeomorphisms with this property form a dense subset of
%$\Homeo(M)$.}

\subsection{A complementary realization on compact metric spaces}
\label{sec:compact-realization}

The manifold theorem realizes classical shadowable points as a coherent
chain component inside a familiar geometric phase space.  On general
compact metric spaces one can go further and prescribe the dynamics on
the entire set of shadowable points. The following theorem is another evidence
that the local-maximality hypothesis in Corollary~\ref{cor:loc_max} cannot
be replaced merely by closedness and invariance of that set.

\begin{theorem}\label{extthm}
Let $X$ be the Cantor set and let $T\colon X\to X$ be a homeomorphism.
There are a compact metric space $\mathcal M$ containing a copy of $X$
and a homeomorphism $f\colon\mathcal M\to\mathcal M$ such that
$f|_X=T$ and
\[
             Sh(f)=Sh^+(f)=X.
\]
If $Y\subset\mathcal M$ is compact and $f$-invariant and $f|_Y$ has the
shadowing property, then $Y\subset X$.  Moreover, $\mathcal M$ can be
chosen so that every Toeplitz subsystem of $f$ is contained in $X$.
\end{theorem}

\begin{proof}
If $T$ has the shadowing property, take $\mathcal M=X$ and $f=T$.
We therefore assume that $T$ does not have shadowing.

Choose nested finite clopen partitions
$\mathcal P_n$ of $X$ whose meshes tend to zero. Let $G_n$ be the
directed graph with vertex set $\mathcal P_n$ and an edge $U\to V$
whenever $T(U)\cap V\ne\emptyset$, and let
$(X_n,\sigma_n)$ be its two-sided edge shift.  For each $U\in\mathcal
P_n$, embed the compact zero-dimensional set
$[U]\cap X_n$ into $U$, using disjoint images for distinct $U$. In this way we regard $X_n$ as a compact subset of $X$, with
$x\in U$ whenever the zeroth symbol of $x\in X_n$ is $U$.

The embeddings may be chosen independently and their only relevant
property is
\begin{equation}\label{eq:layer-approximation}
\begin{split}
 \sup_{x\in X_n}d\bigl(\sigma_n(x),T(x)\bigr)&\longrightarrow0,\\
 \sup_{x\in X_n}d\bigl(\sigma_n^{-1}(x),T^{-1}(x)\bigr)&\longrightarrow0.
\end{split}
\end{equation}
Indeed, if the zeroth and first symbols of $x$ are $U$ and $V$, choose
$z\in U$ with $T(z)\in V$. Both $x,z$ belong to $U$, whereas
$\sigma_n(x),T(z)$ belong to $V$. Uniform continuity of $T$ and
$\operatorname{mesh}(\mathcal P_n)\to0$ give the first limit.  The second
follows in the same way from $T^{-1}$.

For every $n$, choose an infinite minimal subshift
$(N_n,\tau_n)$ which is not Toeplitz, and realize $N_n$ as a compact
subset of $[0,\frac{1}{2n}]$. Equip
$X\times[0,1]\times[0,1]$ with the sum metric and put
\[
 Z=X\times\{0\}\times\{0\},\qquad
 Z_n=X_n\times N_n\times\left\{\frac{1}{n}\right\},
\]
and
\[
 \mathcal M=Z\cup\bigcup_{n\ge1}Z_n
       \subset X\times[0,1]\times[0,1].
\]
The set $\mathcal M$ is closed in the compact product because every sequence
with layer indices $n_k$ tending to infinity accumulates in $Z$.  Define
\[
 f(x,q,r)=
 \begin{cases}
   (T(x),0,0),& (x,q,r)\in Z,\\
   (\sigma_n(x),\tau_n(q),\frac{1}{n}),& (x,q,r)\in Z_n.
 \end{cases}
\]
Each $Z_n$ is a clopen invariant layer.  Equation
\eqref{eq:layer-approximation}, together with
$\diam N_n\to0$ and $\frac{1}{n}\to0$, proves continuity of $f$ at $Z$.
The second limit in \eqref{eq:layer-approximation} proves continuity
of the inverse there.  Thus $f$ is a homeomorphism and $f|_Z$ is
conjugate to $T$.

\begin{figure}[H]
\centering
\begin{tikzpicture}[
  x=1cm,y=1cm,
  every node/.style={text=black,font=\small},
  layer/.style={draw=black,fill=gray!9,rounded corners=2pt,line width=.8pt},
  dynamics/.style={draw=black,->,line width=.8pt},
  guide/.style={draw=black,densely dashed,line width=.6pt}
]
%  \draw[black,->] (-5.25,-.15) -- (-5.25,4.15);
  \node at (-5.3,4.38) {$r$};

  \draw[layer] (-4.65,3.52) rectangle (4.65,4.0);
  \draw[layer] (-4.7,2.52) rectangle (4.7,3.0);
  \draw[layer] (-4.75,1.72) rectangle (4.75,2.2);
  \draw[layer] (-4.8,.79) rectangle (4.8,1.27);
  \draw[layer,fill=gray!14,line width=1pt] (-4.85,-.03) rectangle (4.85,.45);

  \node[anchor=east] at (-5.15,3.75) {$1$};
  \node[anchor=east] at (-5.05,2.75) {$\frac{1}{2}$};
  \node[anchor=east] at (-5.05,1.95) {$\frac{1}{3}$};
  \node[anchor=east] at (-5.05,1.02) {$\frac{1}{n}$};
  \node[anchor=east] at (-5.15,.18) {$0$};

  \node[anchor=west] at (-4.42,3.75)
    {$Z_1=X_1\times N_1\times\{1\}$};
  \node[anchor=west] at (-4.42,2.75)
    {$Z_2=X_2\times N_2\times\{\frac{1}{2}\}$};
  \node[anchor=west] at (-4.42,1.95)
    {$Z_3=X_3\times N_3\times\{\frac{1}{3}\}$};
  \node[anchor=west] at (-4.42,1.02)
    {$Z_n=X_n\times N_n\times\{\frac{1}{n}\}$};
  \node[anchor=west] at (-4.62,.18)
    {$Z=X\times\{0\}\times\{0\}$};

  \draw[dynamics] (1.35,3.63) -- (4.15,3.63)
    node[midway,above] {$\sigma_1\times\tau_1$};
  \draw[dynamics] (1.35,2.63) -- (4.15,2.63)
    node[midway,above] {$\sigma_2\times\tau_2$};
  \draw[dynamics] (1.35,1.83) -- (4.15,1.83)
    node[midway,above] {$\sigma_3\times\tau_3$};
  \draw[dynamics] (1.35,0.9) -- (4.15,0.90)
    node[midway,above] {$\sigma_n\times\tau_n$};
  \draw[dynamics] (1.35,.06) -- (4.15,.06)
    node[midway,above] {$T$};

  \node at (0,1.6) {$\vdots$};
  \draw[guide] (-4.85,.45) -- (-4.65,3.5);
  \draw[guide] (4.85,.45) -- (4.65,3.5);
  \draw[densely dashed,->,line width=.7pt]
    (-.35,1.72) -- (-.35,.43) node[midway,left] {$p$};
\end{tikzpicture}
\caption{The compact space
$\mathcal M=Z\cup\bigcup_{n\ge1}Z_n$.  The clopen invariant layers
$Z_n$ accumulate on the base copy $Z\cong X$, and $p$ is the projection
used in the tracing argument.  Horizontal arrows represent the dynamics. No orbit passes between layers.
\label{fig:compact-layers}
}
\end{figure}

We next prove $Z\subset Sh(f)$.  Let
$p\colon\mathcal M\to Z$ be the projection
$p(x,q,r)=(x,0,0)$.  Given $\eps>0$, first choose $r$ so large that
\[
 \operatorname{mesh}(\mathcal P_r)+\frac{3}{2r}<\frac{\eps}{2}.
\]
Since $\mathcal P_r$ is a finite clopen partition, there is $\xi>0$
such that points at distance less than $\xi$ belong to the same atom of
$\mathcal P_r$. Consequently, if $(b_i)_{i\in\mathbb Z}$ is a
two-sided $\xi$-pseudo-orbit of $T$ and $U_i\in\mathcal P_r$ contains
$b_i$, then $T(b_i)\in U_{i+1}$. Hence $(U_i)_{i\in\mathbb Z}$ is an
admissible path in $G_r$, and there is $y\in X_r$ whose
$\sigma_r$-orbit has symbol $U_i$ at every time $i\in\mathbb Z$.

By \eqref{eq:layer-approximation}, choose $N\ge r$ so large that
\[
 \sup_{m\ge N}\sup_{x\in X_m}
 d\bigl(T(x),\sigma_m(x)\bigr)<\frac{\xi}{3},
 \qquad
 \frac{3}{2N}<\frac{\eps}{2}.
\]
The finite union $\bigcup_{m<N}Z_m$ and its complement are invariant and
positively separated. Consequently, for sufficiently small
$\delta>0$, every two-sided $\delta$-pseudo-orbit
$(a_i)_{i\in\mathbb Z}$ with $a_0\in Z$ avoids these finitely many
layers. Taking also $\delta<\frac{\xi}{3}$, the projected sequence
$(p(a_i))_{i\in\mathbb Z}$ is a two-sided $\xi$-pseudo-orbit of $T$.
Choose $y\in X_r$ as in the preceding paragraph and any $q\in N_r$. 
For every $i\in\mathbb Z$, the first coordinates of
$f^i(y,q,\frac{1}{r})$ and $a_i$ belong to the same atom of $\mathcal P_r$. 
The definition of the sum metric and the choices of $r$ and $N$ therefore
give
\[
 d\bigl(f^i(y,q,\frac{1}{r}),a_i\bigr)
 \leq \operatorname{mesh}(\mathcal P_r)
      +\frac{3}{2r}+\frac{3}{2N}<\eps.
\]
Thus $(y,q,\frac{1}{r})$ traces $(a_i)_{i\in\mathbb Z}$, proving
$Z\subset Sh(f)$.
It remains to exclude the isolated layers even from $\Sh^+(f)$.  Fix
$z=(x,q,\frac{1}{n})\in Z_n$.   Because $Z_n$ is isolated, a sufficiently
accurate pseudo-orbit starting at $z$, as well as every sufficiently
close tracing orbit, remains in $Z_n$. Every positive pseudo-orbit of
$\tau_n$ through $q$ lifts to one of $f|_{Z_n}$ by using the exact
first-coordinate orbit $(\sigma_n^i(x))_{i\ge0}$. Thus
$z\in Sh^+(f)$ would imply $q\in Sh^+(\tau_n)$.  Minimality gives
$H(q)=N_n$, so Theorem~\ref{shadonH} would imply that $\tau_n$ has
shadowing.  This is impossible for an infinite minimal subshift,
since a subshift with the shadowing property is of finite type \cite{Wa78}
and hence contains a periodic point \cite{LM}.  Therefore $\Sh^+(f)\subset Z$.  Together with
$Z\subset Sh(f)\subset Sh^+(f)$, this gives
\[
             Sh(f)=Sh^+(f)=Z\cong X.
\]

Suppose now that $Y\subset\mathcal M$ is compact and invariant and that
$f|_Y$ has shadowing. If $Y_n=Y\cap Z_n$ were nonempty, then $Y_n$
would be a clopen invariant subset of $Y$, so $f|_{Y_n}$ would also
have shadowing.  In the symbolic coordinates, $Y_n$ is a subshift of
$X_n\times N_n$. A subshift has shadowing exactly when it is of
finite type \cite{Wa78}. Hence $Y_n$ would be a shift of finite type.   Its
projection onto $N_n$ is a nonempty compact invariant set and therefore
equals $N_n$ by minimality. This would make $N_n$ a sofic shift.
An infinite minimal subshift is not sofic \cite{LM}, a contradiction.
Hence $Y$ meets no $Z_n$ and therefore $Y\subset Z$.

Finally, let $A\subset\mathcal M$ be a Toeplitz subsystem.  Since $A$
is minimal and the $Z_n$ are clopen and invariant, either $A\subset Z$
or $A\subset Z_n$ for some $n$.   In the latter case the second
coordinate projection maps $A$ onto $N_n$. Factors of Toeplitz systems
are Toeplitz \cite{Toeplitz}, contradicting the choice of $N_n$. Thus
every Toeplitz subsystem is contained in $Z\cong X$.
\end{proof}

Theorem~\ref{extthm} complements the manifold realization in two ways.
	 It prescribes arbitrary dynamics on the entire set of classical
shadowable points, and it does so without deriving those points from a
 shadowing property or intrinsically shadowing subsystem.

\section{Part II: Dynamical consequences of local shadowing}\label{sec:part-consequences}

Part~I established that classical pointwise shadowability can occur
independently of the global shadowing property or its intrinsic presence in subsystems.
That was a separation: local tracing need not originate from any surrounding shadowing. This part establishes the complementary consequence. Although local shadowing is weaker than global shadowing, it is far from dynamically inert; combined with mild recurrence it already forces symbolic dynamics and positive entropy.
We now study dynamical properties that can be derived from the local shadowing property.
Theorems~\ref{mainB} and~\ref{mainC} show that local tracing
produces entropy-bearing subsystems and gives qualitative and quantitative
approximation of invariant measures by ergodic measures. At the end of
the part, Corollary~\ref{cor:dense-consequences} applies this mechanism
back to the positive-entropy manifold components, including those constructed in Part~I.
The argument proceeds from measure supports to a semi-horseshoe coding
mechanism, then to entropy-controlled measure approximation, and finally to
expansivity applications. 

\subsection{Shadowable measures}\label{SectioninvS}

Shadowable measures provide the bridge from local pointwise tracing to the
measure-approximation results below. For that bridge to be effective, a
measure-level hypothesis must give pointwise shadowability on the whole
support with uniform constants. Theorem~\ref{unifsupp} supplies precisely
this uniformization, and the subsequent criterion identifies invariant
positively shadowable measures with the presence of a chain-recurrent point
in $\Sh^+(f)$.  The measure concept was introduced in \cite{LMo} as a tool
for studying measure-theoretic stability.  We begin by recalling it.

\begin{definition}\label{SM}
    A probability measure $\mu$ is shadowable if for every $\eps>0$ there are $B\subset M$ and $\delta>0$ such that $\mu(B)=1$ and every two-sided $\delta$-pseudo-orbit through $B$ (equivalently, with $x_0\in B$) is $\eps$-shadowed.
\end{definition}

Similarly to the case of shadowable points, some arguments below require only forward pseudo-orbits. We therefore record the corresponding weaker notion for measures:

\begin{definition}\label{PSM}
     A probability measure $\mu$ is positively shadowable if for every $\eps>0$ there are $B\subset M$ and $\delta>0$ such that $\mu(B)=1$ and every positive $\delta$-pseudo-orbit through $B$ (that is, with $x_0\in B$) is $\eps$-shadowed.
\end{definition}

Results below are stated at this weaker positive level because their proofs use only forward pseudo-orbits. They therefore apply automatically to classical shadowable measures, because every shadowable measure is positively shadowable. Finally, let us highlight the fact that similarly to Section \ref{SectionSP}, the main issue is to uniformize the pointwise tracing constants.

\begin{theorem}\label{unifsupp}
    If $\mu$ is a positively shadowable measure, then $\su(\mu)\subset Sh^+(f)$ 
\end{theorem}
    \begin{proof}
The proof is divided into two parts. First, we find a full-measure set of positively shadowable points with uniform positive tracing constants.  
        We then extend those constants to $\su(\mu)$.

    Since $\mu$ is positively shadowable, for every $n>0$, there are $\delta_n>0$ and $B_n\subset M$ such that $\mu(B_n)=1$ and every positive $\delta_n$-pseudo-orbit through $B_n$ is $\frac{1}{n}$-shadowed. Let us consider
    $$B=\bigcap_{n=1}^{\infty} B_n.$$
    Therefore $\mu(B)=1$ and, by the construction, for every $\eps>0$, there is $\eta>0$ such that every positive $\eta$-pseudo-orbit through $B$ is shadowed with the prescribed accuracy.

    Notice that we cannot apply Theorem \ref{localshad} here, because $B$ is not necessarily compact. On the other hand, since the shadowing constants are now uniform, we can proceed to extend the shadowing constants to $\overline{B}$ as follows.

     Fix $\eps>0$, $x\in B$, and let $\eta>0$ be given by the uniform positive shadowability on $B$ with accuracy $\frac{\eps}{2}$. Choose $0<\delta\leq\min\{\tfrac{\eta}{2},\tfrac{\eps}{2}\}$ such that if $y,z\in M$ and $d(y,z)\leq \delta$, then $$d(f(y),f(z))\leq\frac{\eta}{2}.$$ Fix any $y\in B_{\delta}(x)$, and let $(x_i)$ be a positive $\delta$-pseudo-orbit through $y$. Therefore, $d(f(y),x_1)\leq \delta$. 
     Consider the sequence $$x'_i=\begin{cases} x_i, \textrm{ if } i\neq 0\\
     x, \textrm{ if } i=0.
         
     \end{cases}  $$
    Exactly as in the proof of Theorem \ref{localshad},  $(x'_i)$ is $\frac{\eps}{2}$-shadowed by some point $z$. It turns out that $z$ is an $\eps$-shadow for $(x_i)$. So, for every $\eps>0$, there is $\delta>0$ such that every positive $\delta$-pseudo-orbit through $B_{\delta}(B)$ is $\eps$-shadowed.
Finally, notice that  $$\su(\mu)=\bigcap_{\delta>0}\overline{B_{\delta}(B)}.$$
    Since $\su(\mu)\subset\overline{B}$, the proof is complete.
    
    \end{proof}

Putting the preceding results together gives a criterion for invariant positively shadowable measures.

\begin{theorem}
    The map $f$ admits an invariant positively shadowable measure if and only if $CR(f)\cap Sh^+(f)\neq\emptyset$.
\end{theorem}

\begin{proof}
First suppose that $\mu$ is an invariant positively shadowable measure
for $f$.  Theorem~\ref{unifsupp} gives
$\su(\mu)\subset Sh^+(f)$.  By the Poincar\'e recurrence theorem,
$\su(\mu)$ contains a recurrent point, and hence
$\Sh^+(f)\cap CR(f)\ne\emptyset$.

Conversely, suppose there is $x\in Sh^{+}(f)\cap CR(f)$. By Theorem \ref{ShadClass} $H(x)\subset Sh^+(f)$. Since $H(x)$ is compact and invariant, it supports an invariant measure $\mu$ satisfying $\su(\mu)\subset H(x)$. Since $\su(\mu)$ is compact, Theorem \ref{localshad} implies $\mu$ is positively shadowable.
\end{proof}

Thus invariant positively shadowable measures do not introduce an
independent source of tracing. Namely, their supports consist of positively
shadowable points, and recurrence places those supports in the chain
classes on which the propagation results of Section~\ref{SectionSP}
operate. We next exploit this classwise tracing to build symbolic
dynamics.

\subsection{Local shadowing and semi-horseshoes}\label{SectionLocEntr}

This subsection supplies the coding lemma used in Part~II.  Lemma~\ref{lemma}
turns finitely many separated return pseudo-orbits through one positively
shadowable point into a semi-horseshoe. Theorems~\ref{nonminimal}
and~\ref{senshors} then produce those pseudo-orbits in the two relevant
cases: a nonminimal chain class, or a minimal non-equicontinuous class.

\begin{lemma}\label{lemma}
Let $f\colon M\to M$ be a homeomorphism and $x\in Sh^+(f)$. Fix
$\eps>0$, and let $\delta>0$ be a corresponding positive tracing
constant at $x$. Suppose there are $n>0$ and $k$ distinct
$\delta$-pseudo-orbits
$X_j=(x_i^j)_{0\le i<n}$, $j=0,\ldots,k-1$, from $x$ to $x$, such that
for every $j\ne j'$ there is $0\le i<n$ with
\[
d(x_i^j,x_i^{j'})>4\eps.
\]
Then there are a compact invariant set $K\subset M$ and an
$f^n$-invariant compact set $\widehat K\subset K$ such that
$f^n|_{\widehat K}$ factors onto the full shift
$\sigma\colon\Sigma_k\to\Sigma_k$. In particular, $f$ has a
semi-horseshoe. Moreover,
\[
K=\bigcup_{j=0}^{n-1}f^j(\widehat K),
\]
and for every $z\in\widehat K$ and $s\in\mathbb Z$ there is
$j\in\{0,\ldots,k-1\}$ such that
\[
d(f^{sn+i}(z),x_i^j)<\eps
\quad\text{for }i=0,\ldots,n-1.
\]
\end{lemma}

\begin{proof}
Let $x\in Sh^+(f)$, fix $\eps>0$, and let $\delta>0$ be provided by the positive shadowability of $x$. For this $\delta$, select $k$ finite $\delta$-pseudo-orbits $X_j=(x^j_i)$ satisfying the assumptions of the lemma and let $n$ be the length of $X_j$. Let $\sigma:\Sigma_k\to \Sigma_k$ be the two-sided full shift of $k$ symbols. We shall use positive pointwise shadowability to construct a subsystem $f^n|_{\hat K}$ that factors over $\sigma$. The proof differs from \cite[Lemma 2.5]{Ka1}, where the map need not be invertible and the resulting subsystem factors over a one-sided shift. Here $f$ is a homeomorphism and the desired factor is the two-sided shift. Since only forward tracing is assumed, an additional compactness argument is needed to recover the negative coordinates. 

The construction is as follows. First, for each $s=(s_l)_{l\in \mathbb{Z}}\in\Sigma_k$ construct a two-sided $\delta$-pseudo-orbit $(x_m)_{m\in\mathbb{Z}}$  through $x$, by replacing $s_l$ by
$X_{s_l}$.
Next, for every $l>0$, construct the positive $\delta$-pseudo-orbits $$(x^l_m)_{m\geq0}=(x_m)_{m\geq -nl}.$$ In this way, every $(x^l_m)$ is a positive $\delta$-pseudo-orbit through $x$. Therefore, the positive shadowability of $x$ implies the existence of $\eps$-shadows of $(x^l_m)$. Denote by $Z^l$ the set of all such $\eps$-shadows. Next, define $$K_s=\{z\in B_{\eps}(x_0); z \textrm{ is a limit point of a sequence } (z_l)_{l\geq 0} \textit{ so that } f^{-nl}(z_l)\in Z^l)\}.$$
Observe that $K_s\neq\emptyset$, because for any choice of $w_l\in Z^l$ we have $f^{nl}(w_l)\in \overline{B}_{\eps}(x_0)$, so the sequence $(f^{nl}(w_l))_{l\geq 0}$ has limit points by compactness.

\vspace{0.1in}
\textit{\textbf{Claim:} Any $z\in K_s$ is an $\eps$-shadow for the two-sided $\delta$-pseudo-orbit $(x_m)_{m\in\mathbb{Z}}$ }
\vspace{0.1in}

To see why the claim holds, fix $z\in K_s$. Therefore, there is a sequence $z_l\in f^{nl}(Z^l)$ such that $z_l\to z$. So, the continuity of $f$ implies $$d(f^m(z),x_m)=\lim\limits_{l\to\infty}d(f^{m}(z_l),x_m).$$
Since by the construction, $$d(f^{m}(z_l),x_m)\leq \eps,$$
for $l$ big enough, we conclude $d(f^m(z),x_m)\leq \eps $, and the claim is proved.

Next, define $$\hat K=\bigcup_{s\in \Sigma}K_s.$$
It is straightforward to verify that $\hat K$ is closed and $f^n$-invariant. 
Define $$\pi:\hat K\to \Sigma_k,$$ by setting $\pi(z)=s$, whenever $z\in K_s$.
 It is immediate to verify that $\pi$ is surjective. To see that $\pi$ is well defined, observe that if $s,s'$ are distinct sequences on $\Sigma$, then any pair of points $z,z'\in K$ satisfying $\pi(z)=s$ and $\pi(z')=s'$ must be different. Indeed, this is due to the assumption  $d(x_i,y_i)>4\eps$. Finally, it is also immediate to verify  the conjugacy equation:
$$\pi\circ f^n=\sigma \circ\pi.$$
To conclude the proof, we write $K=\bigcup_{i=0}^{n-1}f^i(\hat K)$. Thus, the $f^n$-invariance of $\hat K$ implies the $f$-invariance of $K$ and the proof is done.

\end{proof}

\begin{lemma}\label{denseminimals}
    Let $x\in Sh^+(f)\cap CR(f)$. If $z\in H(x)$, then every neighborhood of $z$ contains a minimal point. 
\end{lemma}
\begin{proof}
    Fix $z\in H(x)$.  By Theorem \ref{ShadClass}, $z$ is also positively shadowable. If $z$ is minimal, the result is trivial. So, suppose $z$ is not a minimal point. Therefore, $H(x)$ is not minimal. This implies the existence of a minimal proper subset $K\subset H(x)$. Since $z$ is not minimal, $z\notin K$.
     
Fix $y\in K$, $\eps>0$, and let $\delta>0$ be provided by the positive
shadowability of $z$. Since $z\sim y$, there are finite
$\delta$-pseudo-orbits
\[
X=(z,x_1,\ldots,x_{k_1-2},y)
\quad\text{and}\quad
Y=(y,y_1,\ldots,y_{k_2-2},z).
\]
So, concatenating these two $\delta$-pseudo-orbits we obtain a non-trivial $\delta$-pseudo-orbit $XY$ connecting $z$ to $z$.

    Next, we construct a periodic $\delta$-pseudo-orbit $(p_i)_{i\geq 0}$ with $p_0=z$ by infinite repetition of the cyclic concatenation of $X$ and $Y$. Let $N$ denote its period, so that $p_{kN}=z$ for every $k\geq 0$.
        By the positive shadowability of $z$ there is a point $p$ $\eps$-shadowing $(p_i)$. Moreover, $f^{kN}(p)\in \overline{B}_{\eps}(z)$ for every $k\geq0$, since $p_{kN}=z$. Then, let $q$ be a minimal point of $f^N$ in the compact, $f^N$-invariant set $\omega_{f^N}(p)\subset \overline{B}_{\eps}(z)$. Since a minimal point of $f^N$ is also a minimal point of $f$, $q$ is a minimal point for $f$ contained in $\overline{B_{\eps}(z)}$. Since $\eps$ was chosen arbitrarily, the proof is complete.

\end{proof}

We are now ready to prove the first semi-horseshoe criterion.

\begin{theorem}\label{nonminimal}
    Suppose there is $x\in Sh^+(f)\cap CR(f)$. If $H(x)$ is not a minimal set, then $f$ contains a semi-horseshoe.
\end{theorem}

\begin{proof}
In view of Lemma \ref{lemma}, we only need to find $\eps>0$ and a positive pointwise shadowing constant $\delta$ for which two suitably separated pseudo-orbits can be generated.
By assumptions $H(x)$ is not a minimal set.  Hence, there exists a proper minimal subset $K\subset H(x)$. Fix $z\in H(x)\setminus K$. By Theorem \ref{ShadClass}, $z\in Sh^+(f)$.  We then set  $$\eps=\frac{1}{5}d(z,K)>0.$$
Fix  $0<4\delta<\eps$. First, by Lemma~\ref{denseminimals} we can find  minimal sets arbitrarily close to $z$. In particular, we can construct a non-trivial $\delta$-pseudo-orbit $C_1$ connecting $z$ to $z$. On the other hand, if we fix $y\in K$, since $y\sim z$ there are two finite $\delta$-pseudo-orbits $A_1$ and $A_2$ connecting $y$ to $z$ and $z$ to $y$, respectively. Let $A_3$ be a $\delta$-pseudo-orbit from $y$ to $y$ contained completely in $K$. If we define 
  $$C_2=A_1A_3\cdots A_3A_2,$$ where  $A_3$ is beign repeated sufficiently many times so that   $C_2$ is a $\delta$-pseudo-orbit from $z$ to $z$ and there is $j$ such that $j\#C_1$-th element of $C_2$ belongs to $K$. Let $n_1$ and $n_2$ be the lengths of $C_1$ and $C_2$, respectively. Write $n=n_1n_2$. Therefore, by defining 
$$X_1=C_1C_1\cdots C_1 \textrm{ and } X_2=C_2C_2\cdots C_2, $$
so that both $X_1$ and $X_2$ have $n$ elements, obtain the all the assumptions of Lemma \ref{lemma}. Namely, $j\#C_1$-th element of $X_1$ is $z$, while corresponding element of $X_2$ belongs to $K$, so they are at distance larger than $4\eps$. Therefore, $f$ admits a semi-horseshoe. 

\end{proof}

We also can obtain the semi-horseshoe for $f$ when $H(x)$ is minimal, under extra assumptions.

\begin{theorem}\label{senshors}
   Suppose there is $x\in Sh^+(f)\cap CR(f)$. If $x$ is non-periodic, $H(x)$ is minimal and  not equicontinuous then 
   $f$ contains a semi-horseshoe.
\end{theorem}

    \begin{proof}

        To begin with,  since $H(x)$ is minimal and not equicontinuous. Recall that for minimal systems equicontinuity in the two-sided sense of Section~\ref{basics} is equivalent to forward equicontinuity, because a forward equicontinuous minimal homeomorphism is conjugate to an isometry, see \cite{Kurka}. It follows from \cite{AY} that the map $f|_{H(x)}$ is sensitive to initial conditions.     So, there is  a sensitivity constant $C>0$ such that for any non-empty $U\subset H(x)$ open in $H(x)$, there is $n\in \mathbb{N}$ such that $\diam(f^n(U))>C$. Take any $0<\eps<\frac{C}{4}$.       Since $x$ is not periodic and $H(x)$ is minimal, $H(x)$ has no isolated points. 
         Let $0<\delta<\eps$ be provided by the positive shadowability of $x$ with accuracy $\eps$, and let $0<\eta<\delta$ be such that $d(y,z)\leq \eta$ implies $d(f(y),f(z))\leq \delta$.
       
        Let $x\in U\subset H(x)$ be any open set with diameter smaller than $\eta$.        
         Then, there are $y,z\in U$ and $n\in \mathbb{N}$ such that $d(f^n(y),f^n(z))>C$. 
Since $H(x)$ is minimal, there are $n_y,n_z>n$ such that
$f^{n_z}(z),f^{n_y}(y)\in U$.  We can therefore form the finite
$\delta$-pseudo-orbits
\[
C_1=(x,f(y),\ldots,f^{n_y-1}(y),x)
\quad\text{and}\quad
C_2=(x,f(z),\ldots,f^{n_z-1}(z),x).
\]
    Observe that $C_1$ is a $\delta$-pseudo-orbit connecting $x$ to $x$, because by definition of $U$ we have $d(f(x),f(y))\leq\delta$ and $d(f^{n_y}(y),x)<\delta$. The same argument applies to $C_2$.
By denoting $m=n_yn_z$ and by defining 
$$X_1=C_1C_1\cdots C_1 \textrm{ and } X_2=C_2C_2\cdots C_2, $$

we obtain $\delta$-pseudo-orbits $X_1$ and $X_2$ with $m$ elements that satisfy the assumptions of Lemma \ref{lemma}. Therefore, $f$ admits a semi-horseshoe. 
\end{proof}

\subsection{Measure approximation and entropy flexibility}\label{sec:mainthms}

This subsection proves the measure-approximation results.
Theorem~\ref{mainB} gives qualitative approximation by ergodic measures
supported on semi-horseshoes, while Theorem~\ref{mainC} gives
quantitative entropy targeting and, under entropy expansivity, exact
entropy realization.  Together they form a local analogue of the
measure-approximation and entropy-flexibility phenomena usually obtained
from global shadowing. Since the arguments are long, we
split them into several steps, beginning with a common empirical-measure
estimate.

\begin{lemma}\label{lem:empirical}
Let $\mu_1,\ldots,\mu_k$ be ergodic measures. For every $\eps,R>0$
there is $N>0$ with the following property.  Suppose $n>N$ and
$X=(x_j)_{j\ge0}$ and $Y=(y_j)_{j\ge0}$ are positive
$\eps$-pseudo-orbits satisfying:
\begin{enumerate}
\item\label{lem:empirical:1}  $d(x_j,y_j)\leq\eps$ for every $j\ge0$;

\item\label{lem:empirical:2} for every $i=1,\ldots,k$, there is a
sequence $(p_i^m)_{m\ge0}$ such that, for every $n\ge N$,
\[
d^*(\SE_n(p_i^m),\mu_i)\le\eps;
\]
\item\label{lem:empirical:3}  $Y$ is the concatenation
\[
Y=X_1^1P_1^1\cdots X_k^1P_k^1
  X_1^2P_1^2\cdots X_k^2P_k^2\cdots,
\]
where $P_i^m=(f^j(p_i^m))_{0\le j\le n-1}$ and
$\#X_i^m\le R$ for every $m\ge0$ and $i=1,\ldots,k$;
\item\label{lem:empirical:4} $s_0=0$ and
$s_m=s_{m-1}+kn+\sum_{l=1}^k\#X_l^m.$
\end{enumerate}
Then, for every $m>0$,
\[
d^*\left(\frac1k\sum_{i=1}^k\mu_i,
         \frac1{s_m}\sum_{j=0}^{s_m-1}\delta_{x_j}\right)\le3\eps.
\]
\end{lemma}

\begin{proof}
Fix $\eps,R>0$ and choose $N>0$ such that $\frac{3R}{n}\le\eps$ whenever
$n\ge N$.

For the $m$-th block
\[
Y_m=X_1^mP_1^m\cdots X_k^mP_k^m,
\]
condition~\ref{lem:empirical:3} gives
$kn\le\#Y_m\le k(n+R)$.

By assumption \ref{enum:measure-approx-2} and the item \ref{enum:measure-approx-3} of  Lemma \ref{lem:measure-approx}, we obtain:

$$d^*\left(\frac{1}{k} \sum_{i=1}^k \mu_i, \frac{1}{mk}\sum_{l=1}^m\sum_{i=1}^k\SE_{n}(p^l_i)\right)\leq\eps$$
On the other hand,  for every $m>0$  item \ref{enum:measure-approx-1} of Lemma  \ref{lem:measure-approx} implies:

\begin{eqnarray*}d^*\left(\frac{1}{mk}\sum_{l=1}^m\sum_{i=1}^k\SE_{n}(p^l_i),\frac{1}{s_m}
\sum_{j=0}^{s_m-1}\delta_{y_j}\right)\leq
 \frac{s_m+kmn}{s_mmnk}mkR + \frac{mkR}{s_mmnk}kmn
 \leq \\ \leq \frac{R}{n}+\frac{kmR}{s_m} +\frac{kmR}{s_m}\leq \frac{3R}{n}\leq  \eps
\end{eqnarray*}
Finally, by item \ref{enum:measure-approx-2} of Lemma \ref{lem:measure-approx} and assumption \ref{lem:empirical:1}, it holds
  \begin{eqnarray*}
     d^*\left(\frac{1}{s_m}
\sum_{j=0}^{s_m-1}\delta_{y_j},\frac{1}{s_m}
\sum_{j=0}^{s_m-1}\delta_{x_j}\right)\leq \eps.
 \end{eqnarray*}

So, by combining the two previous inequalities we obtain:

$$d^*\left(\frac{1}{k}\sum_{i=1}^k\mu_i,\frac{1}{s_m}
\sum_{j=0}^{s_m-1}\delta_{x_j}\right)\leq3\eps,$$
and  the proof is complete.
    
\end{proof}

\begin{proof}[Proof of Theorem~\ref{mainB}]
Let $x$ be as in the statement and let $\mu$ be an invariant probability
measure with $K=\Supp(\mu)\subset H(x)$.
    Since $H(x)$ is not equicontinuous, by %following the arguments in 
    Theorems \ref{nonminimal} and \ref{senshors} we can obtain $C>0$ such that for any $\delta>0$ there are $m>0$, and a pair of finite $\delta$-pseudo-orbits $X=(x_i)_{0\leq i \leq m-1}$ and $Y=(y_i)_{0\leq i\leq m-1}$ connecting $x$ to $x$ so that $d(y_j,x_j)>C$, for some $0<j<m$. Note that in the non-minimal case the loops constructed in the proof of Theorem~\ref{nonminimal} are based at a point $z\in H(x)$ rather than at $x$. However, since $x\sim z$, we attach $\delta$-pseudo-orbits from $x$ to $z$ and from $z$ to $x$ to both loops, and repeat the loops so that the resulting pseudo-orbits have equal length and remain separated at a common index.

By Theorems~\ref{ShadClass} and~\ref{localshad}, after choosing
$0<4\eps<C$ we can take $0<2\delta<\eps$ so that every positive
$\delta$-pseudo-orbit through $H(x)$ is $\eps$-shadowed. Apply
Lemma~\ref{lemma1} to $f|_K$ to obtain ergodic measures
$\mu_1,\ldots,\mu_k$ supported on $K$ such that
\[
d^*\left(\frac{1}{k}\sum_{i=1}^k\mu_i,\mu\right)\leq\eps.
\]
Let $\SU=\{U_1,\ldots,U_r\}$ be an open cover of $K$ by sets of
diameter less than $\delta$.

Since $H(x)$ is chain transitive, for every pair $i\ne j$ there is a
$\delta$-pseudo-orbit $Z_{i,j}$ connecting $U_i$ to $U_j$.  Let $Z_i$
be a $\delta$-pseudo-orbit connecting $U_i$ to $x$,
$i=1,\ldots,r$. Select a generic point $p_i$ of $\mu_i$ for every
$i=1,\ldots,k$, and put
\[
P^n_i=(f^j(p_i))_{0\le j\le n-1}.
\]
 Let $Z_0$ be a $\delta$-pseudo-orbit connecting $x$ and $p_1$. 

 Next, fix $n>0$ and construct the following pair of $\delta$-pseudo-orbits:

$$P_1=XZ_0P^n_1Z_{i_1,j_1}P^n_2Z_{i_2,j_2}\cdots P^n_kZ_{i_k} \textrm{ and } P_2=YZ_0P^n_1Z_{i_1,j_1}P^n_2Z_{i_2,j_2}\cdots P^n_kZ_{i_k},$$
so that:
\begin{itemize}
    \item $f^n(p_l)\in U_{i_l}$ and $p_{l+1}\in U_{j_l}$ for
    $l=1,\ldots,k-1$.
    \item $f^n(p_k)\in U_{i_k}$.
    
\end{itemize}

Note that by finiteness, there is $R>0$ bounding the length of $Z_{i,j}$, $Z_i$, $Z_{i_k}XZ_0$ and $Z_{i_k}YZ_0$, for every $i=1,\ldots,r$. 
If we denote  the lengths of $Z_0$,  $Z_{i_l,j_l}$ and $Z_{i_k}$ by $s^n_0$, $s^n_l$ and $s^n_k$, respectively, we obtain that $P_1$ and $P_2$ have length
$$s_n=m+kn+\sum_{l=0}^ks^n_l\leq m+kn+(k+1)R.$$
Since $p_i$ is a generic point of $\mu_i$, there is $N>0$ such that $d^*(\SE_n(p_i),\mu_i)\leq\eps$, for $n>N$ and $i=1,\ldots,k$. 

Moreover, by construction there are  $x_j\in P_1$ and $y_j\in P_2$ such that  $d(x_j,y_j)>C$.
Hence, by Lemma~\ref{lemma}, there is a compact subset $\hat K\subset M$ such that $f^{s_n}|_{\hat K}$ factors over a full shift of $2$ symbols. Let $$K_\eps=\bigcup_{i=0}^{s_{n}-1}f^i(\hat K).$$
So, $K_{\eps}$ has positive topological entropy and therefore it supports an ergodic measure $\nu_{\eps}$ with positive entropy. Let $p_{\eps}$ be a generic point of $\nu_{\eps}$. Since the orbit of a generic point consists of generic points, we can assume $p_{\eps}\in \hat K$. Since $p_{\eps}$ is generic for $\nu_{\eps}$, for all sufficiently large $r$ $$d^*(\SE_{rs_n}(p_\eps), \nu_{\eps})\leq \eps.$$

We shall show that $d^*(\nu_{\eps},\mu)\leq 6\eps$ and since $\eps$ was chosen arbitrarily it will finish the proof. For this purpose, it is necessary to recall from the proof of Lemma \ref{lemma}, that every sequence $(r_j)\in \Sigma_2$ codes a $\delta$-pseudo-orbit  $ C^\infty(r)=(C(r_j))_{j\in\Z}$. Moreover, for every point $x\in \hat K$, there is $r\in \Sigma_2$ such that $x$ 
   $\eps$-traces the pseudo orbit $\xi$ with $\sigma^t(\xi)
  \in C^\infty(r)$ for some $0\leq t<m$. In particular, this holds for $p_{\eps}$, which $\eps$ traces a $\delta$-pseudo-orbit $\xi=(\xi_t)_{t\in \Z}$, which is a concatenation of $P_1$ and $P_2$. Consequently, $(\xi_j)_{j\geq 0}$ and $(f^j(p_{\eps}))$ are under the hypothesis of Lemma \ref{lem:empirical}. 
  With this in mind, we have to deal with the following inequality. 
       
   \begin{eqnarray*} d^*(\nu_{\eps},\mu)\leq d^*(\nu_{\eps},\SE_{rs_n}(p_{\eps}))+ d^*(\SE_{rs_n}(p_{\eps}),\frac{1}{rs_n}\sum_{t=0}^{rs_n-1}\delta_{\xi_t})+\\
   +d^*(\frac{1}{rs_n}\sum_{t=0}^{rs_n-1}\delta_{\xi_t},\frac{1}{k}\sum_{i=1}^k\mu_i)+ d^*(\frac{1}{k}\sum_{i=1}^k\mu_i,\mu) 
\end{eqnarray*}
Observe that the first and fourth terms on the right-hand side of the preceding inequality are trivially less than or equal to $\eps$. The second term is less than or equal to $\eps$ by (2) of Lemma \ref{lem:measure-approx}. The third term is less than or equal to $3\eps$ as a consequence of Lemma \ref{lem:empirical}. This completes the proof.

\end{proof}

We now prove Theorem~\ref{mainC}. The proof follows the
orbit-concatenation scheme used for Theorem~\ref{mainB}, with the
additional requirement of controlling the lower entropy bound. The
resulting semi-horseshoes need not be contained in $\su(\mu)$, but their
invariant measures remain close to $\mu$ in the weak${}^*$ topology.  We
organize the construction into the following steps:

\begin{itemize}
    \item \textbf{Step 1:} Fix the constants supplied by positive pointwise shadowability and entropy expansiveness, and find ergodic measures approximating $\mu$ and its entropy.
    \item \textbf{Step 2:} For the measures obtained in Step 1, choose
    separated sets whose cardinalities give an entropy lower bound
    greater than $c$.
    \item \textbf{Step 3:} We then use positive pointwise tracing together with the points obtained in Step 2 to construct a closed invariant subset $K$ whose dynamics can be coded by a semi-horseshoe. 

    \item \textbf{Step 4:} Construct an ergodic measure on the semi-horseshoe with entropy greater than $c$, and show that every measure on $K$ having a generic point belongs to a fixed neighborhood of $\mu$.
    
    \item  \textbf{Step 5:} Identify the fibres of the factor map with dynamical balls for $f|_K$ and use entropy expansiveness to obtain the required entropy estimates.
    \item \textbf{Step 6:} Use properties of full shifts to obtain entropy flexibility.
\end{itemize}

We now give the details following the above general steps.

\begin{proof}[Proof of Theorem~\ref{mainC}]

We first write the proof when $h_\mu(f)<\infty$.  If
$h_\mu(f)=\infty$, fix an auxiliary finite number $H>c$.  Parts~(2) of
Lemma~\ref{lemma1} and~(3) of Lemma~\ref{lemma2}, applied with
sufficiently small parameters, replace \eqref{eq:muAn} and
\eqref{entrcompAB} below by the corresponding lower bounds with $H$ in
place of $h_\mu(f)$. All subsequent entropy estimates use only these
lower bounds. Choosing $\kappa<\frac{H-c}{8}$ therefore gives the same
construction and conclusions.  We henceforth assume
$h_\mu(f)<\infty$.

{\bf Step 1:}  Fix $x\in Sh^+(f)\cap CR(f)$ and an invariant measure
$\mu$ supported on $H(x)$. Fix $c\in[0,h_\mu(f))$, and let
$\eps_0>0$ be the approximation accuracy required in the statement.
If $f$ is entropy expansive, let $e>0$ be an entropy-expansiveness
constant; otherwise set $e=1$.
Choose
\[
0<\kappa\leq
\min\left\{\eps_0,\frac{h_\mu(f)-c}{8}\right\}.
\]
The bound $\kappa\leq\eps_0$ ensures that the final estimate
$d^*(\mu_c,\mu)<\kappa$ gives the required accuracy.

Apply Lemma~\ref{lemma1} to the restricted system
$f|_{\su(\mu)}$. There are ergodic measures
$\mu_1,\ldots,\mu_k$ supported on $\su(\mu)\subset H(x)$ such that
$$\left| \frac{1}{k}\sum_{i=1}^kh_{\mu_i}(f)
-h_{\mu}(f)\right|\leq\frac{\kappa}{16} \quad\textrm{ and }\quad d^*(\nu,\mu)\leq\frac{\kappa}{8},$$
where $\nu=\frac{1}{k}\sum_{i=1}^k\mu_i$.  The entropy function is
affine on $\SM_f(M)$ (see, for example,
\cite[Theorem~8.1]{Walters}), so
$h_{\nu}(f)=\frac{1}{k}\sum_{i=1}^kh_{\mu_i}(f)$.

For each $i$, apply Lemma~\ref{lemma2} with entropy tolerance
$\frac{\kappa}{16}$ and with the weak${}^*$ topology neighborhood
\[
\mathcal V_i=\left\{\zeta\in\SM_f(M):
       d^*(\zeta,\mu_i)<\frac{\kappa}{8}\right\}.
\]
Let $\rho_i>0$ be the resulting separation scale, and put
$\rho=\min_i\rho_i$. Choose
\[
0<\beta<\frac18\min\{\rho,e,\kappa\},
\]
and then choose $0<\delta<\frac{\beta}{2}$ using
Theorem~\ref{localshad} for the compact set $H(x)$, small enough that
every positive $2\delta$-pseudo-orbit through
$B_{2\delta}(H(x))$ is $\beta$-shadowed.
 
\vspace{0.1in}

{\bf Step 2:}
Lemma~\ref{lemma2} gives $n_0>0$ such that, for every $n\geq n_0$ and
every $i$, there is an $n$-$\rho$-separated set
$A^n_i\subset\su(\mu_i)\subset H(x)$ satisfying
\[
\left|\frac{\log(\#A^n_i)}{n}-h_{\mu_i}(f)\right|
  <\frac{\kappa}{16},
\qquad
d^*(\SE_n(y),\mu_i)<\frac{\kappa}{8}
   \quad(y\in A_i^n).
\]
In what follows, we will enlarge $n_0$ without further mention.
Denote $A^n= \prod_{i=1}^{k}\#A_i^n$. Consequently, we obtain 
\begin{equation}\label{eq:muAn}\left|\frac{\log(A^n)}{kn}-h_{\mu}(f)\right|\leq \frac{1}{k}\sum_{i=1}^k\left|\frac{\log(\#A^n_i)}{n}-h_{\mu_i}(f)\right|+|h_{\nu}(f)
-h_{\mu}(f)|<\frac{\kappa}{8}.
\end{equation}

\begin{equation}\label{empirapprox}
y\in A^n_i
\quad\Longrightarrow\quad
d^*(\SE_n(y),\mu_i)<\frac{\kappa}{8}.
\end{equation}

Recall that, by Lemma~\ref{lem:measure-approx}, if $z$ $\beta$-shadows
a pseudo-orbit $z_0,\ldots,z_{m-1}$, then
\begin{equation}\label{approx-shadow}
d^*\left(\SE_m(z),\frac{1}{m}\sum_{i=0}^{m-1}\delta_{z_i}\right)
 <\beta<\frac{\kappa}{8}.
\end{equation}

\vspace{0.1in}

{\bf Step 3:}
Fix an open cover
$\SU=\{U_1,\ldots,U_p\}$ of $\su(\mu)$ by sets of diameter less than
$\delta$.  Since the measures $\mu_i$ were obtained from the restricted
system on $\su(\mu)$, the same cover contains every $\su(\mu_i)$.
Chain transitivity of $H(x)$ gives, for every pair
$U_l,U_{l'}\in\SU$, a $\delta$-pseudo-orbit of uniformly bounded length
connecting $U_l$ to $U_{l'}$; let $R$ be a common upper bound for these
lengths.

For $1\le i\le k$, $n\ge n_0$, and $U_l,U_{l'}\in\SU$, put
\[
 A^n_{i,l,l'}=A^n_i\cap U_l\cap f^{-n}(U_{l'}).
\]
The pigeonhole principle gives a pair $l,l'$ such that
$$
 \#A^n_i\geq \#(A^n_{i,l,l'})\geq \frac{\#A^n_i}{p^2}. 
$$ 
Choose one such pair and denote
$B^n_i=A^n_{i,l,l'}$, $U_l=V_i$, and $U_{l'}=W_i$.
After enlarging $n_0$ if necessary, for every $i$ we have

$$
0\leq \frac{\log(\#A^n_i)}{n}-\frac{\log(\#B^n_i)}{n}\leq 
\frac{2\log(p)}{n}
\leq \frac{\kappa}{8},
$$

and as a consequence of \eqref{eq:muAn} we get:
\begin{equation}\label{entrcompAB}
\left|\frac{1}{kn}\sum_{i=1}^k\log(\# B^n_i)-h_{\mu}(f)\right|<\frac{\kappa}{4}.
\end{equation}

Fix $n\geq n_0$.  For $i=1,\ldots,k$, put $s_i=\#B^n_i$, enumerate
$B_i^n=\{y_0^i,\ldots,y_{s_i-1}^i\}$, and define the following
$\delta$-pseudo-orbits:
\begin{itemize}
    \item For each $i=1,\ldots,k$ and 
     $j=0,\ldots,s_i-1$, let
     $C^{j}_i=(y_j^i,\ldots,f^{n-1}(y_j^i))$.
    \item For $i=1,\ldots,k-1$, let $C_i$ be a
    $\delta$-pseudo-orbit of length $m_i\leq R$ connecting $W_i$ to
    $V_{i+1}$.
    \item Let $C_k$ be a $\delta$-pseudo-orbit of length $m_k\leq R$
    connecting $W_k$ to $V_1$.
\end{itemize}
By the  construction, for any choices of $0\leq l_i\leq s_i-1$ the concatenation
\[
C(l_1,\ldots,l_k)
   =C^{l_1}_1C_1C^{l_2}_2C_2\cdots C^{l_k}_kC_k
\]
is a $\delta$-pseudo-orbit connecting $y_{l_1}$ to $y_{l_1}$. 

Furthermore, every $C(l_1,\ldots,l_k)$ has the following properties:
\begin{itemize}
\item For any choice of $l_1,\ldots,l_k$,
$C(l_1,\ldots,l_k)$ is a $\delta$-pseudo-orbit
 completely contained in $H(x)\subset Sh^+(f)$.
\item For any choice of $l_1,\ldots,l_k$,
$C(l_1,\ldots,l_k)$ has length
$m=nk+\sum_{i=1}^k m_i$.

\item If $l_i\ne l_i'$ for some $i$, the $n$-$\rho$-separation of
$B_i^n$ gives an index $0\le r<m$ such that
\[
d\bigl(C(l_1,\ldots,l_k)_r,
       C(l'_1,\ldots,l'_k)_r\bigr)>\rho.
\]
Consequently there are $s=s_1\cdots s_k$ distinguishable
concatenations, and points $\beta$-shadowing distinct concatenations
are different.
\end{itemize}

For simplicity, enumerate the preceding pseudo-orbits
$C(l_1,\ldots,l_k)$ as $C(t)$, where $0\le t<s$.
Strictly speaking, the pseudo-orbits $C(t)$ do not share a common base
point, as required in Lemma~\ref{lemma}, because their initial points
vary in $V_1$.  Since $\diam V_1<\delta$, however, every bi-infinite
concatenation of the blocks $C(t)$ is a $2\delta$-pseudo-orbit
contained in $H(x)$.  The choice of $\delta$ in Step~1 therefore allows
the proof of Lemma~\ref{lemma} to be applied without change. Hence
there is a compact and invariant subset  $ K\subset M$ and $f^m$ invariant set $\hat K\subset K$ such that $f^m|_{\hat K}$ factors over the full shift of $s$ symbols.
  
  Recall from Lemma \ref{lemma}, that every sequence $(r_j)\in \Sigma_s$ codes a $2\delta$-pseudo-orbit  $ C^\infty(r)=(C(r_j))_{j\in\Z}$. Moreover, for every point $x\in K$, there is $r\in \Sigma_s$ such that $x$
   $\beta$-traces the pseudo orbit $\xi$ with $\sigma^t(\xi)
  \in C^\infty(r)$ for some $0\leq t<m$. Observe that we cannot require uniqueness of shadowing points, since we are not assuming the expansiveness of $f$.  On the other hand, by the choice of $\beta$, if $x$ and $y$ shadow distinct pseudo-orbits $C(r)$, then $x$ and $y$ are $m$-$\frac{\rho}{2}$-separated. 
  Thus, we have  $$\log(S(m,\frac{\rho}{2},K))\geq \log(s)=\log(s_1\cdots s_k)=\sum_{i=1}^k\log(\#B^n_i).$$
    Consequently, for every $j>0$, it holds $$\frac{1}{jm}\log(S( jm,\frac{\rho}{2},K))\geq \frac{j}{j(n+R)
     k}\sum_{i=1}^k\log(\#B^n_i).$$
Therefore $$h(f|_K)\geq \frac{1}{(n+R)k}\sum_{i=1}^k\log(\#B^n_i).$$
We may assume that $n_0$ is large enough that
$$ \left|\frac{1}{(n+R)k}\sum_{i=1}^k\log(\#B^n_i)- \frac{1}{nk}\sum_{i=1}^k\log(\#B^n_i)    \right|\leq\frac{\kappa}{8}.$$
Combining the preceding lower bound with \eqref{entrcompAB}, we obtain
\[
 h(f|_K)>h_\mu(f)-\frac{3\kappa}{8}>c.
\]
Only this lower estimate is needed here; no upper estimate for
$h(f|_K)$ is asserted.

\vspace{0.1in}
{\bf Step 4:}

By the variational principle, there is an ergodic measure $\mu_c$ supported on $K$ and satisfying $h_{\mu_c}(f)>c$. We claim that for any ergodic measure $\lambda\in \SM_f(K)$, we have $d^*(\lambda,\mu)\leq\kappa$.  Fix any generic point $x\in K$ of $\lambda$. By definition,  \begin{equation}\label{genericapprox}d^*(\SE_a(x),\lambda)\to 0 ,\end{equation}
when $a\to \infty$. On the other hand $f^t(x)\in \hat{K}$ for some $0\leq t<m$, hence there is a sequence $r\in\Sigma_s$ such that $\pi(f^t(x))=r$. In particular, $f^t(x)$ traces the $\delta$-pseudo-orbit $C^\infty(r)$. 
We then obtain by \eqref{approx-shadow} and Lemma~\ref{lem:measure-approx} that for any $a>t+m$, if $b$ is the largest integer such that $bm\leq a$, then for empirical measure $$
\lambda_{C^\infty(r)}^b=\frac{1}{bm}\sum_{p=0}^{b-1}\sum_{j=0}^{m-1}\delta_{C(r_p)_j}$$
we have:
\begin{equation}\label{oapprox}d^*\left(\SE_a(x),\lambda_{C^\infty(r)}^b\right)\leq \frac{3\kappa}{8}.\end{equation}
This implies, by \eqref{empirapprox}, Lemma~\ref{lem:empirical}, and the
choice of $\nu=\frac1k\sum_{i=1}^k\mu_i$, that for sufficiently large
$b$ (equivalently, sufficiently large $a$)
$$
d^*(\lambda,\mu)\leq d^*(\lambda,\SE_a(x))+d^*\left(\SE_a(x),\lambda_{C^\infty(r)}^b\right)+d^*\left(\lambda_{C^\infty(r)}^b,\frac{1}{b}\sum_{p=0}^{b-1}\nu\right)+d^*(\nu,\mu)<\kappa.
$$

Summing up, $d^*(\mu_c,\mu)<\kappa$ and $h_{\mu_c}(f)>c$ as claimed.
\vspace{0.1in}

\vspace{0.1in}
{\bf Step 5:} Suppose now that $f$ is entropy expansive with constant
$e$. Let $\pi\colon\hat K\to\Sigma_s$ be the factor map from
$f^m|_{\hat K}$ onto $\sigma$.  We have
$h(f^m|_{\hat K})\geq h(\sigma)$, while
\cite[Theorem~17]{Bo} gives
$$h(f^m|_{\hat K})\leq h(\sigma) + \sup_{r\in \Sigma_s}\{h(f^m, \pi^{-1}(r))\}.$$

Every point of $\pi^{-1}(r)$ $\beta$-shadows the pseudo-orbit coded by
$r$. Hence, for $x,y\in\pi^{-1}(r)$,
\[
d(f^j(x),f^j(y))\leq2\beta<e
\quad\text{for every }j\in\Z.
\]
Thus $\pi^{-1}(r)\subset B^\infty_{2\beta}(x)$ for every
$x\in\pi^{-1}(r)$, and entropy expansiveness gives
$h(f^m,\pi^{-1}(r))=0$. Consequently,
$h(f^m|_{\hat K})=h(\sigma)$.
 \vspace{0.1in}   

{\bf Step 6:}
Since $K=\bigcup_{j=0}^{m-1}f^j(\hat K)$, Step~5 and the entropy
estimate in Step~3 give
\[
h(\sigma)=h(f^m|_{\hat K})=m h(f|_K)>mc.
\]

By \cite[Theorem~2.5]{Gr}, there is a compact invariant set
$\Lambda\subset\Sigma_s$ such that $\sigma|_\Lambda$ is minimal and
uniquely ergodic and
$h(\sigma|_\Lambda)=mc$. Put
$\hat K_c=\pi^{-1}(\Lambda)$, and let $\nu_c$ be the unique invariant
measure supported on $\Lambda$.
Observe that $\hat{K}_c$ is compact and $f^m$-invariant and factors over $\sigma|_{\Lambda}$. Moreover, if $\eta$ is an ergodic measure for $f^m|_{\hat{K}_c}$, we have that $\eta$ factors over the unique ergodic measure of $\Lambda$, that is $\pi_*\eta=\nu_c$. In particular, $h_{\eta}(f^m)\geq h_{\nu_c}(\sigma|_{\Lambda})=mc$. Since $f$ is entropy-expansive, arguing with Bowen's inequality exactly as in Step 5 we get $h(f^m|_{\hat{K}_c})\leq h(\sigma|_{\Lambda})=mc$, and therefore $$h(f^m|_{\hat{K}_c})=h_{\eta}(f^m)=mc.$$
Denote $K_c=
\bigcup_{j=0}^{m-1}f^j(\hat{K}_c)$ and let $\mu_c$ be any ergodic measure supported on $K_c$. By Jacobs ergodic decomposition theorem for entropy (e.g. see \cite[Theorem~8.4]{Walters}) we have that
$$
h_{\mu_c}(f^m)=\int_{\SM_{f^m}^e(K_c)} h_\eta(f^m)\;\text{d}\tau(\eta).
$$
But every $\eta\in \SM_{f^m}^e(K_c)$ satisfies $\eta(f^i(\hat{K}_c))=1$ for some $i$, and hence it is isomorphic (by $f^i$) to some $\eta\in \SM_{f^m}^e(\hat{K}_c)$. But we already know that in such a case $h_\eta(f^m)=mc$, so
$$
h_{\mu_c}(f^m)=\int_{\SM_{f^m}^e(K_c)} mc\; \text{d}\tau(\eta)=mc
$$
and therefore $h_{\mu_c}(f)=c$.  Since $\su(\mu_c)\subset K$, the
estimate proved above gives $d^*(\mu_c,\mu)<\kappa$, completing the proof.

\end{proof}

Theorem~\ref{mainC} also shows that the local theory contains the
classical global setting as a special case.  Its first immediate
consequence is entropy flexibility for entropy-expansive homeomorphisms
with the shadowing property.

\begin{corollary}
An entropy-expansive homeomorphism with shadowing has entropy flexibility:
for every $c\in[0,h(f))$, there are an ergodic
measure $\mu_c$ and a compact invariant set $K_c\subset M$ such that
\[
        h(f|_{K_c})=h_{\mu_c}(f)=c.
\]
In particular, every countably expansive homeomorphism with the shadowing
property has entropy flexibility.
\end{corollary}

\begin{proof}
Let $f$ be an entropy-expansive homeomorphism with the shadowing property,
and fix $c\in[0,h(f))$.  By the variational principle and the ergodic
decomposition theorem, there is an ergodic invariant measure $\mu$ such
that $h_\mu(f)>c$.  The support of an ergodic measure is chain transitive,
so, for every $x\in\operatorname{supp}(\mu)$,
\[
        \operatorname{supp}(\mu)\subset H(x).
\]
Since $f$ has the shadowing property, every point belongs to $\Sh^+(f)$.
Theorem~\ref{mainC}, applied to $x$ and $\mu$, therefore gives an
ergodic measure $\mu_c$ such that
\[
        h_{\mu_c}(f)
        =h\bigl(f|_{\operatorname{supp}(\mu_c)}\bigr)=c.
\]
Taking $K_c=\operatorname{supp}(\mu_c)$ proves the first assertion.

For the final assertion, suppose that \(f\) is countably expansive and has the
shadowing property. Set
$Y=CR(f)$ and
$g=f|_Y$.
Since \(f\) has the shadowing property, \(CR(f)=\Omega(f)\), and hence
$
h(g)=h(f|_{\Omega(f)})=h(f).
$
Moreover, \(g\) is countably expansive and has the shadowing property. Since the
chain recurrent set has the restriction property,
$\Omega(g)=CR(g)=CR(f|_{CR(f)})=CR(f)=Y.$
It follows from \cite[Theorem~2.5]{ACCV} that \(g\) is entropy expansive.
Therefore, the first part of the corollary, applied to \(g\), gives, for every
\(c\in[0,h(g))=[0,h(f))\), an ergodic \(g\)-invariant measure \(\mu_c\) and a
compact \(g\)-invariant set \(K_c\subset Y\) such that
$
h(g|_{K_c})=h_{\mu_c}(g)=c.
$
Since \(g=f|_Y\), the measure \(\mu_c\) and the set \(K_c\) are also
\(f\)-invariant and
$
h(f|_{K_c})=h_{\mu_c}(f)=c.
$
Thus \(f\) has entropy flexibility.
\end{proof}

\subsection{Expansivity and further entropy consequences}\label{pointsentr}

This subsection applies the coding mechanism rather than introducing a new
source of complexity. Pointwise or measurable expansivity supplies the
separation and non-equicontinuity; local tracing then converts it into a
semi-horseshoe. Theorems~\ref{entrcexp} and~\ref{entrSM} are the resulting
pointwise and measure-theoretic statements.
They also answer the question proposed in \cite{AR} while strengthening
its hypotheses.  In \cite{AR}, positive entropy was obtained by combining
classical shadowable points with uniformly expansive or uniformly
$N$-expansive points.  In our approach we will use weaker version of pointwise countable
expansivity. Let us first recall this concept.

\begin{definition}
Let $f\colon M\to M$ be a homeomorphism. We say that $x\in M$ is
\emph{countably expansive} if there is $e>0$ such that
$B^{\infty}_e(x)$ is countable.
\end{definition}
%
%If $x$ is recurrent and non-periodic, then $\overline{O(x)}$ has no
%isolated points because an isolated point in this transitive orbit closure would
%belong to the orbit of $x$ and would force that orbit to be periodic.
%Thus every nonempty relatively open subset of $\overline{O(x)}$ contains
%a nonempty perfect subset and is uncountable.

\begin{theorem}\label{expsens}
Let $f\colon M\to M$ be a homeomorphism and let $x\in M$ be
non-periodic, recurrent, and countably expansive. Then there is $C>0$
such that every nonempty relatively open set
$U\subset\overline{O(x)}$ satisfies
\[
\diam(f^n(U))>C
\]
for some $n\in\mathbb Z$.
\end{theorem}
\begin{proof}
Assume $x$ satisfies the hypotheses. Consequently,
every nonempty open subset of $\overline{O(x)}$ is uncountable. Fix
$z\in\overline{O(x)}$ and let $U$ be a neighborhood of $z$ in
$\overline{O(x)}$. Let $e>0$ be such that $B^{\infty}_e(x)$ is at most
countable. Since $f$ is a homeomorphism,
\[
f^n(B^{\infty}_e(x))=B^{\infty}_e(f^n(x))
\quad\text{for every }n\in\mathbb Z,
\]
so $B^{\infty}_e(f^n(x))$ is at most countable. Choose
$n\in\mathbb Z$ such that $f^n(x)\in U$ and put
$U'=U\cap B_e(f^n(x))$. The set $U'$ is a nonempty relatively open
subset of $\overline{O(x)}$, hence is uncountable, whereas
$B^{\infty}_e(f^n(x))$ is countable. Thus there are $y\in U'$ and
$m\in\mathbb Z$ with
$d(f^m(y),f^{m+n}(x))>e$. Since both $y$ and $f^n(x)$ belong to
$U'\subset U$, we obtain $\diam(f^m(U))>e$. The assertion follows with
$C=e$.
\end{proof}

The preceding results give the following pointwise criterion, related to
the question raised in \cite{AR}.

\begin{theorem}\label{entrcexp}
Let $f\colon M\to M$ be a homeomorphism. If $f$ has a non-periodic,
positively shadowable, chain-recurrent, and countably expansive point,
then $f$ has a semi-horseshoe. In particular, $f$ has positive
topological entropy.
 \end{theorem}

\begin{proof}
    
    Suppose $x$ satisfies all the required assumptions. Since $x$ is non-periodic, $H(x)$ is not a periodic orbit. If $H(x)$ is not minimal, the existence of a semi-horseshoe is granted by 
    Theorem \ref{nonminimal}.

    If $H(x)$ is minimal, to apply Theorem~\ref{senshors} we need only show
that $H(x)$ is not equicontinuous. If it were equicontinuous, some
nonempty relatively open set $U\subset H(x)$ would satisfy
$\diam(f^n(U))<C$ for all $n\in\mathbb Z$, where $C$ is the constant from
Theorem~\ref{expsens}. This contradicts that theorem.
\end{proof}

Within the setting of homeomorphisms, this criterion uses the following
weaker pointwise assumptions than the corresponding results of
\cite{AR}:
  \begin{itemize}
     \item a chain-recurrent point in place of a non-wandering point;
     \item positive shadowability in place of classical shadowability,
           while still covering every point of $\Sh(f)$;
     \item countable expansivity instead of either uniform expansivity
           or the uniformly \(N\)-expansive conditions.
  \end{itemize}
Theorem~\ref{entrcexp} is formulated for homeomorphisms, whereas
\cite{AR} also treats continuous maps. Corresponding one-sided versions
for noninvertible maps follow from the same positive-time tracing
argument, as discussed below.

Next, we study the consequences of combining measurable expansivity with
measure shadowing. This combination also yields entropy consequences.

\begin{definition}\label{def:expansive_meas}
    A Borel  probability measure $\mu$ is said to be expansive, if there is $e>0$ such that $\mu(B^{\infty}_e(x))=0$, for every $x\in M$.
\end{definition}

\begin{remark}
Every expansive measure is non-atomic, because
$\{x\}\subset B^\infty_e(x)$ for every $x\in M$. Let us stress the fact that the invariance of
measure is not assumed in Definition~\ref{def:expansive_meas}.
\end{remark}

\begin{theorem}\label{entrSM}
     Let $f:M\to M$ be a homeomorphism. If $f$ admits an expansive and positively shadowable measure $\mu$ giving positive measure to a chain-recurrent class, then $f$ has a semi-horseshoe. In particular, $f$ has positive topological entropy.
\end{theorem}

\begin{proof}
Let $\mu$ be an expansive positively shadowable measure; it is
non-atomic by the preceding remark. Suppose $\mu(H(x))>0$ for a
chain-recurrent class and define
\[
\nu(A)=\frac{\mu(A\cap H(x))}{\mu(H(x))}
\]
for measurable $A$. Then $\nu$ remains expansive, positively shadowable,
and non-atomic, while $\su(\nu)\subset H(x)$. In particular, $H(x)$ is
not a periodic orbit.

Next, take $y\in \su(\nu)\cap H(x)$ and let $e>0$ be the expansiveness constant of $\nu$. Since $y\in \su(\nu)$, we have $\nu(B_{\delta}(y)\cap H(x))>0$, for any $0<\delta\leq e$, in particular $B_{\delta}(y)\cap H(x)\not \subset B^{\infty}_e(y)\cap H(x)$. This implies the existence of $z,w\in B_{\delta}(y)\cap H(x)$ and $n\in \mathbb{Z}$ such that $$d(f^n(z),f^n(w))>e.$$
 Finally, Theorem \ref{unifsupp} gives $\su(\nu)\subset Sh^+(f)$. Choose $y\in\su(\nu)$. Then $y\in Sh^+(f)\cap CR(f)$ and $H(y)=H(x)$, so Theorem \ref{ShadClass} yields $H(x)\subset Sh^+(f)$. The conclusion now follows exactly as in Theorem \ref{entrcexp}, according to whether $H(x)$ is minimal or nonminimal.

\end{proof}

Combining the preceding results gives sufficient conditions for a
countably expansive homeomorphism to have a semi-horseshoe.

\begin{corollary}

Let $f$ be a countably expansive homeomorphism. Then $f$ has a semi-horseshoe if one of the following conditions is satisfied:
\begin{enumerate}
    \item There is $x\in Sh^+(f)\cap CR(f)$ such that  the chain-recurrent class of $x$ is non-periodic.
    \item There is a non-atomic positively shadowable measure giving positive measure to a chain-recurrent  class $H(x)$.
    \item $f$ has the shadowing property and possesses a non-periodic chain-recurrent point. 
\end{enumerate}
\end{corollary}

\begin{proof}
Item~(1) follows from Theorems~\ref{nonminimal} and~\ref{entrcexp}.
Item~(3) reduces to item~(1), because global shadowing gives
$\Sh(f)=M\subset Sh^+(f)$ and a non-periodic chain-recurrent point belongs
to a non-periodic chain class. For item~(2), \cite[Theorem~2.1]{AC}
implies that every non-atomic probability measure of a countably
expansive homeomorphism is expansive. The conclusion then follows from
Theorem~\ref{entrSM}.
\end{proof}

\subsection{Final remarks}

The preceding theory now gives intrinsic dynamical consequences on the
components constructed in Part~I, thereby connecting the realization and
consequence parts of the paper.

\begin{corollary}\label{cor:dense-consequences}
Let $M$ be a compact manifold of dimension at least three.  There is a
$C^0$-dense family of homeomorphisms $T\in\Homeo(M)$ satisfying
Theorem~\ref{thm:shadowable-dense} for which the distinguished component
$D$ has the following additional properties:
\begin{enumerate}
\item $T$ has a semi-horseshoe associated with $D$;
\item for every invariant probability measure $\mu$ supported on $D$ and
      every $\eps>0$, there is an ergodic measure $\nu$ with
      $h_\nu(T)>0$ and $d^*(\nu,\mu)<\eps$.
\end{enumerate}
\end{corollary}

\begin{proof}
In dimensions at least three, the construction in
Theorem~\ref{thm:shadowable-dense} gives
$0<h_{\rm top}(T|_D)<\infty$.  Since $T|_D$ is transitive and has positive
entropy, $D$ is not a finite periodic orbit.  Choose a transitive point
$x\in D$.  It is recurrent, $H(x)=D$, and
$x\in Sh(T)\subset Sh^+(T)$.

The restriction $T|_D$ is not equicontinuous, because an equicontinuous
homeomorphism has zero topological entropy.  If $D$ is not minimal,
Theorem~\ref{nonminimal} gives a semi-horseshoe; if $D$ is minimal,
Theorem~\ref{senshors} does so.  This proves item~(1).
Theorem~\ref{mainB}, applied to $x$, proves item~(2).
\end{proof}

\begin{remark}
In the two-dimensional model, the distinguished component is conjugate
to an irrational rotation and is therefore minimal, equicontinuous, and
of zero entropy. Together with Corollary~\ref{cor:dense-consequences},
this shows within the dense manifold realization itself that local
shadowing does not force complexity: the additional hypotheses in
Part~II distinguish precisely when symbolic and entropy phenomena follow.
\end{remark}

We conclude by indicating the corresponding forward-time formulation for
a continuous, not necessarily invertible map $f\colon M\to M$. Since a
point then has no distinguished full past, positive pointwise
shadowability is the natural hypothesis. The propagation arguments and
the consequences in Part~II that use only positive pseudo-orbits remain
valid with this formulation.  In the semi-horseshoe lemmas of
Section~\ref{SectionLocEntr}, the resulting subsystem factors onto a
one-sided full shift instead of a two-sided full shift.Thus the
forward-time results of Part~II admit continuous-map analogues; we do not
formulate them separately here.

\section{Open problems}\label{sec:open}

We close the paper with several questions which, in our opinion, delimit
the natural next steps of the theory developed here.

\begin{problem}[Two-sided propagation]\label{q:propagation}
Theorem~\ref{ShadClass} shows that $x\in Sh^+(f)\cap CR(f)$ implies
$H(x)\subset Sh^+(f)$, and the proof visibly breaks down for the
two-sided set: a two-sided pseudo-orbit through $y\in H(x)$ cannot simply
be prefixed by a chain from $x$. Suppose $x\in\Sh(f)\cap CR(f)$. Is
$H(x)\subset\Sh(f)$? A counterexample would further separate the
two pointwise notions on the level of chain classes.
\end{problem}

\begin{problem}[Positive entropy in dimension two]\label{q:dim2}
In our two-dimensional model the distinguished component is conjugate to
an irrational rotation, hence has zero entropy. Can the component $D$ in
Theorem~\ref{mainA} be chosen with positive topological entropy when
$\dim M=2$? Equivalently, is there a homeomorphism of the disc with a
transitive, positive-entropy chain component $D\subseteq\Sh(T)$ such that
no chain class near $D$ (including $D$ itself) has the shadowing property?
\end{problem}

\begin{problem}[Smooth realization]\label{q:smooth}
The constructions of Part~I are purely topological, as is the underlying
Denjoy--Rees technique. Do there exist $C^r$ diffeomorphisms, $r\ge1$,
with the properties listed in Theorem~\ref{mainA}? Already the smooth
realization of a single chain component $D\subseteq\Sh(T)$ without
intrinsic shadowing appears to require new tools.
\end{problem}

\begin{problem}[Entropy flexibility beyond entropy expansivity]\label{q:flex}
Does the exact-entropy conclusion of Theorem~\ref{mainC} hold for
asymptotically entropy-expansive maps, or, more generally, whenever
$(M,f)$ admits a principal symbolic extension in the sense of
\cite{BD}? The proof of Step~5 uses entropy expansivity only through the
vanishing of the tail entropy along fibres of the coding map, which
suggests that symbolic-extension techniques could replace it.
\end{problem}

\begin{problem}[Prescribing $\Sh(f)$ on manifolds]\label{q:realization}
Theorem~\ref{extthm} prescribes the full set of shadowable points on a
suitable compact metric space. Which pairs (a compactum $K\subset M$
together with dynamics on $K$) arise as $\Sh(f)$ for a homeomorphism $f$
of a compact manifold $M$? In particular, can $\Sh(f)$ be a prescribed
Cantor system when $f\in\Homeo(M)$ and $\dim M\ge2$?
\end{problem}

\section*{Acknowledgements}

This research was partly supported by EU funds through the Operational
Programme Johannes Amos Comenius, under project
CZ.02.01.01\slash 00\slash 23\_021\slash 0008759.

\begin{table}[h]
\begin{tabularx}{\linewidth}{p{1.5cm}  X}
\includegraphics [width=1.4cm]{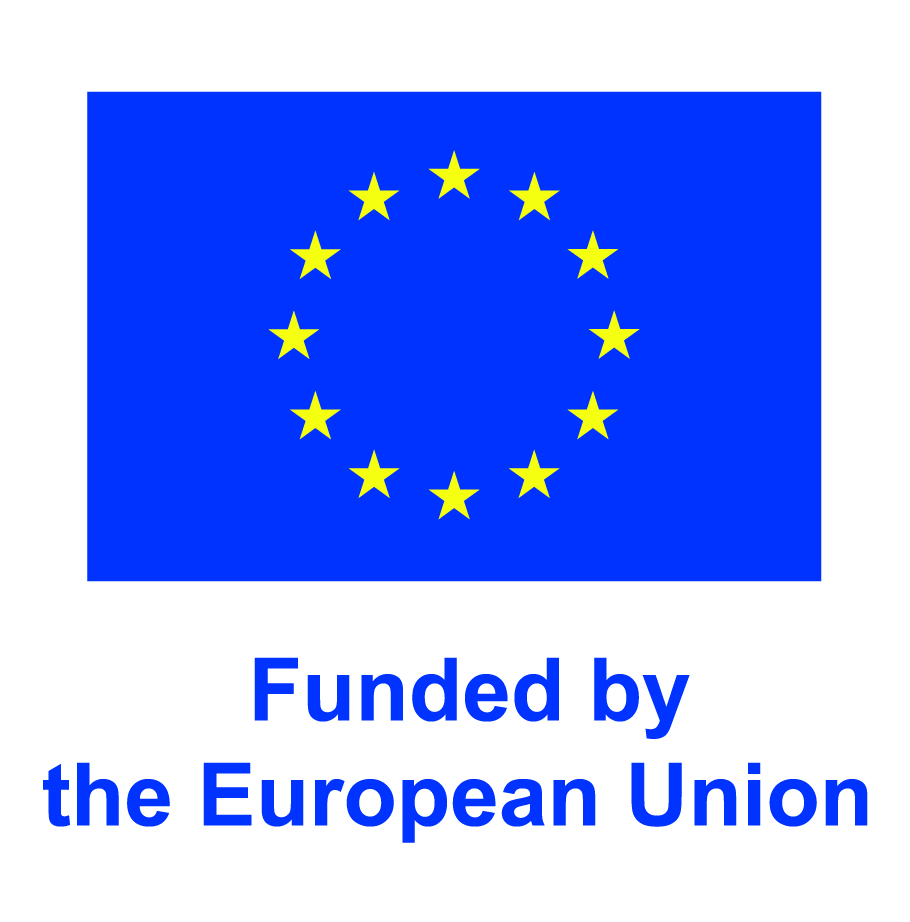} &
\vspace{-1.3cm}
This research is part of a project that has received funding from
the European Union's European Research Council Marie Sklodowska-Curie Project No. 101151716 -- TMSHADS -- HORIZON--MSCA--2023--PF--01.For the purpose of open access, and in fulfilment of the obligations arising from the grant agreement, the author has applied a Creative Commons Attribution 4.0 International (CC BY 4.0) license to any Author Accepted Manuscript version arising from
this submission.\\
\end{tabularx}
\end{table}

% \section*{Statements}

% -No data was used for the research described in the article.

% -The authors declare that they have no conflict of interest.

\end{document}